\documentclass[a4paper,10pt,reqno]{amsart}
\usepackage{amsmath}%
\usepackage{amsthm}
\usepackage{thmtools}
\usepackage{amsfonts}%
\usepackage{amscd}
\usepackage{amssymb}%
\usepackage{graphicx}
\usepackage{enumerate}
\usepackage[all]{xy}
\usepackage{varioref}
\usepackage{hyperref}
\usepackage[capitalise]{cleveref}
\usepackage{mathtools}
\usepackage{ stmaryrd }
\usepackage{xcolor}
\usepackage[pageref]{backref}
\usepackage{dsfont}
\usepackage{url}

\usepackage{xfrac}          
\usepackage{nicefrac}       
\usepackage[normalem]{ulem}       
\usepackage{tikz-cd}
\usepackage{centernot}

\newtheorem{teo}{Theorem}[section]
\theoremstyle{plain}
\newtheorem{prop}[teo]{Proposition}
\newtheorem{lema}[teo]{Lemma}
\newtheorem{coro}[teo]{Corollary}

\theoremstyle{definition}
\newtheorem{defi}[teo]{Definition}

\newtheorem{exem}[teo]{Example}
\newtheorem{exems}[teo]{Examples}

\theoremstyle{remark}

\newtheorem{obs}[teo]{Remark}

\numberwithin{equation}{section}

\newcommand{\im}{\mathrm{im}\,}
\newcommand{\id}{\mathrm{id}}

\newcommand{\h}{\mathrm{h}}
\DeclareMathOperator{\U}{U}
\DeclareMathOperator{\Ugr}{U^{gr}}
\DeclareMathOperator{\isogr}{\cong_{gr}}

\newcommand{\supp}{\mathrm{supp}}

\newcommand{\rann}{\mathrm{r.ann}}

\DeclareMathOperator{\socgr}{soc_{gr}}
\DeclareMathOperator{\soc}{soc}
\newcommand{\radgr}{\mathrm{rad}^{\mathrm{gr}}}

\newcommand{\rad}{\mathrm{rad}}

\DeclareMathOperator{\subgr}{\leq_{gr}}

\newcommand{\proj}{\mathrm{proj}}

\DeclareMathOperator{\Fun}{Fun}
\DeclareMathOperator{\Homgr}{Hom_{gr-\mathit{R}}}
\DeclareMathOperator{\Endgr}{End_{gr-\mathit{R}}}
\newcommand{\Hom}{\mathrm{Hom}}
\newcommand{\HOM}{\mathrm{HOM}}
\newcommand{\End}{\mathrm{End}}
\newcommand{\END}{\mathrm{END}}

\newcommand{\fgmod}{\mathrm{mod}}

\newcommand{\M}{\mathrm{M}}
\newcommand{\UT}{\mathrm{UT}}

\DeclareMathOperator{\ima}{im}

\newcommand{\maxi}{\mathfrak{m}}

\begin{document}

	\title[Groupoid graded chain conditions and Jacobson radical]{Graded chain conditions and graded Jacobson radical of groupoid graded modules}
	\author[Z. Cristiano]{Zaqueu Cristiano}
	\address[Zaqueu Cristiano]{Department of Mathematics - IME, University of S\~ao Paulo,
		Rua do Mat\~ao 1010, S\~ao Paulo, SP, 05508-090, Brazil}
	\email[Zaqueu Cristiano]{zaqueucristiano@usp.br}%
	
	\author[W. Marques de Souza]{Wellington Marques de Souza}
	\address[Wellington Marques de Souza]{Department of Mathematics - IME, University of S\~ao Paulo,
		Rua do Mat\~ao 1010, S\~ao Paulo, SP, 05508-090, Brazil}
	\email[Wellington Marques de Souza]{tomarquesouza@usp.br}%
	
	\author[J. S\'anchez]{Javier S\'anchez}
	\address[Javier S\'anchez]{Department of Mathematics - IME, University of S\~ao Paulo,
		Rua do Mat\~ao 1010, S\~ao Paulo, SP, 05508-090, Brazil}
	\email[Javier S\'anchez]{jsanchez@ime.usp.br}%
    \thanks{This study was financed in part by the Coordenação de Aperfeiçoamento de Pessoal de Nível Superior - Brasil (CAPES) – Finance Code  001}
	\thanks{The first author was supported in part by grant \texttt{\#}2021/14132-2, S\~ao Paulo Research Foundation (FAPESP), Brazil}
	\thanks{The third author was supported by grant \texttt{\#}2020/16594-0, S\~ao Paulo Research Foundation (FAPESP), Brazil}
    
	\date{December 12, 2025}
	\subjclass[2020]{Primary 16W50; 16P20; 16P40; 16N20; 16L30; 20L05. Secondary 18E05; 18A25; 16D60; 16D50; 16G60.} %
	\keywords{Groupoid graded module; gr-artinian module; gr-noetherian module; gr-length of a module; graded Jacobson radical; gr-socle; gr-semilocal ring; artinian category; noetherian category; semilocal category.
}
	\dedicatory{To César Polcino Milies, on the occasion of his 80th birthday}

\begin{abstract}
    In this work, we  continue to lay the groundwork for the theory of groupoid graded rings and modules.
    The main topics we address include graded chain conditions, the graded Jacobson radical, and the gr-socle for graded modules. We present several descending (ascending) chain conditions for graded modules and we refer to the most general one as $\Gamma_0$-artinian ($\Gamma_0$-noetherian).    
    We show that $\Gamma_0$-artinian (resp. $\Gamma_0$-noetherian) modules share many properties with artinian (noetherian) modules in the classical theory.
    However, we present an example of a right $\Gamma_0$-artinian ring that is not right $\Gamma_0$-noetherian. Following the pattern of the classical case, we examine the basic properties of the graded Jacobson radical and the gr-socle for groupoid graded modules.
    We also establish some fundamental properties of the graded Jacobson radical of groupoid graded rings.
    Finally, we introduce the notion of  gr-semilocal ring, which simultaneously generalizes the concepts of  semilocal ring and  (small) semilocal  category.
\end{abstract}

	\maketitle

 \tableofcontents

\section{Introduction}

Chain conditions and the Jacobson radical are among the main topics in classical ring and module theory. 
There are many important results concerning artinian and noetherian rings and modules, and semilocal rings (rings whose quotient by the Jacobson radical is semisimple) \cite{AndersonFuller}, \cite{Lam1}, \cite{Lam2}, \cite{Goodearl}, \cite{Facchini}. 
Some of these concepts and results have been generalized to the group graded context \cite{NastasescuVanOystaeyen}, \cite{Hazrat}.

A groupoid, a small category in which all morphisms are invertible, is a natural generalization of a group. To the best of the authors' knowledge, the first work that began the study of groupoid graded rings and modules as a separate class was \cite{Lund}. Usually, first results in the theory 
were obtained supposing that either the graded ring has an identity element and/or that the groupoid has a finite number of objects \cite{Liu_Li}, \cite{Lundstrom_strongly_cohomology}, \cite{Verhulst}. A systematic study of groupoid graded rings and modules that do not satisfy these conditions began in \cite{CLP} where the first results about groupoid graded semisimple modules were also given. We continued this line of research by establishing the basic theory of gr-semisimple groupoid graded rings in \cite{Artigoarxiv}. 
In these works,
graded rings are supposed to have a good set of local units. Such rings are referred to as object unital, a term that we adopt in this work. This concept was already used in \cite{Oinert_Lundstrom_Miyashita} under the name of locally unital ring. 
This generality  prevents some results or concepts from having natural generalizations.
On the other hand, this framework allows us to obtain interesting applications to small preadditive categories \cite{Artigoarxiv}.

Groupoid graded rings and modules have been studied in recent years.  Some papers laid the ground of the theory \cite{Lund,CLP,Artigoarxiv}, or generalized  previous classical results \cite{Oinert_Lundstrom_Miyashita}, \cite{CLP2} \cite{Lundstrom_strongly_cohomology}, and some others have  appeared in different contexts such as partial actions \cite{BFP}, \cite{BP}, \cite{MoreiraOinert} and Leavitt path algebras \cite{Goncalves_Yoneda}.

In this work, we continue to establish the foundation of the theory of groupoid graded rings and modules, with special focus on graded chain conditions and the graded Jacobson radical of graded modules. As far as the authors are aware, apart from \cite{georgijevic} (where groupoids are a more general concept than ours but finite)  and \cite{lundstrom2024chain} (where the chain conditions 
 do not refer to graded chain conditions), there are no works dealing with the main objects of our study.

Given a groupoid $\Gamma$, we consider standard chain conditions, $\Gamma_0$-chain conditions, and strong $\Gamma_0$-chain conditions, all of which generalize the graded chain conditions from the group graded setting.
Standard chain conditions arise as the most natural definition, but graded rings satisfying them are necessarily unital. 
This motivates the introduction of $\Gamma_0$-chain conditions, which yield results similar to those in the classical context \cite{Artigoarxiv} and are connected to artinian/noetherian categories \cite{Mit}. 
We present some examples of right $\Gamma_0$-artinian (resp. $\Gamma_0$-noetherian) rings that arise naturally from categories of modules.
In this context, we prove that a skeleton of the category of finitely generated modules over a representation-finite algebra is both artinian and noetherian.
We prove a graded version of the Bass–Papp Theorem, characterizing right $\Gamma_0$-noetherian rings in terms of gr-injective modules.
We also characterize a gr-semisimple ring as a right $\Gamma_0$-artinian ring with graded Jacobson radical zero. But the theory does not extend verbatim.
In contrast to the group graded case, we show that a right $\Gamma_0$-artinian ring does not need to be right $\Gamma_0$-noetherian.

The theory of the Jacobson radical of modules can be found in \cite{AndersonFuller}. 
This notion can be defined either via maximal submodules or via superfluous submodules. 
As shown in \cite{AndersonFuller}, the socle of a module (defined using either simple submodules or essential submodules) exhibits many properties that are dual to those of the Jacobson radical.
In this work, we define and prove groupoid graded versions of these properties.
To the knowledge of the authors, not all group graded versions of these results have  been established.

We also study some basic properties of the graded Jacobson radical of groupoid graded rings. 
Further and deeper results, especially those related to the nilpotency of the graded Jacobson radical, will appear in another paper currently in preparation. It is when approaching this problem that the strong $\Gamma_0$-chain conditions appeared. 
In the present work, we are particularly interested in the case where the quotient of the ring by its graded Jacobson radical is a gr-semisimple ring (or, equivalently, a right $\Gamma_0$-artinian ring).
We refer to such rings as gr-semilocal rings.
In one of our main results concerning these rings, we show that a $\Gamma$-graded ring is gr-semilocal if and only if the ring $R_e$ is semilocal for every idempotent $e$ in the groupoid $\Gamma$. 
In particular, we conclude that groupoid graded rings arising from semilocal categories are gr-semilocal, thus providing many examples of such rings.  We point out that (ungraded) semilocal rings form a very important class in Ring and Module Theory, see for example \cite{Camps_Dicks}, \cite{Facchini_Herbera}, \cite{HerberaShamsuddin1995}, \cite{Facchini}.

We note that, when the theory carries over to the groupoid graded context, the proofs typically follow the same lines as in the ungraded case, though they tend to be more technical.

\medskip

In the preliminaries section, we present the basic definitions of groupoids, graded rings, and graded modules that will be used throughout the paper.

In the \Cref{sec: chain conditions}, we study graded chain conditions. 
We define and establish properties of standard graded chain conditions, $\Gamma_0$-chain conditions, and strong $\Gamma_0$-chain conditions. 
In the last two parts of this section, we  present an example of a right $\Gamma_0$-artinian ring that is not right $\Gamma_0$-noetherian, and we prove a graded version of the Bass–Papp Theorem.

\Cref{sec: gr-Jacobson gr-socle} is dedicated to the study of the graded Jacobson radical and the gr-socle of graded modules.

In \Cref{sec: rad(R)}, we focus on the graded Jacobson radical of graded rings. 
We prove some basic results, including a graded version of Nakayama's Lemma and a characterization of gr-semisimple rings in terms of graded chain conditions and the graded Jacobson radical.

Finally, in \Cref{sec: gr-semilocal rings}, we study gr-semilocal rings.

\section{Preliminaries}
\label{sec: preliminares}

In this section, we present some basic definitions and results that will be used throughout the paper.

\begin{defi}
A \emph{groupoid} $\Gamma$ is a small category in which every morphism is an isomorphism. 
We will denote the set of \emph{objects} of $\Gamma$ by $\Gamma_0$. 
The identity morphism of $e\in\Gamma_0$ will be denoted again by $e$. 
We will identify the groupoid with its set of morphisms and sometimes we will refer to $\Gamma_0$ as the set of idempotents of $\Gamma$. 
If $e,f\in\Gamma_0$ and $\gamma$ is a morphism from $e$ to $f$, we will write $d(\gamma)=e$, $r(\gamma)=f$, and $\gamma\in f\Gamma e$. 
Notice that $d(\gamma^{-1})=r(\gamma)=f$ and $r(\gamma^{-1})=d(\gamma)=e$.
For $e\in\Gamma_0$, we denote by $e\Gamma$ (resp. $\Gamma e$) the set of morphisms $\gamma\in\Gamma$ with $r(\gamma)=e$ (resp. $d(\gamma)=e$).
If $d(\delta)\neq r(\gamma)$, we will say that $\delta\gamma$ \emph{is not defined} in $\Gamma$.  
\end{defi}

\begin{exems}
\label{ex:groupoids}
    \begin{enumerate}[(1)]

      \item Any group can be regarded as a groupoid with only one object.

        \item Let $X$ be a nonempty set and let $\Gamma:=X\times X$.
        Then, $\Gamma$ is a groupoid where, for each $(y,x),(z,w)\in \Gamma$, the composition  $(z,w)(y,x)$ is defined if and only if $w=y$, and in this event, $(z,y)(y,x)=(z,x)$.
        Note that $\Gamma_0=\{(x,x):x\in X\}$.

        \item Generalizing the previous example, let $X$ be a nonempty set, $G$ be a group and let $\Gamma:=X\times G\times X$. 
        For each $(z,h,w),(y,g,x)\in \Gamma$, the composition  $(z,h,w)(y,g,x)$ is defined if and only if $w=y$, and in this event, 
        $(z,h,y)(y,g,x)=(z,hg,x)$.
        This endow $\Gamma$ with a structure of groupoid with $\Gamma_0=\{(x,e_G,x):x\in X\}$
         Notice  that we obtain the foregoing example when  $G$ is the trivial group.

    \end{enumerate}   
\end{exems}

For a detailed study of groupoids the interested reader is referred to \cite{Higgins} or \cite{IR}. 

\medskip

\emph{Throughout the text, let $\Gamma$ be a groupoid.}

\medskip

\begin{defi}
An additive group $X$ is \emph{$\Gamma$-graded} if there exists a family $\{X_\gamma : \gamma \in \Gamma\}$ of subgroups of $X$ such that $X = \bigoplus_{\gamma \in \Gamma} X_\gamma$.
\end{defi}
Graded rings and graded modules (which we define below) are examples of graded additive groups.
We now introduce some notation and terminology common to both graded rings and graded modules.
\begin{defi}
Let $X=\bigoplus_{\gamma\in\Gamma}X_\gamma$ be a $\Gamma$-graded additive group. 
For each $\gamma\in\Gamma$, $X_\gamma$ is called the \emph{homogeneous component of degree $\gamma$ of $X$}, its nonzero elements are said to be \emph{homogeneous of degree $\gamma$} and we write $\deg(x)=\gamma\in\Gamma$ when $0\neq x\in X_\gamma$. 
For each $\gamma\in\Gamma$ and $x\in X$, we write $x=\sum_{\gamma\in\Gamma}x_\gamma$ with $x_\gamma\in X_\gamma$ (with all but finitely many $x_\gamma$ nonzero) and we call $x_\gamma$ the \emph{homogeneous component of degree $\gamma$ of $x$}. 
The set of the \emph{homogeneous elements} of $X$ is $\h(X):=\bigcup_{\gamma\in\Gamma}X_\gamma$. 
The \emph{support} of $X$ is $\supp(X):=\{\gamma\in\Gamma:X_\gamma\neq0\}$ and the \emph{support} of $x\in X$ is $\supp(x):=\{\gamma\in\Gamma:x_\gamma\neq0\}$. 
When $\sigma,\tau\in\Gamma$ are such that $d(\sigma)\neq r(\tau)$, we will adopt the convention $X_{\sigma\tau}:=\{0\}$. 
If $Y$ is a $\Gamma$-graded additive group, then a homomorphism of groups $g:X\to Y$ is said to be a \emph{gr-homomorphism of groups} if $g(X_\sigma)\subseteq Y_\sigma$ for all $\sigma\in\Gamma$. 
If, moreover, $g$ is bijective, then we say that $g$ is a \emph{gr-isomorphism of groups}, that $X$ is \emph{gr-isomorphic} to $Y$ as groups, and we will denote it by $X\isogr Y$.
\end{defi}

\subsection{Groupoid graded rings}\label{sec:grupoid graded rings}
\emph{Throughout this work, rings are assumed to be associative but not necessarily unital.}
\begin{defi}
Let $R$ be a ring.
We say that $R$ is a \emph{$\Gamma$-graded ring} if there is a family $\{R_\gamma\}_{\gamma\in\Gamma}$ of additive subgroups of $R$ such that 
$R=\bigoplus_{\gamma\in\Gamma}R_\gamma$ and
$R_\gamma R_\delta\subseteq  R_{\gamma\delta}$, for each $\gamma,\delta\in\Gamma$. 
Following \cite{CLP}, we say that $R$ is \emph{object unital}
if, for all $e\in\Gamma_0$, the ring $R_e$ is unital  with identity element $1_e$, and for all $\gamma\in \Gamma$ and $a\in R_{\gamma}$, we have $1_{r(\gamma)}a=a1_{d(\gamma)}=a$.
\end{defi}
 
It follows from \cite[Proposition 2.1.1]{Lund} that the object unital $\Gamma$-graded ring $R=\bigoplus_{\gamma\in\Gamma}R_\gamma$ is unital if and only if the set
$$\Gamma_0'(R):=\{e\in\Gamma_0 \colon 1_e\neq 0\}$$ 
is finite. 
In this event, $1_R=\sum_{e\in\Gamma_0'}1_e$.

Every group graded (unital) ring is an (object unital) groupoid graded ring.
Next we present other examples of (object unital) groupoid graded rings.

\begin{exem}
\label{exem: M_I(A)}
		Let $A$ be a (unital) ring and let $I$ be a non-empty set. Consider the ring $\M_I(A)$ of $I \times I$ matrices with entries in $A$ where only finitely many entries are nonzero.
		Then $\M_I(A)$ has a natural structure of an (object unital) $I \times I$-graded ring via $\M_I(A)_{(i,j)} := A E_{ij}$, where $E_{ij}$ denotes the elementary matrix with 1 at the $(i,j)$-entry and 0 elsewhere.
        When $I$ is a partially ordered set, we can consider the $I \times I$-graded subring $\UT_I(A)$ of $\M_I(A)$ consisting of matrices $(a_{ij})_{ij}$ such that $a_{ij} = 0$ whenever $i \not\leq j$.
        \qed
	\end{exem}

        \begin{exem} \label{ex: anel de categoria}
        Let $\mathcal{C}$ be a \emph{small preadditive category}, i.e., the class of objects is a set and every class of morphisms is an additive group such that the composition of morphisms is $\mathbb{Z}$-bilinear.
        Denote the set of objects of $\mathcal{C}$ by $\mathcal{C}_0$ and the set of morphisms from $A\in\mathcal{C}_0$ to $B\in\mathcal{C}_0$ by $\mathcal{C}(A,B)$. 
        We have a ring
        \[R_\mathcal{C}:=\bigoplus_{A,B\in \mathcal{C}_0}{\mathcal{C}}(A,B)\] 
        where, given morphisms $f$ and $g$, the product $fg$ is defined by $f\circ g$ if the composition is possible and $0$ otherwise. 
        The ring $R_\mathcal{C}$ has a natural $\mathcal{C}_0\times \mathcal{C}_0$-grading via 
        \[(R_\mathcal{C})_{(A,B)}:={\mathcal{C}}(B,A)\]
        for each $A,B\in \mathcal{C}_0$. 
        Considering the identity morphisms $I_A\in \mathcal{C}(A,A)$, $A\in \mathcal{C}_0$, it is straightforward to show that $R_\mathcal{C}$  is object unital.\qed
\end{exem}

\emph{In this work, all groupoid graded rings are supposed to be object unital. }

\begin{defi}
Let $R$ be a $\Gamma$-graded ring. 
\begin{enumerate}[\rm (1)]
    \item A \emph{graded right ideal} $U$ of $R$ is a right ideal of $R$ such that $U=\bigoplus_{\gamma\in\Gamma}U_\gamma$, where $U_\gamma=U\cap R_\gamma$ for each $\gamma\in\Gamma$. 
\emph{Graded left ideals} are defined similarly.

\item We say that $U$ is a \emph{graded ideal} of $R$ if $U$ is a graded right ideal and a graded left ideal of $R$.
\end{enumerate}
\end{defi}

If $U$ is a graded ideal of the $\Gamma$-graded ring $R$, then $R/U$ is a $\Gamma$-graded ring via $(R/U)_\gamma:=\dfrac{R_\gamma+U}{U}$ for each $\gamma\in\Gamma$. 
Analogously to the ungraded case, there exists a bijective correspondence between graded right (resp. left) ideals of $R/U$ and graded right (resp. left) ideals of $R$ containing $U$.

\begin{defi}
Let $R$ be a $\Gamma$-graded ring, $\gamma\in\Gamma$ and $a\in R_\gamma$.
We say that $a$ is \emph{right} (resp. \emph{left}) \emph{gr-invertible} if there exists $b\in R_{\gamma^{-1}}$ such that $ab=1_{r(\gamma)}$ (resp. $ba=1_{d(\gamma)}$).
The element $a$ is said to be \emph{gr-invertible} if it is both right and left gr-invertible. 
Note that if $b,c\in R_{\gamma^{-1}}$ are such that $ab=1_{r(\gamma)}$ and $ca=1_{d(\gamma)}$, then $b=1_{d(\gamma)}b=cab=c1_{r(\gamma)}=c=:a^{-1}$.
We denote by $\Ugr(R)$ the set of all gr-invertible homogeneous elements of $R$.
If $R\neq0$ and $\Ugr(R)=\h(R)\setminus\{0\}$, then $R$ is called a \emph{gr-division ring}. 
This class of rings was studied in \cite[\S 4]{Artigoarxiv}.
\end{defi}

\begin{defi}
Given $\Gamma$-graded rings $R$ and $S$, a \emph{gr-homomorphism of rings} is a homomorphism of rings $\varphi:R\to S$ which is also a gr-homomorphism of $\Gamma$-graded abelian groups and satisfies $\varphi(1^R_e)=1^S_e$ for all $e\in\Gamma_0$, where $1^R_e$ (resp. $1^S_e$) denotes the unity of the ring $R_e$ (resp. $S_e$). A \emph{gr-isomorphism} of $\Gamma$-graded rings is a bijective gr-homomorphism of $\Gamma$-graded rings. When there exists a gr-isomorphism of $\Gamma$-graded rings $\varphi: R\to S$, we say that $R$ is \emph{gr-isomorphic} to $S$ as rings and we write $R\cong_{gr}S$.
\end{defi}

\subsection{Groupoid graded modules}

Throughout this subsection, let $\Gamma$ be a groupoid and $R=\bigoplus_{\gamma\in \Gamma}R_\gamma$ be a $\Gamma$-graded ring. 

\begin{defi}
If $M$ is a right $R$-module, we say that $M$ is \emph{$\Gamma$-graded} if there exists a family $\{M_\gamma:\gamma\in\Gamma\}$ of additive subgroups of $M$ such that $M=\bigoplus_{\gamma\in\Gamma}M_\gamma$ and, for each $\sigma,\tau\in\Gamma$, we have that $M_\sigma R_\tau \subseteq  M_{\sigma\tau}$ if $\sigma\tau$ is defined, and $M_\sigma R_\tau=0$ otherwise.
A submodule $N$ of $M$ is a \emph{graded submodule}, and we denote $N\subgr M$, if $N=\bigoplus_{\gamma\in\Gamma}N_\gamma$ where $N_\gamma:=N\cap M_\gamma$ for each $\gamma\in\Gamma$ (or, equivalently, $N$ is generated by homogeneous elements of $M$). 
In this case, the quotient module $M/N$ is $\Gamma$-graded via $(M/N)_\gamma:=\dfrac{M_\gamma+N}{N}$ for each $\gamma\in\Gamma$.
Similarly to the ungraded context, there exists a bijective correspondence between the graded submodules of $M/N$ and the graded submodules of $M$ containing $N$. We also have that any intersection and any sum of graded submodules of $M$ is a graded submodule of $M$.
Analogously, we define \emph{$\Gamma$-graded left $R$-modules} and its graded submodules.
\end{defi}

Remember that a right $R$-module $M$ is said to be \emph{unital} if $MR=M$. 
By \cite[Proposition 5]{CLP}, a $\Gamma$-graded right $R$-module $M$ is unital if and only if $m_\sigma1_{d(\sigma)}=m_\sigma$ 
for all $\sigma\in\Gamma$ and  $m_\sigma\in M_\sigma$.

\medskip

In this work, \emph{all groupoid graded modules are supposed to be unital.}

\medskip

Clearly $R_R$ is a (unital) $\Gamma$-graded right $R$-module and {graded right ideals} of $R$ are precisely the graded submodules of $R_R$.

\begin{defi}
Let $R$ be a $\Gamma$-graded ring, and let $\{M_j:j\in J\}$ be a family of $\Gamma$-graded right $R$-modules. 
We define the \emph{graded direct product} of this family, denoted by $\prod_{j\in J}^{\rm gr}M_j$, as the $\Gamma$-graded right $R$-module whose homogeneous component of degree $\gamma\in\Gamma$ is the additive group $\prod_{j\in J}(M_j)_\gamma$.
The direct sum $\bigoplus_{j\in J}M_j$ is also a $\Gamma$-graded right $R$-module via $(\bigoplus_{j\in J}M_j)_\gamma:=\bigoplus_{j\in J}(M_j)_\gamma$ for each $\gamma\in\Gamma$. 
If $M$ is a $\Gamma$-graded right $R$-module and $N\subgr M$, we say that $N$ is a \emph{graded direct summand} of $M$ if there exists a graded submodule $X$ of $M$ such that $M=N\oplus X$. 
\end{defi}

\begin{defi}
A $\Gamma$-graded right $R$-module $M$ is called \emph{gr-simple} if $M\neq\{0\}$ and its only graded submodules are $\{0\}$ and $M$. If $M$ is a sum of gr-simple graded submodules, then $M$ is said to be \emph{gr-semisimple}. 
By \cite[Proposition 52]{CLP}, every gr-semisimple graded module is a direct sum of gr-simple graded submodules.
If $R_R$ is gr-semisimple, then we say that $R$ is a \emph{gr-semisimple ring}. 
By \cite[Corollary 5.28]{Artigoarxiv}, $R$ is a gr-semisimple ring if and only if $_RR$ is gr-semisimple.
\end{defi}

\begin{defi}
Let  $M$ be a $\Gamma$-graded right  $R$-module and $\sigma\in\Gamma$. 
For each $\gamma\in\Gamma$, define $M(\sigma)_\gamma:=M_{\sigma\gamma}$. 
The \emph{shift  of $M$ by $\sigma$} is the $\Gamma$-graded right $R$-module $M(\sigma):=\bigoplus_{\gamma\in\Gamma}M(\sigma)_\gamma$. 
By \cite[Lemma 2.6]{Artigoarxiv}, $M(\sigma)$ equals to $M(r(\sigma))$ as $R$-modules.
If $\Delta_0\subseteq\Gamma_0$, we denote $M(\Delta_0):=\bigoplus_{e\in\Delta_0}M(e)$.
Analogously, if $M$ is a $\Gamma$-graded left $R$-module and  $\sigma\in\Gamma$, then the \emph{shift  of $M$ by $\sigma$} is the $\Gamma$-graded left $R$-module $(\sigma)M:=\bigoplus_{\gamma\in\Gamma}((\sigma)M)_\gamma$, where $((\sigma)M)_\gamma:=M_{\gamma\sigma}$ for each $\gamma\in\Gamma$.
One can verify that $(\sigma)M$ equals to $(d(\sigma))M$ as $R$-modules.
\end{defi}

Note that if $M$ is a $\Gamma$-graded right  
$R$-module, then $M=\bigoplus_{e\in\Gamma_0}M(e)$. This decomposition will be important, so we define $$\Gamma'_0(M):=\{e\in\Gamma_0:M(e)\neq0\}.$$
 By \cite[Lemma 2.7]{Artigoarxiv}, $\Gamma'_0(R)=\Gamma'_0(R_R)=\Gamma'_0(_RR)$. 

\begin{defi}
Let  $M$ and $N$ be $\Gamma$-graded right  $R$-modules.
Given $\gamma\in\Gamma$, we say that a homomorphism of right $R$-modules $g:M\to N$ is a \emph{homomorphism of degree $\gamma$} if $g(M_\sigma)\subseteq N_{\gamma\sigma}$  for each $\sigma\in\Gamma$, where we understand that $N_{\gamma\sigma}=\{0\}$ whenever $\gamma\sigma$ is not defined in $\Gamma$. 
By \cite[Lemma 2.9]{Artigoarxiv}, if $g:M\to N$ has degree $\gamma$, then $M(\Gamma_0\setminus\{d(\gamma)\})\subseteq\ker g$ and $\im g\subseteq N(r(\gamma))$.
For all $\gamma\in\Gamma$, $\HOM_R(M,N)_\gamma$ will denote the additive group of the homomorphisms $g:M\to N$ of degree $\gamma$. Hence, we can define the $\Gamma$-graded additive group $\HOM_R(M,N):=\bigoplus_{\gamma\in\Gamma}\HOM_R(M,N)_\gamma$. 
For left $R$-modules, we replace the condition $g(M_\sigma)\subseteq N_{\gamma\sigma}$ with $(M_\sigma)g\subseteq N_{\sigma\gamma}$.
 Considering the decomposition $M=\bigoplus_{e\in\Gamma_0}M(e)$, we denote by $\mathds{1}_e$ the canonical projection $M\to M(e)$, which is an element of $\END_R(M)_e$. By \cite[Proposition 2.10]{Artigoarxiv}, this is the unity of degree $e$ in the $\Gamma$-graded ring $\END_R(M):=\HOM_R(M,M)$. 
Furthermore, by \cite[Lemma 3.4]{Artigoarxiv}, $R\isogr\END(R_R)$ as $\Gamma$-graded rings.
\end{defi}

\begin{defi}
Let $M$, $N$ be $\Gamma$-graded right  $R$-modules. A homomorphism of modules $g:M\to N$  is said to be a \emph{gr-homomorphism of modules} if $g(M_\sigma)\subseteq N_\sigma$ for all $\sigma\in\Gamma$. 
We denote by $\Homgr(M,N)$ the abelian group consisting of the gr-homomorphisms of modules $g:M\to N$.  
We denote by $\Endgr(M)$ the ring of gr-homomorphisms $M\to M$. 
It is easy to show that if $g\in\Homgr(M,N)$, then $\ker g$ is a graded submodule of $M$ and $\im g$ is a graded submodule of $N$.
A bijective gr-homomorphism of modules $g:M\to N$ is called a \emph{gr-isomorphism of modules}. 
We say that $M$ is \emph{gr-isomorphic} to $N$ as modules, denoted by $M\cong_{gr}N$, if there exists a gr-isomorphism of modules $g:M\to N$.
\end{defi}


\section{Chain conditions on graded modules}
\label{sec: chain conditions}

We present several chain conditions that can be defined for groupoid graded rings and modules. The first kind,  we refer to  them as standard chain conditions, are the straightforward generalization of chain conditions for group graded rings and modules. Although useful, they fail to address important cases because most groupoid graded   modules (rings) are naturally decomposed as an infinite direct sum of graded submodules (one-sided ideals).
The second kind of chain conditions  apply more broadly and are called $\Gamma_0$-chain conditions. They are the natural extension to groupoid graded rings of the concept  categorically artinian (categorically noetherian) given in  \cite[Definition~1.1]{Abrams_ArandaPino_Perera_SilesMolina} and \cite[Section~4.2]{Abrams_Ara_SilesMolina}. Strongly $\Gamma_0$-chain conditions still deal with infinite direct sums, but they refer to a certain kind of chains to be defined later which appear naturally as kernels and images of gr-homomorphims and as powers of graded ideals among others.

\subsection{Standard chain conditions}

\begin{defi}
	Let $R$ be a $\Gamma$-graded ring and $M$ be a $\Gamma$-graded right $R$-module.
	\begin{enumerate}[\rm (1)]
		\item We say that $M$ is \emph{gr-artinian} if $M$ satisfies the  descending chain condition (DCC) for graded submodules, that is, there does not exist a strictly decreasing infinite chain of graded submodules of $M$.

        \item We say that the ring $R$ is \emph{right gr-artinian} if the right $R$-module $R_R$ is gr-artinian.
        
        \item We say that $M$ is \emph{gr-noetherian} if $M$ satisfies the ascending chain condition (ACC) for graded submodules, that is, there does not exist a strictly increasing infinite chain of graded submodules of $M$.

        \item We say that the ring $R$ is \emph{right gr-noetherian} if the right $R$-module $R_R$ is gr-noetherian.

        \item We say that $M$ is \emph{finitely generated} if it is generated by finitely many elements. Equivalently, it is generated by finitely many homogeneous elements.

        \item We say that $M$ is \emph{finitely gr-cogenerated} if, for every family $\{M_i:i\in I\}$ of graded submodules of $M$ such that $\bigcap_{i\in I}M_i=0$, there exist $i_1,\dots,i_n\in I$ such that $\bigcap_{t=1}^nM_{i_t}=0$.
	\end{enumerate}
\end{defi}

\begin{obs}
\label{obs: G0-conditions}
Let $R$ be a $\Gamma$-graded ring and $M$ be a $\Gamma$-graded right  $R$-module. 

    (1) It is straightforward to check that $M$ is finitely generated if and only if, 
    for every family $\{M_i:i\in I\}$ of graded submodules of $M$ such that $\sum_{i\in I}M_i=M$, there exist $i_1,\dots,i_n\in I$ such that $\sum_{t=1}^nM_{i_t}=M$. In this case, $\Gamma'_0(M)$ is finite. 

    Notice this equivalence of the concept of finitely generated graded module is the dual statement of the concept of finitely gr-cogenerated module.

    (2) If $M$ is finitely gr-cogenerated, then $\Gamma'_0(M)$ is finite, since $\bigcap_{e\in\Gamma'_0(M)}M(\Gamma_0\setminus\{e\})=0$ and $\bigcap_{t=1}^nM(\Gamma_0\setminus\{e_t\})=M(\Gamma_0\setminus\{e_1,\dots e_n\})$ for each $e_1,\dots, e_n\in\Gamma'_0(M)$.

    (3) If $M$ is gr-artinian or gr-noetherian, then, from the direct sum decomposition $M=\bigoplus_{e\in\Gamma_0}M(e)$, we have that $M(e)=0$ for all but finitely many $e\in\Gamma_0$.
\end{obs}

As we have just observed, being a finitely generated (resp. finitely gr-cogenerated, gr-artinian, gr-noetherian) module implies that $\Gamma'_0(M)$ is finite. This motivates the definition of the $\Gamma_0$-concepts in \cref{def: subsec:Gamma_0 conditions}.

\medskip 

We begin with the following characterization of gr-artinian/gr-noetherian modules.

\begin{prop}
\label{prop: gr-noeth <-> todo submod eh fg}
    Let $R$ be a $\Gamma$-graded ring and $M$ be a $\Gamma$-graded right $R$-module. The following assertions hold:
    \begin{enumerate}[\rm (1)]
        \item $M$ is gr-noetherian if and only if  $N$ is finitely generated for each $N\subgr M$.
        \item $M$ is gr-artinian if and only if $M/N$ is finitely gr-cogenerated for each $N\subgr M$.
    \end{enumerate}
\end{prop}

\begin{proof}
	(1) As in the ungraded context, if $M$ is gr-noetherian and $N\subgr M$, then the set $\{L\subgr N:L \textrm{~is finitely generated}\}$ has a maximal element, which must be $N$.
    Conversely, if $M_1\subseteq M_2\subseteq\cdots$ is an increasing infinite chain of graded submodules of $M$ and $N:=\bigcup_{i=1}^\infty M_i$ is finitely generated, then there exists $k\geq1$ such that $N=M_k$, and thus $M_k=M_{k+l}$ for each $l\geq 0$.

    (2) We adapt the proof from the ungraded case \cite[Proposition 10.10]{AndersonFuller}.
    Suppose that $M$ is gr-artinian, $N\subgr M$, and $\{P_i:i\in I\}$ is a family of graded submodules of $M$ containing $N$ such that $\bigcap_{i\in I}P_i/N=0$.
    Since $M$ is gr-artinian, the family $\{\bigcap_{i \in I'} P_i : I' \subseteq I \text{ finite} \}$ has a minimal element, which is easily seen to be $N$.
    Therefore, $\bigcap_{i\in I'}P_i/N=\frac{\bigcap_{i\in I'}P_i}{N}=0$ for some finite subset $I'\subseteq I$, and it follows that $M/N$ is finitely gr-cogenerated.
    Conversely, suppose that $M/N$ is finitely gr-cogenerated for every $N\subgr M$.
    If $M_1\supseteq M_2\supseteq\cdots$ is a descending infinite chain of graded submodules of $M$ and $N:=\bigcap_{i=1}^\infty M_i$, then since $\bigcap_{i=1}^\infty M_i/N=0$, there exists $k\in\mathbb{N}$ such that $\bigcap_{i=1}^k M_i/N=0$, and therefore $N=\bigcap_{i=1}^k M_i=M_k$, from which it follows that $M_k=M_{k+l}$ for every $l\geq 0$.
\end{proof}

Now we show that these chain conditions are inherited by submodules and extensions. 

\begin{prop}
\label{prop: M art ou noeth <--> N e M/N tbm}
    Let $R$ be a $\Gamma$-graded ring, $M$ be a $\Gamma$-graded right $R$-module and $N$ be a graded submodule of $M$. Then $M$ is gr-noetherian (resp. gr-artinian) if and only if both $N$ and $M/N$ are gr-noetherian (resp. gr-artinian).
\end{prop}

\begin{proof}
	The ``only if'' part is clear.
    Suppose that both $N$ and $M/N$ are gr-noetherian. Let $M_1\subseteq M_2\subseteq\cdots$ be an increasing infinite chain of graded submodules of $M$. 
    Then $M_1\cap N\subseteq M_2\cap N\subseteq\cdots$ is an increasing infinite chain of graded submodules of $N$ and $\frac{M_1+N}{N}\subseteq \frac{M_2+N}{N}\subseteq\cdots$ is an increasing infinite chain of graded submodules of $M/N$. 
    Thus, there exists $k\geq1$ such that $M_k\cap N=M_{k+l}\cap N$ and $\frac{M_k+N}{N}=\frac{M_{k+l}+N}{N}$ for each $l\geq 0$. 
    It easy to check that $M_k=M_{k+l}$ for each $l\geq 0$, proving that $M$ is gr-noetherian. 
    The gr-artinian case is similar.
\end{proof}

Note that \Cref{prop: M art ou noeth <--> N e M/N tbm} implies that if the $\Gamma$-graded ring $R$ is  right gr-noetherian (resp. artinian) and $I$ is a graded ideal of $R$, then the graded ring $R/I$ is right gr-noetherian (resp. artinian) because the graded right ideals of $R/I$ are the same as its right graded $R$-submodules.

\begin{prop}
\label{prop: soma finita de art/noeth}
    Let $R$ be a $\Gamma$-graded ring. 
    \begin{enumerate}[\rm (1)]
        \item Any finite direct sum of gr-noetherian (resp. gr-artinian) $\Gamma$-graded right $R$-modules is gr-noetherian (resp. gr-artinian).
        \item If $N_1,\dotsc,N_n$ are gr-noetherian (resp. gr-artinian) graded submodules of the $\Gamma$-graded module $M$, then $N_1+\dotsb+N_n$ is gr-noetherian (resp. gr-artinian).
    \end{enumerate}
\end{prop}

\begin{proof}
    (1) Let $M_1,\dots,M_n$ be gr-noetherian (resp. gr-artinian) $\Gamma$-graded right $R$-modules. Since ${(M_1\oplus M_2)}/{M_2}\,\isogr \,{M_1},$ 
    it follows from \Cref{prop: M art ou noeth <--> N e M/N tbm} that $M_1\oplus M_2$ is gr-noetherian (resp. gr-artinian). 
    For each $1< t\leq n$, we have $\frac{M_1\oplus\cdots \oplus M_t}{M_t}\isogr M_1\oplus\cdots \oplus M_{t-1}$. 
    Thus, the result follows by induction on $t$.

    (2) Consider the natural gr-homomorphism $g\colon N_1\oplus\dotsb\oplus N_n\rightarrow M$. Then $\frac{N_1\oplus\dotsb\oplus N_n}{\ker g}\,\isogr\, \im g=N_1+\dotsb+N_n$. The result follows from (1) and \Cref{prop: M art ou noeth <--> N e M/N tbm}. 
\end{proof}

\begin{coro}
\label{coro: R noeth --> todo fg eh noeth}
    Let $R$ be a $\Gamma$-graded ring which is right gr-noetherian (resp. right gr-artinian) ring
    and $M$ be a $\Gamma$-graded right $R$-module. If $M$ is finitely generated, then $M$ is gr-noetherian (resp. gr-artinian).
\end{coro}

\begin{proof}
	(1)     Let $m_1,\dots,m_k\in\h(M)$ be generators of $M$. 
    If we set $\gamma_i:=\deg(m_i)$ for each $1\leq i \leq k$, then
    \begin{align*}
        g: R(\gamma_1^{-1})\oplus\cdots\oplus R(\gamma_k^{-1})&\longrightarrow M\\
        (a_1,\dots,a_k)&\longmapsto m_1a_1+\cdots+m_ka_k 
    \end{align*}
    is a surjective gr-homomorphism. 
    Thus, $M\isogr\frac{R(\gamma_1^{-1})\oplus\cdots\oplus R(\gamma_k^{-1})}{\ker (g)}$ is gr-noetherian (resp. gr-artinian) by \Cref{prop: soma finita de art/noeth}(1) and \Cref{prop: M art ou noeth <--> N e M/N tbm}, since any shift of a gr-noetherian (resp. gr-artinian) module is also gr-noetherian (resp. gr-artinian).
\end{proof}

In order to state the following corollary, we need the following definition.

\begin{defi}
    Let $R$ be a $\Gamma$-graded ring. We say that $S\subseteq R$ is a \emph{full gr-subring} of $R$ if $S$ is closed under sums and products, $S=\bigoplus_{\gamma\in\Gamma}(S\cap R_\gamma)$ and $1_e$, the identity element of the ring $R_e$, belongs to $S_e$ for all $e\in \Gamma_0$.
\end{defi}

 Note that if $S$ is a full gr-subring of the $\Gamma$-graded ring $R$, then $S$ is an (object unital) $\Gamma$-graded ring and we can regard $R$ as (unital) $\Gamma$-graded right $S$-module.

\begin{coro}
    Let $S$ be a  full gr-subring of the $\Gamma$-graded ring $R$. If $S$ is right gr-noetherian (resp. artinian) and $R$ is finitely generated as a right $S$-module, then $R$ is a right gr-noetherian (resp. gr-artinian) ring.
\end{coro}

\begin{proof}
By  \Cref{coro: R noeth --> todo fg eh noeth}, $R$ is a right gr-noetherian (resp. artinian) $S$-module. Since all graded right ideals of $R$ are also graded right $S$-submodules, the result holds.   
\end{proof}

Now we present some results that will be needed later. They relate chain conditions of a ring or module with those of a certain subobject. 

\begin{lema}
\label{lem: M gr-art => M1_f gr-art}
     Let $R$ be a $\Gamma$-graded ring, $\Delta_0\subseteq\Gamma_0$, and  $R_\Delta=\bigoplus_{e,f\in\Delta_0}1_eR1_f$.
   The following statements hold:
    \begin{enumerate}[\rm (1)]
        \item  If $M$ is a $\Gamma$-graded gr-artinian (resp. gr-noetherian)  right $R$-module, then $\sum_{e\in\Delta_0}M1_e$ is a gr-artinian (resp. gr-noetherian) right $R_\Delta$-module.
        \item If $R(\Delta_0)=\bigoplus_{e\in\Delta_0}R(e)$ is a gr-artinian (resp. gr-noetherian) $R$-module, then $R_\Delta$ is a right gr-artinian (resp. right gr-noetherian) ring.
    \end{enumerate}
    \end{lema}

\begin{proof}
We only present the gr-artinian case of the statements; the gr-noetherian case is proved similarly.

(1)    An infinite strict descending chain of graded submodules of $\sum_{e\in\Delta_0}M1_e$ 
$$M_1\supsetneq M_2\supsetneq \dotsb \supsetneq M_n\supsetneq \dotsb$$
implies the existence of the infinite descending chain of graded submodules of $M$
$$M_1R\supseteq M_2R\supseteq \dotsb \supseteq M_nR\supseteq \dotsb.$$  
Note that $M_iR1_e=M_i1_e$  for each $e\in\Delta_0$ because the homogeneous elements in $M_iR\setminus M_i$ are of degree $\gamma\in \Gamma$ with $d(\gamma)\notin \Delta_0$.
Hence, this last chain is strict since $M_nR=M_{n+1}R$ implies $$M_n=\sum_{e\in\Delta_0}M_nR1_e=\sum_{e\in\Delta_0}M_{n+1}R1_e=M_{n+1}$$ for each $n\in\mathbb{N}$.

Statement (2) follows from the fact that $\sum_{e\in\Delta_0}R(\Delta_0)1_e=R_\Delta$.
\end{proof}

\begin{lema}
\label{lem: M gr-art <= M1_f gr-art}
    Let $R$ be a $\Gamma$-graded ring such that $R=\sum_{f\in\Delta_0}R1_fR$ for some finite subset $\Delta_0\subseteq\Gamma_0$. The following statements hold:
    \begin{enumerate}[\rm (1)]
        \item Suppose that $M$ is a $\Gamma$-graded right $R$-module such that $M1_f$ is a gr-artinian (resp. gr-noetherian) graded right $1_fR1_f$-module for each $f\in\Delta_0$. 
    Then $M$ is a gr-artinian (resp. gr-noetherian) $R$-module.

    \item If $e\in\Gamma_0$ 
    and $1_eR1_f$ is a gr-artinian (resp. gr-noetherian)  right $1_fR1_f$-module for each $f\in\Delta_0$, then $R(e)$ is a gr-artinian (resp. gr-noetherian) $R$-module.
    \end{enumerate} 
    
   \end{lema}

\begin{proof}
We only present the gr-artinian case of the statements; the gr-noetherian case is proved similarly.

 (1)   Let $M_1\supseteq M_2\supseteq \dotsb \supseteq M_n\supseteq \dotsb$ be a descending chain of graded submodules of $M$.
For each $f\in\Delta_0$, we have a descending chain of graded $1_fR1_f$-submodules of $M1_f$
\[M_11_f\supseteq M_21_f\supseteq \dotsb \supseteq M_n1_f\supseteq \dotsb.\]
Since $M1_f$ is a gr-artinian graded right $1_fR1_f$-module and $\Delta_0$ is finite, there exists $n\in\mathbb{N}$ such that $M_n1_f=M_{n+l}1_f$ for all $f\in\Delta_0$ and $l\geq0$.
Then 
\begin{align*}
  M_n=M_nR=\sum_{f\in\Delta_0}M_nR1_fR &=\sum_{f\in\Delta_0}M_n1_fR  \\
  &=\sum_{f\in\Delta_0}M_{n+l}1_fR =\sum_{f\in\Delta_0}M_{n+l}R1_fR=M_{n+l}  
\end{align*} for each $l\geq0$.
Therefore, $M$ is gr-artinian.

(2) This statement follows applying (1) to $M=R(e)=1_eR$.
\end{proof}

The following result is useful to prove versions of Fitting's Lemma for gr-endomorphisms.
 The proof is similar to that of the classical setting \cite[Lemma~11.6]{AndersonFuller}.

\begin{lema}
\label{lem: Fitting para gr-end}
    Let $R$ be a $\Gamma$-graded ring, $M$ a $\Gamma$-graded right $R$-module, and $g\in \Endgr(M)$. 
    The following assertions hold:
    \begin{enumerate}[\rm (1)]
        \item Suppose that $M$ is gr-noetherian. 
        Then there exists $n\in\mathbb{N}$ such that $g(\ker g^n)\subseteq\ker g^n$ and $\ker g^n\cap\im g^n=0$.
        Moreover, if $g$ is surjective, then $g$ is a gr-isomorphism.
        \item Suppose that $M$ is gr-artinian. 
        Then there exists $n\in\mathbb{N}$ such that $g(\im g^n)\supseteq \im g^n$ and $\ker g^n+\im g^n=M$.
        Moreover, if $g$ is injective, then $g$ is a gr-isomorphism.
    \end{enumerate}
\end{lema}
	
\begin{proof}
    (1) Since $M$ is gr-noetherian and we have the following chain
    \[0\subseteq \ker g\subseteq \ker g^2\subseteq \ker g^3\subseteq\cdots,\]
    it follows that there exists $n\in\mathbb{N}$ such that $\ker g^n = \ker g^{n+l}$ for each $l\geq0$.
    In particular, $g(\ker g^n)\subseteq \ker g^n$.
    If $y\in\ker g^n$ and $y=g^n(x)$ for some $x\in M$, then $g^{2n}(x)=g^n(y)=0$, and it follows that $x\in\ker g^{2n}=\ker g^n$, whence $y=0$.
    Therefore, $\ker g^n\cap\im g^n=0$.
    If $g$ is surjective, then $\ker g\subseteq \ker g^n\cap \im g = \ker g^n\cap \im g^n=0$, whence $g$ is a gr-isomorphism.

    (2) Since  $M$ is gr-artinian and we have the following chain
    \[M\supseteq\im g\supseteq\im g^2\supseteq\im g^3\supseteq\cdots,\]
    it follows that there exists $n\in\mathbb{N}$ such that $\im g^n = \im g^{n+l}$ for each $l\geq0$.
    In particular, $g(\im g^n)\supseteq \im g^n$.
    If $x\in M$, then since $g^n(x)\in\im g^{2n}$, we have $g^n(x)=g^{2n}(y)$ for some $y\in M$, and it follows that $x=(x-g^n(y))+g^n(y)\in\ker g^n+\im g^n$.
    Therefore, $\ker g^n+\im g^n=M$.
    If $g$ is injective, then $g^n$ also is injective, and it follows that $M=\ker g^n+\im g^n=\im g^n$, from which $g$ is surjective.
\end{proof}

Similarly, we have the following result for homomorphisms with degree.

\begin{lema}
\label{lem: Fitting}
    Let $R$ be a $\Gamma$-graded ring and $M$ a $\Gamma$-graded right $R$-module. 
    Take $e\in\Gamma_0$, $\gamma\in e\Gamma e$, and $g\in \END_R(M)_\gamma$. 
    The following assertions hold:
    \begin{enumerate}[\rm (1)]
        \item Suppose that $M(e)$ is gr-noetherian. 
        Then there exists $n\in\mathbb{N}$ such that $g(\ker g^n)\subseteq\ker g^n$ and $\ker g^n\cap\im g^n=0$.
        Moreover, if $\im g= M(e)$, then $g$ is gr-invertible in $\END_R(M)$.
        \item Suppose that $M(e)$ is gr-artinian. 
        Then there exists $n\in\mathbb{N}$ such that $g(\im g^n)\supseteq \im g^n$ and $\ker g^n+\im g^n=M$.
        Moreover, if $ M(e)\cap \ker g=0$, then $g$ is gr-invertible in $\END_R(M)$.
    \end{enumerate}
\end{lema}
	
\begin{proof}
    (1) Since $M(e)$ is gr-noetherian, the following chain terminates:
    \[0\subseteq M(e)\cap\ker g\subseteq M(e)\cap\ker g^2\subseteq M(e)\cap\ker g^3\subseteq\cdots.\]
    Thus, there exists $n\in\mathbb{N}$ such that $M(e)\cap\ker g^n = M(e)\cap\ker g^{n+l}$ for each $l\geq0$. 
    Since $M(\Gamma_0\setminus\{e\})\subseteq\ker g^t$ for each $t\geq0$, it follows that $\ker g^n = \ker g^{n+l}$ for each $l\geq0$.
    In particular, $g(\ker g^n)\subseteq \ker g^n$.
    With the same argument as in \Cref{lem: Fitting para gr-end}(1), we show that $\ker g^n\cap\im g^n=0$.
    Now suppose that $\im g=M(e)$.
    Then $\im g^n=M(e)$, and it follows that $M(e)\cap \ker g\subseteq M(e)\cap \ker g^n=0$.
    Hence, the restriction $g|_{M(e)}:M(e)\to M(e)$ is bijective, and its inverse extends to a gr-inverse of $g$ in $\END_R(M)$.

    (2) Since  $M(e)$ is gr-artinian, the following chain terminates:
    \[M(e)\supseteq\im g\supseteq\im g^2\supseteq\im g^3\supseteq\cdots.\]
    Thus, there exists $n\geq0$ such that $\im g^n=\im g^{n+l}$ for each $l\geq0$. 
    In particular, $g(\im g^n)\supseteq \im g^n$.
    With the same argument as in \Cref{lem: Fitting para gr-end}(2), we prove that $\ker g^n+\im g^n=M$.
    Now suppose that $M(e)\cap\ker g=0$.
    Then $M(e)\cap\ker g^n=0$, and it follows that $\ker g^n=M(\Gamma_0\setminus\{e\})$, from which $\im g^n=M(e)$ and $\im g=M(e)$.
    Hence, the restriction $g|_{M(e)}:M(e)\to M(e)$ is bijective, and its inverse extends to a gr-inverse of $g$ in $\END_R(M)$.
\end{proof}

Now we deal with the concept of finite gr-length and its properties.

\begin{defi}
    Let $R$ be a $\Gamma$-graded ring and $M$ be a $\Gamma$-graded right $R$-module. A \emph{graded submodule series} for $M$ is a finite chain
    \begin{equation}
    \label{eq: serie de submod}
        M_0=\{0\}\subgr M_1\subgr \cdots \subgr M_n=M.
    \end{equation}
    A \emph{gr-refinement} of \eqref{eq: serie de submod} is a graded submodule series
    \[M'_0=\{0\}\subgr M'_1\subgr \cdots \subgr M'_{n'}=M\]
    such that $\{M_j:0\leq j\leq n\}\subseteq \{M'_j:0\leq j\leq n'\}$.
    A graded submodule series $N_0=\{0\}\subgr N_1\subgr \cdots \subgr N_t=M$ is \emph{gr-equivalent} to \eqref{eq: serie de submod} if $n=t$ and there exists a permutation $\pi$ of $\{1,2,\dots,n\}$ such that $M_j/M_{j-1}\isogr N_{\pi(j)}/N_{\pi(j)-1}$ for each $1\leq j\leq n$. 
    We say that \eqref{eq: serie de submod} is a \emph{gr-composition series} if $M_j/M_{j-1}$ is gr-simple for each $1\leq j \leq n$. 
    In this case, $n$ is called the \emph{gr-length} of the composition series, and each $M_j/M_{j-1}$ is called a \emph{gr-composition factor}. 
    When $M$ has a gr-composition series, we say that $M$ has \emph{finite gr-length}. 
\end{defi}

 Note that if $M=M(e_1)\oplus\cdots\oplus M(e_n)$ for some $e_1,\dots,e_n$, then $M$ has a graded submodule series as \eqref{eq: serie de submod} putting $M_j:=M(e_1)\oplus\cdots\oplus M(e_j)$ for each $j=1,\dots,n$.

We present some basic properties of graded submodule series, whose proofs are analogous to those in the ungraded case.

\begin{prop}
\label{prop: comp finito = art + noet}
    Let $R$ be a $\Gamma$-graded ring and $M$ be a $\Gamma$-graded right $R$-module. Then $M$ has {finite gr-length} if and only if $M$ is both gr-artinian and gr-noetherian. 
\end{prop}

\begin{proof}
    If $M_0=\{0\}\subgr M_1\subgr \cdots \subgr M_n=M$ is a gr-composition series, then since gr-simple modules are gr-artinian and gr-noetherian, it follows from \Cref{prop: M art ou noeth <--> N e M/N tbm} that $M_j$ is both gr-artinian and gr-noetherian for each $0\leq j\leq n$.

    Conversely, suppose $M$ is gr-artinian and gr-noetherian. Since $M$ is gr-artinian, we can build a chain $M_0=\{0\}\subgr M_1\subgr M_2\subgr \cdots$ such that, if $M/M_j\neq0$, then $M_{j+1}/M_j$ is a gr-simple graded submodule of $M/M_j$. Since $M$ is gr-noetherian, we have $M_n=M$ for some $n\in\mathbb{N}$.
\end{proof}

The following result is a graded version of Schreier and Jordan-Hölder Theorems.

\begin{teo}
\label{teo: Jordan-Holder}
    Let $R$ be a $\Gamma$-graded ring and $M$ be a $\Gamma$-graded right $R$-module. 
    The following assertions hold:
    \begin{enumerate}[\rm (1)]
        \item Any two graded submodules series of $M$ have isomorphic refinements.
        \item Suppose that $M$ has finite gr-length. 
        Then any two gr-composition series for $M$ are gr-equivalent.
    \end{enumerate}
\end{teo}

\begin{proof}
    This follows in an analogous way to \cite[Theorems 4.10 and 4.11]{Goodearl}.
\end{proof}

\begin{defi}
    Let $R$ be a $\Gamma$-graded ring and $M$ be a $\Gamma$-graded right $R$-module of finite gr-length. 
    By \Cref{teo: Jordan-Holder}(2), we can define the (\emph{gr-composition}) \emph{gr-length} of $M$ as the gr-length $c_{\rm gr}(M)\in\mathbb{N}$ of any gr-composition series for $M$.
\end{defi}

Similarly to the ungraded case, the only graded module with gr-composition gr-length 0 is the zero module, and the graded modules with gr-composition gr-length 1 are precisely the gr-simple modules.
A direct sum of $n$ gr-simple graded modules is a module with gr-composition gr-length $n$. 

\begin{prop}\label{prop: comprimento finito M <---> comp finito N  e M/N}
 Let $R$ be a $\Gamma$-graded ring, $M$ be a $\Gamma$-graded right $R$-module and $N\subgr M$. Then $M$   has finite  gr-length if and only if $N$ and $M/N$ have finite gr-length. In this event, $c_{\rm gr}(M)=c_{\rm gr}(N)+c_{\rm gr}(M/N)$.
\end{prop}
\begin{proof}
 The first statement follows from \Cref{prop: M art ou noeth <--> N e M/N tbm} and \Cref{prop: comp finito = art + noet}. The second statement can be verified    in an analogous way to \cite[Proposition 4.12]{Goodearl}.
\end{proof}

Now we present a Fitting's Lemma for gr-homomorphisms and homomorphisms with degree between graded modules. The proof is very similar to the ungraded original result.
\begin{prop}
\label{prop: Fitting Lemma for graded modules of fin. gr-length}
    Let $R$ be a $\Gamma$-graded ring and $M$ a $\Gamma$-graded right $R$-module. 
    The following assertions hold:
    \begin{enumerate}[\rm (1)]
        \item Suppose that $g\in \Endgr(M)$ and $M$ has finite gr-length.
        Then there exists $n\in\mathbb{N}$ such that $M=\ker g^n\oplus\im g^n$, $g(\ker g^{n})\subseteq\ker g^{n}$, and the restriction $g|_{\im g^n}:\im g^n\to \im g^n$ is a gr-isomorphism.
        Moreover, if $g$ is injective or surjective, then $g$ is a gr-isomorphism.
        \item Suppose that $e\in\Gamma_0$, $\gamma\in e\Gamma e$, $g\in \END_R(M)_\gamma$, and $M(e)$ has finite gr-length.
        Then there exists $n\in\mathbb{N}$ such that $M=\ker g^n\oplus\im g^n$, $g(\ker g^{n})\subseteq\ker g^{n}$, and the restriction $g|_{\im g^n}:\im g^n\to \im g^n$ is bijective.
        Moreover, if $M(e)\cap \ker g=0$ or $\im g= M(e)$, then $g$ is gr-invertible in $\END_R(M)$.
    \end{enumerate}
\end{prop}
	
\begin{proof}
    (1) $M$ is gr-artinian and gr-noetherian by \Cref{prop: comp finito = art + noet}.
    By \Cref{lem: Fitting para gr-end}, there exist $n_1,n_2\in\mathbb{N}$ such that $g(\ker g^{n_1})\subseteq\ker g^{n_1}$, $g(\im g^{n_2})\supseteq \im g^{n_2}$, $\ker g^{n_1}\cap\im g^{n_1}=0$ and $\ker g^{n_2}+\im g^{n_2}=M$.
    Moreover, if $g$ is injective or surjective, then $g$ is a gr-isomorphism.
    Note that $\ker g^{n_1}=\ker g^{n_1+l}$ and $\im g^{n_2}=\im g^{n_2+l}$ for all $l\geq0$.
    Therefore, $n:=\max\{n_1,n_2\}$ is such that $\ker g^{n}\oplus\im g^{n}=M$, $g(\ker g^{n})\subseteq\ker g^{n}$, and $g(\im g^{n})\supseteq \im g^{n}$.
    In particular, the restriction $g|_{\im g^n}:\im g^n\to \im g^n$ is surjective.
    The injectivity follows from $\ker(g|_{\im g^n})=\ker g\cap \im g^n \subseteq \ker g^n\cap \im g^n=\{0\}$. 

    (2) Follows in a similar way to (1), using \Cref{lem: Fitting}.
\end{proof}


\subsection{$\Gamma_0$-chain conditions}\label{subsec:Gamma_0 conditions}

These are the most general chain conditions that we will consider.

\begin{defi}\label{def: subsec:Gamma_0 conditions}
	Let $R$ be a $\Gamma$-graded ring and  $M=\bigoplus_{\gamma\in\Gamma}M_\gamma$ be a $\Gamma$-graded right $R$-module.
    \begin{enumerate}[\rm(1)]
        \item We say that $M$ is \emph{$\Gamma_0$-artinian} if $M(e)$ is a gr-artinian $R$-module for each $e\in\Gamma_0$. 
        \item The ring $R$ is said to be \emph{right $\Gamma_0$-artinian} if $R_R$ is a $\Gamma_0$-artinian  $R$-module.
        \item We say that $M$ is \emph{$\Gamma_0$-noetherian} if $M(e)$ is a gr-noetherian $R$-module for each $e\in\Gamma_0$. 
        \item The ring $R$ is said to be \emph{right $\Gamma_0$-noetherian} if $R_R$ is a $\Gamma_0$-noetherian  $R$-module.
        \item We say that $M$ is \emph{$\Gamma_0$-finitely generated} if $M(e)$ is a finitely generated $R$-module for each $e\in\Gamma_0$. 
        \item We say that $M$ is \emph{$\Gamma_0$-finitely gr-cogenerated} if $M(e)$ is a finitely gr-cogenerated $R$-module for each $e\in\Gamma_0$.
    \end{enumerate}  
\end{defi}

It is straightforward to check that if $M$ is finitely generated (resp. finitely gr-cogenerated) and $N\subgr M$, then $M/N$ (resp. $N$)  is also finitely generated (resp. finitely gr-cogenerated). 
In particular, every finitely (gr-co)generated module is $\Gamma_0$-finitely (gr-co)generated.

It is easy to show that if $M$ is $\Gamma_0$-artinian (resp. $\Gamma_0$-noetherian), then $M(\sigma)$ is gr-artinian (resp. gr-noetherian) for each $\sigma\in\Gamma$, since $M(\sigma)$ and $M(r(\sigma))$ have the same homogeneous components.

The concepts of $\Gamma_0$-artinian (respectively, $\Gamma_0$-noetherian) are the natural extension to groupoid graded rings of the concept  categorically artinian (categorically noetherian) given in  \cite[Definition~1.1]{Abrams_ArandaPino_Perera_SilesMolina} and \cite[Section~4.2]{Abrams_Ara_SilesMolina}.

Now we characterize gr-noetherian (resp. gr-artinian) modules among the $\Gamma_0$-noetherian ($\Gamma_0$-artinian) modules.

\begin{lema}
\label{lem: finitos e => gr-art/noet = G0-art/noet}
    Let $R$ be a $\Gamma$-graded ring and $M$ be a $\Gamma$-graded right $R$-module. 
    The following assertions are equivalent:
    \begin{enumerate}[\rm (1)]
        \item $M$ is gr-noetherian (resp. gr-artinian).
        \item $M$ is $\Gamma_0$-noetherian (resp. $\Gamma_0$-artinian) and $\Gamma'_0(M)$ is finite.
    \end{enumerate}
\end{lema}

\begin{proof}
	$(1)\implies(2)$: If $M$ is gr-noetherian (resp. gr-artinian), then it follows from \Cref{prop: M art ou noeth <--> N e M/N tbm} that $M(e)$ is gr-noetherian (resp. gr-artinian) for each $e\in\Gamma_0$.
    The finiteness of $\Gamma'_0(M)$ follows from the direct sum decomposition $M=\bigoplus_{e\in\Gamma_0}M(e)$.

    $(2)\implies(1)$: (2) and \Cref{prop: soma finita de art/noeth}(1) imply that $M=\bigoplus_{e\in\Gamma'_0(M)}M(e)$ is gr-noetherian (resp. gr-artinian).
\end{proof}

We have the following characterization of $\Gamma_0$-artinian and $\Gamma_0$-noetherian modules.

\begin{prop}
\label{prop: G0-noeth <-> todo submod eh G0-fg}
    Let $R$ be a $\Gamma$-graded ring and $M$ be a $\Gamma$-graded right $R$-module. The following assertions hold:
    \begin{enumerate}[\rm (1)]
        \item $M$ is $\Gamma_0$-noetherian if and only if  $N$ is $\Gamma_0$-finitely generated for each $N\subgr M$.
        \item $M$ is $\Gamma_0$-artinian if and only if $M/N$ is $\Gamma_0$-finitely gr-cogenerated for each $N\subgr M$.
    \end{enumerate}
\end{prop}

\begin{proof}
	(1)  For each $e\in\Gamma_0$, $M(e)$ is gr-noetherian if and only if all graded submodules of $M(e)$ are finitely generated by \Cref{prop: gr-noeth <-> todo submod eh fg}(1). 
    
    (2) For each $e\in\Gamma_0$, $M(e)$ is gr-artinian if and only if all graded quotients of $M(e)$ are finitely gr-cogenerated by \Cref{prop: gr-noeth <-> todo submod eh fg}(2). 
\end{proof}

Now we show that the $\Gamma_0$-chain conditions are inherited by graded submodules,  quotients and extensions.

\begin{prop}
\label{prop: M G0-art ou G0-noeth <--> N e M/N tbm}
    Let $R$ be a $\Gamma$-graded ring, $M$ be a $\Gamma$-graded right $R$-module and $N$ be a graded submodule of $M$. Then $M$ is $\Gamma_0$-noetherian (resp. $\Gamma_0$-artinian) if and only if both $N$ and $M/N$ are $\Gamma_0$-noetherian (resp. $\Gamma_0$-artinian).
\end{prop}

\begin{proof}
	The result follows from \Cref{prop: M art ou noeth <--> N e M/N tbm} because, for each $e\in\Gamma_0$, $N(e)$ is a graded submodule of $M(e)$ and $M(e)/N(e)\isogr(M/N)(e)$.
\end{proof}

\begin{coro}\label{coro: R G_0 noeth/art => R/J G_0 noeth/art.}
   Let $R$ be a $\Gamma$-graded ring and $J$ be a graded ideal of $R$. If $R$ is a right $\Gamma_0$-noetherian (resp. $\Gamma_0$-artinian) ring, then $R/J$ is a right $\Gamma_0$-noetherian (resp. $\Gamma_0$-artinian) ring.
\end{coro}

\begin{proof}
    By \Cref{prop: M G0-art ou G0-noeth <--> N e M/N tbm}, $R/J$ is a right $\Gamma_0$-noetherian (resp. $\Gamma_0$-artinian) right $R$-module. The result follows because the right $R$-submodules  and the right $R/J$-submodules of $R/J$ coincide.
    \end{proof}

\begin{defi}
 Let $R$ be a $\Gamma$-graded ring. A family $\{M_i:i\in I\}$  of $\Gamma$-graded right $R$-modules is \emph{$\Gamma_0$-finite} if the set
 $I_e:=\{i\in I: M_i(e)\neq0\}$ is finite for all $e\in\Gamma_0$. 
 \end{defi}

\begin{prop}
\label{prop: soma finita de G0 art/noeth}
        Let $R$ be a $\Gamma$-graded ring and $\{M_i:i\in I\}$ be a $\Gamma_0$-finite family of $\Gamma$-graded right $R$-modules.
    \begin{enumerate}[\rm(1)]
        \item If $M_i$ is $\Gamma_0$-noetherian (resp. $\Gamma_0$-artinian) for each $i\in I$, then $\bigoplus\limits_{i\in I}M_i$ is $\Gamma_0$-noetherian (resp. $\Gamma_0$-artinian).
        \item If $M_i$ is a $\Gamma_0$-noetherian (resp. $\Gamma_0$-artinian) submodule of a $\Gamma$-graded module $M$ for each $i\in I$, then $\sum_{i\in I}M_i$ is $\Gamma_0$-noetherian (resp. $\Gamma_0$-artinian).
    \end{enumerate}
    In particular, any finite sum of $\Gamma_0$-noetherian (resp. $\Gamma_0$-artinian) $\Gamma$-graded right $R$-modules is $\Gamma_0$-noetherian (resp. $\Gamma_0$-artinian). 
    \end{prop}

\begin{proof}
    Set $X=\bigoplus\limits_{i \in I}M_i$.
    
(1)  Since $X(e)$ is a finite (direct) sum of gr-noetherian (resp. gr-artinian) $R$-modules  for each $e\in\Gamma_0$, the result follows from \Cref{prop: soma finita de art/noeth}(1).

(2) Consider the natural gr-homomorphism $g\colon X\rightarrow M$. Then $\frac{X}{\ker g}\,\isogr\, \im g=\sum_{i\in I}M_I$. The result follows from (1) and \Cref{prop: M G0-art ou G0-noeth <--> N e M/N tbm}.

The last statement holds because a finite family of $\Gamma$-graded right $R$-modules is $\Gamma_0$-finite. 
\end{proof}

\begin{prop}
\label{prop: R G0-noeth --> todo G0-fg eh G0-noeth}
    Let $R$ be a right $\Gamma_0$-noetherian (resp. right $\Gamma_0$-artinian) ring. 
    and $M$ be a $\Gamma$-graded right $R$-module. The following statements hold:
    \begin{enumerate}[\rm(1)]
        \item If $M$ is finitely generated, then $M$ is gr-noetherian (resp. gr-artinian).
        \item If $M$ is $\Gamma_0$-finitely generated, then $M$ is $\Gamma_0$-noetherian (resp. $\Gamma_0$-artinian). 
    \end{enumerate}
\end{prop}

\begin{proof}
	(1) We can proceed as in the proof of \Cref{coro: R noeth --> todo fg eh noeth}, since any shift of the $\Gamma_0$-noetherian (resp. $\Gamma_0$-artinian) module $R_R$ is gr-noetherian (resp. gr-artinian).
    
    (2) Since $M(e)$ is finitely generated for each $e\in\Gamma_0$, the result follows from (1).
\end{proof}

\begin{coro}
    Let $S$ be a full gr-subring of the $\Gamma$-graded ring $R$. Suppose that $S$ is a right $\Gamma_0$-noetherian (resp. $\Gamma_0$-artinian) ring.
    \begin{enumerate}[\rm (1)]
    \item If $R$ is finitely generated as a right $S$-module, then  $R$ is a right gr-noetherian (resp. gr-artinian) ring. 
        \item If  $R$ is $\Gamma_0$-finitely generated as a right $S$-module, then $R$ is a right $\Gamma_0$-noetherian (resp. $\Gamma_0$-artinian) ring.
    \end{enumerate}
\end{coro}

\begin{proof}
 (1) By \Cref{prop: R G0-noeth --> todo G0-fg eh G0-noeth}(1), $R$ is a right gr-noetherian (resp. gr-artinian) $S$-module. Since all graded right ideals of $R$ are  graded right $S$-modules as well, the result follows.

 (2) By \Cref{prop: R G0-noeth --> todo G0-fg eh G0-noeth}(2), $R$ is a right $\Gamma_0$-noetherian (resp. $\Gamma_0$-artinian) $S$-module. Since all graded right ideals of $R$ are graded right $S$-modules as well, the result follows.
\end{proof}

It is shown in \cite[p.~53]{Artigoarxiv} that being right $\Gamma_0$-artinian implies being locally right gr-artinian in the following sense: if $R=\bigoplus_{\gamma\in\Gamma}R_\gamma$ is a right $\Gamma_0$-artinian ring, then, for any finite subset $\Delta_0\subseteq \Gamma'_0(R)$, the $\Delta$-graded ring $R_\Delta=\bigoplus_{e,f\in\Delta_0}1_eR1_f$ is right gr-artinian.
We prove an analogous version for modules.

\begin{lema}
\label{lem: M G_0-art => 1_eM1_f gr-art}
    Let $R$ be a $\Gamma$-graded ring.
    Let $\Delta_0\subseteq\Gamma_0$ and denote $R_\Delta=\bigoplus\limits_{e,f\in\Delta_0}1_eR1_f$. The following statements hold:
    \begin{enumerate}[\rm (1)]
        \item If $M$ is a  $\Gamma$-graded $\Gamma_0$-artinian (resp. $\Gamma_0$-noetherian) right $R$-module, then $\sum_{e\in\Delta_0}M1_e$ is a $\Gamma_0$-artinian (resp. $\Gamma_0$-noetherian) right $R_\Delta$-module.
        \item  If $R$ is a right $\Gamma_0$-artinian (resp. right $\Gamma_0$-noetherian) ring, then $R_\Delta$ is a right $\Delta_0$-artinian (resp. right $\Delta_0$-noetherian) ring.
    \end{enumerate}
    \end{lema} 

\begin{proof}
    (1) Denote $M':=\sum_{f\in\Delta_0}M1_f$.
    By \Cref{lem: M gr-art => M1_f gr-art}(1), $M'(e)=\sum_{f\in\Delta_0}M(e)1_f$ is a gr-artinian (resp. gr-noetherian) right $R_\Delta$-module for each $e\in\Gamma_0$.
    
    (2) It suffices to note that if $M=R(\Delta_0)=\bigoplus_{e\in\Delta_0}R(e)$, then $M'=R_\Delta$.
\end{proof}

On the other hand, it is not true that if $R$ is locally right gr-artinian in the foregoing sense, then $R$ is right $\Gamma_0$-artinian \cite[p.~53]{Artigoarxiv}. We proceed to show  that being locally right gr-artinian (resp. right gr-noetherian) implies being right $\Gamma_0$-artinian (resp. $\Gamma_0$-noetherian) under some natural conditions.

\begin{lema}
\label{lem: M G_0-art <= M1_f gr-art}
    Let $R$ be a $\Gamma$-graded ring such that $R=\sum_{f\in\Delta_0}R1_fR$ for some finite subset $\Delta_0\subseteq\Gamma_0$. The following statements hold:
    \begin{enumerate}[\rm (1)]
        \item Suppose that $M$ is a $\Gamma$-graded right $R$-module such that $M1_f$ is a $\Gamma_0$-artinian (resp. $\Gamma_0$-noetherian) graded right $1_fR1_f$-module for each $f\in\Delta_0$. 
    Then $M$ is a $\Gamma_0$-artinian (resp. $\Gamma_0$-noetherian) $R$-module.
    \item If $1_eR1_f$ is a gr-artinian (resp. gr-noetherian) graded right $1_fR1_f$-module for each $e\in\Gamma_0$ and $f\in\Delta_0$, then $R$ is a right $\Gamma_0$-artinian (resp. right $\Gamma_0$-noetherian) ring.
    \end{enumerate}
   \end{lema}

\begin{proof}
   (1) Since $M(e)1_f=(M1_f)(e)$ is a gr-artinian (resp. gr-noetherian) graded right $1_fR1_f$-module for each  $e\in\Gamma_0$ and $f\in\Delta_0$, it follows from \Cref{lem: M gr-art <= M1_f gr-art}(1) that $M(e)$ is a gr-artinian (resp. gr-noetherian) $R$-module for each $e\in\Gamma_0$.

(2) This statement follows from \Cref{lem: M gr-art <= M1_f gr-art}(2).
\end{proof}

\begin{defi}
    Following \cite[p. 18-19]{Mit},  a small preadditive category $\mathcal{C}$ is a \emph{right artinian category} (resp. \emph{right noetherian category}) if the functor $\mathcal{C}(-,X)$ is an artinian (resp. noetherian) object in the abelian category $\Fun(\mathcal{C}^{op},\mathcal{A}b)$ of additive contravariant functors for each $X\in\mathcal{C}_0$. 
    In other words, the subobjects of $\mathcal{C}(-,X)$ satisfy the descending (resp. ascending) chain condition.   
\end{defi}

Consider the graded ring $R_\mathcal{C}$ defined in \Cref{ex: anel de categoria}. By \cite[Theorem 8.1(2)]{Artigoarxiv}, $\mathcal{C}$ is a right artinian (resp. noetherian) category if and only if $R_\mathcal{C}$ is a right $\Gamma_0$-artinian (resp. $\Gamma_0$-noetherian) ring, where $\Gamma=\mathcal{C}_0\times\mathcal{C}_0$.

Examples of right artinian (resp. noetherian) categories can be obtained from the following result. Other examples can be found in \Cref{prop: mod-A art e noet}.

\begin{prop}
\label{prop: proj-A art e noet}
    Let $A$ be a unital ring.
    Denote by $\proj$-$A$ the category of finitely generated projective right $A$-modules, let $\mathcal{C}$ be a small full subcategory of $\proj$-$A$ with $A_A\in\mathcal{C}_0$, and set $\Gamma := \mathcal{C}_0 \times \mathcal{C}_0$.
    The following assertions hold:
    \begin{enumerate}[\rm (1)]
        \item $R_\mathcal{C}$ is a right $\Gamma_0$-artinian (resp.  $\Gamma_0$-noetherian) ring if and only if $A$ is a right artinian (resp.  noetherian) ring.
        \item $R_\mathcal{C}$ is a left $\Gamma_0$-artinian (resp.  $\Gamma_0$-noetherian) ring if and only if $A$ is a left artinian (resp.  noetherian) ring.
    \end{enumerate}
\end{prop}

\begin{proof}
    The ``only if'' part is a consequence of \Cref{lem: M G_0-art => 1_eM1_f gr-art}(2) for $\Delta_0=\{(A,A)\}$, since $A\cong\mathcal{C}(A,A)=1_{(A,A)}R_\mathcal{C}1_{(A,A)}$.

    For the ``if part'', we will use \Cref{lem: M G_0-art <= M1_f gr-art} and its left version to prove that if $A$ is right/left artinian (resp. noetherian), then $R_\mathcal{C}$ is a right/left $\Gamma_0$-artinian (resp. $\Gamma_0$-noetherian) ring.

    Set $f:=(A,A)\in\Gamma_0$.   
    Let $(Q,P)\in\Gamma$ and $g\in (R_\mathcal{C})_{(Q,P)}=\mathcal{C}(P,Q)$.
    Since $P_A$ is projective, we can choose $m \in \mathbb{N}$, projections  $\pi_P \in \Hom_A(A^m, P)$ and $\pi_i \in \Hom_A(A^m, A)$, and inclusions $\iota_P \in \Hom_A(P,A^m)$ and $\iota_i \in \Hom_A(A, A^m)$ such that $\id_P=\pi_P\iota_P$ and $\id_{A^m}=\sum_{i=1}^m\iota_i\pi_i$.
    Then, 
    \begin{multline*}
        g=g\id_P=g\pi_P\iota_P=g\pi_P\id_{A^m}\iota_P
    =\sum_{i=1}^mg\pi_P\iota_i\pi_i\iota_P=
    \\=\sum_{i=1}^m(g\pi_P\iota_i)\id_A(\pi_i\iota_P)
    \in R_\mathcal{C}1_f R_\mathcal{C}.
    \end{multline*} 
    Therefore, $R_\mathcal{C}=R_\mathcal{C}1_fR_\mathcal{C}$.

    Now let $P\in\mathcal{C}_0$ and $e:=(P,P)\in\Gamma_0$.
    Then $1_eR_\mathcal{C}1_f=\mathcal{C}(A,P)\cong P_A$ is a finitely generated right $\mathcal{C}(A,A)$-module. 
    Thus, if $A\cong\mathcal{C}(A,A)=1_fR_\mathcal{C}1_f$ is right artinian (resp. noetherian), then $1_eR_\mathcal{C}1_f$ is an artinian (resp. noetherian) right $1_fR_\mathcal{C}1_f$-module.
    On the other hand, $1_fR_\mathcal{C}1_e=\mathcal{C}(P,A)$ is a finitely generated left $A$-module by \cite[Remark 2.11]{Lam2}. Thus, if $A\cong\mathcal{C}(A,A)=1_fR_\mathcal{C}1_f$ is left artinian (resp. noetherian), then $1_fR_\mathcal{C}1_e$ is an artinian (resp. noetherian) left $1_fR_\mathcal{C}1_f$-module. 
\end{proof}

We proceed to state an interesting property of modules satisfying a $\Gamma_0$-chain condition.

\begin{prop}
\label{prop: pre-Fitting para G0}
    Let $R$ be a $\Gamma$-graded ring, $M$ be a $\Gamma$-graded right $R$-module and $g\colon M\rightarrow M$ be a gr-homomorphism. The following statements hold:
    \begin{enumerate}[\rm (1)]
        \item Suppose that $M$ is $\Gamma_0$-noetherian. 
        Then, for each $e\in\Gamma_0$, there exists $n_e\in\mathbb{N}$ such that $g(M(e)\cap\ker g^{n_e})\subseteq\ker g^{n_e}$ and $M(e)\cap\ker g^{n_e}\cap\im g^{n_e}=0$.
        Moreover, if $g$ is surjective, then $g$ is a gr-isomorphism.
        \item Suppose that $M$ is $\Gamma_0$-artinian. 
        Then, for each $e\in\Gamma_0$, there exists $n_e\in\mathbb{N}$ such that $g(\im g^{n_e})\supseteq M(e)\cap\im g^{n_e}$ and $(M(e)\cap\ker g^{n_e})+(M(e)\cap\im g^{n_e})=M(e)$.
        Moreover, if $g$ is injective, then $g$ is a gr-isomorphism.
    \end{enumerate}
\end{prop}

\begin{proof}
    It suffices to apply \Cref{lem: Fitting para gr-end} to $g|_{M(e)}\in \Endgr(M(e))$ for each $e\in\Gamma_0$.
\end{proof}

Now we deal with the concept of $\Gamma_0$-finite gr-length and its properties.

\begin{defi}
 Let $R$ be a $\Gamma$-graded ring and $M$ be a $\Gamma$-graded right $R$-module.   If $M(e)$ has finite gr-length for each $e\in\Gamma_0$, then we say that $M$ has \emph{$\Gamma_0$-finite gr-length}.
\end{defi}

Note  that if $R$ is a gr-semisimple ring, then for each $e\in\Gamma'_0(R)$, $R(e)$ is a finite direct sum of gr-simple modules \cite[Lemma 5.12(2)]{Artigoarxiv}, and thus $R(e)$ has an obvious gr-composition series, so that $R_R$ has $\Gamma_0$-finite gr-length.
More generally, any $\Gamma_0$-finitely generated gr-semisimple module has $\Gamma_0$-finite gr-length.

\begin{prop}
\label{prop: G0 comp finito = G0 art + G0 noet}
    Let $R$ be a $\Gamma$-graded ring and $M$ be a $\Gamma$-graded right $R$-module. Then $M$ has {$\Gamma_0$-finite gr-length} if and only if $M$ is both $\Gamma_0$-artinian and $\Gamma_0$-noetherian. 
\end{prop}

\begin{proof}
   By \Cref{prop: comp finito = art + noet}, $M(e)$ has finite gr-length if and only if $M(e)$ is gr-artinian and gr-noetherian.
\end{proof}

\begin{defi}
Let $R$ be a $\Gamma$-graded ring and $M$ be a $\Gamma$-graded right $R$-module of {$\Gamma_0$-finite gr-length}. We define  $c_{\rm \Gamma_0-gr}(M)$, the   \emph{$\Gamma_0$-gr-length} of $M$,  as the sequence $c_{\rm \Gamma_0-gr}(M)=\big(c_{\rm gr}(M(e))\big)_{e\in\Gamma_0}\in \mathbb{N}^{\Gamma_0}$. 
\end{defi}

\begin{prop}
 Let $R$ be a $\Gamma$-graded ring, $M$ be a $\Gamma$-graded right $R$-module and $N\subgr M$. Then $M$   has $\Gamma_0$-finite gr-length if and only if $N$ and $M/N$ have $\Gamma_0$-finite gr-length. In this event, $c_{\rm \Gamma_0-gr}(M)=c_{\rm \Gamma_0-gr}(N)+c_{\rm \Gamma_0-gr}(M/N)$.
\end{prop}
\begin{proof}
 The first statement follows from \Cref{prop: G0 comp finito = G0 art + G0 noet}  and \Cref{prop: M G0-art ou G0-noeth <--> N e M/N tbm}.
 The second statement follows using \Cref{prop: comprimento finito M <---> comp finito N  e M/N} applied to $M(e)$, $N(e)$ and $(M/N)(e)$ for each $e\in\Gamma_0$:
 \begin{align*}
    c_{\rm \Gamma_0-gr}(M)&=\big(c_{\rm gr}(M(e))\big)_{e\in\Gamma_0}\\ 
   &=\big(c_{\rm gr}(N(e)+c_{\rm gr}(M(e)/N(e))\big)_{e\in\Gamma_0} \\ 
   &=\big(c_{\rm gr}(N(e))\big)_{e\in\Gamma_0}+\big(c_{\rm gr}(M(e)/N(e))\big)_{e\in\Gamma_0}\\ 
   &=c_{\rm \Gamma_0-gr}(N)+c_{\rm \Gamma_0-gr}(M/N).\qedhere
\end{align*}
\end{proof}

In the following result, we show that  any skeleton $\mathcal{C}$ of the category of finitely generated modules over a representation-finite $k$-algebra is artinian and noetherian (on both sides). Equivalently, that the ring $R_\mathcal{C}$ has $\Gamma_0$-finite gr-length as a left and as a right graded  $R_\mathcal{C}$-module 

\begin{prop}
\label{prop: mod-A art e noet}
    Let $k$ be a field, $A$ be a finite dimensional $k$-algebra, and \mbox{$\fgmod$-$A$} be the category of all finitely generated right $A$-modules.
    \begin{enumerate}[\rm (1)]
        \item Suppose that $\mathcal{C}$ is a small full subcategory of $\fgmod$-$A$ and there exists a finite subset 
    $\mathcal{X}\subseteq \mathcal{C}_0$ such that each object of $\mathcal{C}$ is isomorphic to a finite direct sum of elements of $\mathcal{X}$.
    Then $R_\mathcal{C}$ is both a left and a right $\Gamma_0$-artinian ring, and both a left and a right $\Gamma_0$-noetherian ring, where $\Gamma := \mathcal{C}_0 \times \mathcal{C}_0$.

    \item If $\mathcal{C}$ is a skeleton of $\fgmod$-$A$ and $A$ is representation-finite, that is, there exist finite isoclasses of indecomposable $A$-modules, then $R_\mathcal{C}$ is both a left and a right $\Gamma_0$-artinian ring, and both a left and a right $\Gamma_0$-noetherian ring.
    \end{enumerate}
\end{prop}

\begin{proof}
    Set  $\Delta_0:=\{(X,X):X\in\mathcal{X}\}\subseteq\Gamma_0$.

    Let $M,N\in\mathcal{C}_0$ and $g\in(R_\mathcal{C})_{(N,M)}=\mathcal{C}(M,N)$.
    Take $X_1,X_2,\dots,X_n\in\mathcal{X}$ such that $M\cong X_1\oplus X_2\oplus\dots\oplus X_n$.
    Then, there exist $p_i\in\mathcal{C}(M,X_i)$ and $q_i\in\mathcal{C}(X_i,M)$ for each $1\leq i \leq n$ such that $\sum_{i=1}^nq_ip_i=\id_M$.
    It follows that
    \[g=g\cdot\id_M=\sum_{i=1}^ngq_ip_i=\sum_{i=1}^ngq_i(\id_{X_i})p_i\in\sum_{f\in\Delta_0}R_\mathcal{C}1_fR_\mathcal{C}.\]
    Therefore, $R_\mathcal{C}=\sum_{f\in\Delta_0}R_\mathcal{C}1_fR_\mathcal{C}$.

    Let $X\in\mathcal{X}$ and $f=(X,X)\in\Delta_0$. 
    Since $X$ is a finite dimensional $k$-vector space, it follows that $1_fR_\mathcal{C}1_f=\mathcal{C}(X,X)$ is a finite dimensional $k$-algebra, and hence a left and right artinian (and noetherian) ring. 
    For each $M\in\mathcal{C}_0$, setting $e=(M,M)\in\Gamma_0$, we have that $1_eR_\mathcal{C}1_f=\mathcal{C}(X,M)$ and $1_fR_\mathcal{C}1_e=\mathcal{C}(M,X)$ are finite dimensional $k$-vector spaces, and thus they are finitely generated $1_fR_\mathcal{C}1_f$-modules. 
    Therefore, $1_eR_\mathcal{C}1_f$ and $1_fR_\mathcal{C}1_e$ are artinian and noetherian $1_fR_\mathcal{C}1_f$-modules for all $f\in\Delta_0$ and $e\in\Gamma_0$.

    It follows from \Cref{lem: M G_0-art <= M1_f gr-art}(2) and its left version that $R_\mathcal{C}$ is both left and right $\Gamma_0$-artinian, and both left and right $\Gamma_0$-noetherian.
\end{proof}

We end this subsection with a version of Fitting's Lemma for gr-endomorphisms of graded modules of $\Gamma_0$-finite length.

\begin{prop}
    Let $R$ be a $\Gamma$-graded ring and $M$ a $\Gamma$-graded right $R$-module with $\Gamma_0$-finite gr-length. 
    For any gr-endomorphism $g\in\Endgr(M)$, there exists a sequence of positive integers $(n_e)_{e\in\Gamma_0}$ such that the graded submodules \linebreak $K=\bigoplus_{e\in\Gamma_0}\ker(g^{n_e}|_{M(e)})$ and $V=\bigoplus_{e\in\Gamma_0}\im (g^{n_e}|_{M(e)})$ satisfy $M=K\oplus V$, the restriction $g|_V\colon V\rightarrow V$ is a gr-isomorphism, and  $g(K)\subseteq K$.
    Moreover, if $g$ is injective or surjective, then $g$ is a gr-isomorphism.
\end{prop}

\begin{proof}
 For each $e\in\Gamma_0$, $M(e)$ is a graded module of finite gr-length. 
 Therefore, we can apply \Cref{prop: Fitting Lemma for graded modules of fin. gr-length}(1) to the gr-endomorphism $g|_{M(e)}\colon M(e)\rightarrow M(e)$ for each $e\in\Gamma_0$.
\end{proof}


\subsection{Strong $\Gamma_0$-chain conditions}

In this section, we consider the last type of chain conditions we will be dealing with in our work.
For that, we begin defining the kind of chains that these chain conditions have to do with.

\begin{defi}
Let $R$ be a $\Gamma$-graded ring and $M$ be a $\Gamma$-graded right $R$-module.

\begin{enumerate}[\rm (1)]

\item A descending chain of graded submodules of $M$
$$ M_1 \supseteq M_2 \supseteq \cdots \supseteq M_n \supseteq \cdots $$
is said to be \emph{tight} if
$$ \Gamma'_0 \left( \dfrac{M_i}{M_{i+1}} \right) \supseteq \Gamma'_0 \left( \dfrac{M_{i+1}}{M_{i+2}} \right), \quad \text{for all } i \ge 1. $$

\item An ascending chain of graded submodules of $M$
$$ M_1 \subseteq M_2 \subseteq \cdots \subseteq M_n \subseteq \cdots $$
is said to be \emph{tight} if
$$ \Gamma'_0 \left( \dfrac{M_{i+1}}{M_i} \right) \supseteq \Gamma'_0 \left( \dfrac{M_{i+2}}{M_{i+1}} \right), \quad \text{for all } i \ge 1. $$

\end{enumerate}

\end{defi}

 Note that if $\{M_i\}_{i \geq 1}$ is a chain of graded submodules of $M$, then this chain is tight if and only if, for each $e \in \Gamma_0$ and $i \geq 1$, the following implication holds:
$$
M_i(e) = M_{i+1}(e) \implies M_i(e) = M_{i+l}(e) \text{ for all } l \ge 1.
$$

Some important examples of tight chains are presented next.

\begin{exems}
Let $R$ be a $\Gamma$-graded ring, $J$ a  graded right ideal of $R$, $M$ a $\Gamma$-graded right $R$-module and $g\colon M\rightarrow M$ be a gr-homomorphism. 
\begin{enumerate}[(1)]
    \item $R\supseteq J\supseteq J^2\supseteq \dotsb \supseteq J^n\supseteq \dotsb$ is a tight descending chain of $R.$
    \item More generaly, $M\supseteq MJ\supseteq MJ^2\supseteq\dotsb \supseteq MJ^n\supseteq\dotsb $ is a tight descending chain of $M.$
    \item $\im g\supseteq\im g^2\supseteq\dotsb \supseteq\im g^n\supseteq \dotsb $ is a tight descending chain of $M$.
    \item $\ker g\subseteq \ker g^2\subseteq \dotsb\subseteq\ker g^n\subseteq \dotsb$ is a tight ascending chain of $M$.  
\end{enumerate}
\end{exems}

We are ready to define the strong $\Gamma_0$-chain conditions.

\begin{defi} 
Let $R$ be a $\Gamma$-graded ring.
We say that a $\Gamma$-graded right $R$-module $M$ is
\begin{enumerate}[\rm (1)]

\item  \emph{strongly $\Gamma_0$-artinian} if every tight descending chain of graded submodules of $M$ stabilizes.

\item \emph{strongly $\Gamma_0$-noetherian} if every tight ascending chain of graded submodules of $M$ stabilizes.

\end{enumerate}

\end{defi}

In the next lemma, we relate all the chain conditions we have defined.

\begin{lema}
\label{lem: fort G0-art/noet => G0-art/noet}
Let $R$ be a $\Gamma$-graded ring and $M$ be a $\Gamma$-graded right $R$-module.
The following implications hold true:
\begin{gather*}
    M \text{ is gr-artinian }\phantom{x}\implies \phantom{x} M \text{ is strongly $\Gamma_0$-artinian }\phantom{x}\implies\phantom{x} M \text{ is $\Gamma_0$-artinian };\\
    M \text{ is gr-noetherian }\implies M \text{ is strongly $\Gamma_0$-noetherian }\implies M \text{ is $\Gamma_0$-noetherian }.
\end{gather*}

\end{lema}

\begin{proof} 
Clearly, if $M$ is gr-artinian (resp. gr-noetherian), then $M$ is strongly $\Gamma_0$-artinian (resp. strongly $\Gamma_0$-noetherian).

To see the other implication, it suffices to note that, for each $e\in\Gamma_0$, every infinite strictly descending (resp. ascending) chain $\{M_i\}_{i\geq 1}$ of graded submodules of $M(e)$ is tight, with $\Gamma'_0 \left( {M_i}/{M_{i+1}} \right)=\{e\}$ (resp. $\Gamma'_0 \left( {M_{i+1}}/{M_i} \right)=\{e\}$) for all $i\geq 1$.
\end{proof}

The following useful result characterizes the strong $\Gamma_0$-conditions.

\begin{prop}  
\label{prop: carac fort G0-art/noet} 

Let $R$ be a $\Gamma$-graded ring and $M$ be a $\Gamma$-graded right $R$-module.
The following assertions are equivalent:

\begin{enumerate}[\rm (1)] 

\item $M$ is strongly $\Gamma_0$-artinian (resp. strongly $\Gamma_0$-noetherian).
\item $M$ is $\Gamma_0$-artinian (resp. $\Gamma_0$-noetherian) and there exist $\Delta_0 \subseteq \Gamma_0$ and $n_0 \in \mathbb{N}$ such that $\Gamma_0 \setminus \Delta_0$ is finite and  $M(e)$ is of finite gr-length not greater than $n_0$ for each $e \in \Delta_0$.

\end{enumerate}
\end{prop}

\begin{proof} 
We will prove the artinian case. The noetherian case is handled in a completely analogous way.

$ (1) \Rightarrow (2) $: 
By \Cref{lem: fort G0-art/noet => G0-art/noet}, if $M$ is strongly $\Gamma_0$-artinian, then $M$ is $\Gamma_0$-artinian.
Suppose that $M$ is strongly $\Gamma_0$-artinian but (2) does not hold. Then we can inductively find pairwise distinct idempotents $e_1, e_2, \hdots, e_n, \hdots \in \Gamma_0$ such that, for each $n > 1$, $e_n \in \Gamma_0 \setminus \{e_1, \dots, e_{n-1}\}$ is such that $M(e_n)$ contains a strictly descending chain $M_{n1} \supsetneq M_{n2} \supsetneq \cdots \supsetneq M_{nn}$.
For each $k \in \mathbb{N}$, set $M_k := \bigoplus\limits_{n = k}^{\infty} M_{nk}$. 
Note that
$$ M_1 \supsetneq M_2 \supsetneq \cdots \supsetneq M_n \supsetneq \cdots $$
and that $\{ e_j : j \geq n \} \supseteq \Gamma'_0 \left( \frac{M_n}{M_{n+1}} \right) \supseteq \{ e_j : j > n \}$ for all $n$, which shows that the chain $ M_1 \supsetneq M_2 \supsetneq \cdots \supsetneq M_n \supsetneq \cdots $ is tight but does not stabilize, a contradiction.

$(2) \Rightarrow (1)$: Suppose that $\Delta_0$ and $n_0$ are as in (2). 
Let $M_1 \supseteq M_2 \supseteq \cdots \supseteq M_n \supseteq \cdots$ be a tight descending chain of graded submodules of $M$.
For each $e \in \Delta_0$, $M(e)$ does not contain a strictly descending chain of graded submodules of length greater than $n_0$ and, therefore, since $M_1 \supseteq M_2 \supseteq \cdots \supseteq M_n \supseteq \cdots$ is tight, it follows that $M_{n_0}(e) = M_{n_0+ l}(e)$ for all $l \geq 0$.
Consider $\Gamma_0 \setminus \Delta_0 = \{e_1, \dots, e_k\}$. 
Since $M$ is $\Gamma_0$-artinian, we have that, for each $i = 1, \dots, k$, the chain $M_1(e_i) \supseteq M_2(e_i) \supseteq \hdots \supseteq M_n (e_i) \supseteq \hdots$ stabilizes at some $n_i \in \mathbb{N}$. 
Then $n := \max \{ n_0, n_1, \dots, n_r \}$ is such that $M_n = M_{n+l}$ for all $l \geq 0$, which shows that $M$ is strongly $\Gamma_0$-artinian.
\end{proof}

Now we characterize gr-noetherian (resp. gr-artinian) modules among the strongly $\Gamma_0$-noetherian ($\Gamma_0$-artinian) modules.

\begin{lema}
\label{lem: finitos e => gr-art/noet = strongly G0-art/noet}
    Let $R$ be a $\Gamma$-graded ring and $M$ be a $\Gamma$-graded right $R$-module. 
    The following assertions are equivalent:
    \begin{enumerate}[\rm (1)]
        \item $M$ is gr-noetherian (resp. gr-artinian).
        \item $M$ is strongly $\Gamma_0$-noetherian (resp. strongly $\Gamma_0$-artinian) and $\Gamma'_0(M)$ is finite.
    \end{enumerate}
\end{lema}

\begin{proof}
	$(1)\implies(2)$: By \Cref{lem: finitos e => gr-art/noet =  G0-art/noet}, $M$ is $\Gamma_0$-noetherian (resp. $\Gamma_0$-artinian) and $\Gamma'_0(M)$ is finite. By \Cref{prop: carac fort G0-art/noet}, $M$ is strongly $\Gamma_0$-noetherian (resp. strongly $\Gamma_0$-artinian).

    $(2)\implies(1)$: By \Cref{prop: carac fort G0-art/noet}, $M$ is $\Gamma_0$-noetherian (resp. $\Gamma_0$-artinian). The result then follows from \Cref{lem: finitos e => gr-art/noet = G0-art/noet}.
\end{proof}

In \Cref{exs: Chain conditions},  we show that the implications in \Cref{lem: fort G0-art/noet => G0-art/noet} are not reversible in general. On the other hand, observe that if $\Gamma'_0(M)$ is finite, then all implications in \Cref{lem: fort G0-art/noet => G0-art/noet} are reversible by \Cref{lem: finitos e => gr-art/noet = G0-art/noet}.

\begin{exems}\label{exs: Chain conditions}
 (1) $M$ being strongly $\Gamma_0$-artinian (resp. strongly $\Gamma_0$-noetherian) does not imply that $M$ is gr-artinian (resp. gr-noetherian).
    For example, let $D$ be a division ring, $\Gamma:=\mathbb{N}\times\mathbb{N}$, and let $R:=D^{(\mathbb{N})}$ with the natural $\Gamma$-grading where $\supp(R)=\Gamma_0$.
    Note that $R_R$ is $\Gamma_0$-artinian and, for each $n\geq1$, $e_n=(n,n)\in\Gamma_0$ is such that $R(e_n)=D$ is gr-simple. 
    Therefore, $R_R$ is both strongly $\Gamma_0$-artinian and strongly $\Gamma_0$-noetherian by \Cref{prop: carac fort G0-art/noet} (taking $\Delta_0=\Gamma_0$ and $n_0=1$).
    However, $R_R$ is clearly neither gr-artinian nor gr-noetherian.

    (2) $M$ being $\Gamma_0$-artinian (resp. $\Gamma_0$-noetherian) does not imply that $M$ is strongly $\Gamma_0$-artinian (resp. strongly $\Gamma_0$-noetherian). 
    In fact, let $\{D_n:n\in\mathbb{N}\}$ be a family of division rings, $\Gamma:=\mathbb{N}\times\mathbb{N}$, and let $R:=\bigoplus_{n\in\mathbb{N}}D_n^n$ with $\Gamma$-grading given by $R_{(n,n)}=D_n^n$ for each $n\in\mathbb{N}$. 
    Then, for each $e=(n,n)$ where $n\in\mathbb{N}$, $R(e)=D_n^n$ is both gr-artinian and gr-noetherian. 
    However, $R_R$ does not satisfy \Cref{prop: carac fort G0-art/noet}(2).  
\end{exems}

Now we show that the strong $\Gamma_0$-chain conditions are inherited by graded submodules,  quotients and extensions.

\begin{prop}
\label{prop: M fortemente G0-art ou G0-noeth <--> N e M/N tbm}
    Let $R$ be a $\Gamma$-graded ring, $M$ be a $\Gamma$-graded right $R$-module and $N$ be a graded submodule of $M$. Then $M$ is strongly $\Gamma_0$-noetherian (resp. strongly $\Gamma_0$-artinian) if and only if both $N$ and $M/N$ are strongly $\Gamma_0$-noetherian (resp. strongly $\Gamma_0$-artinian).
\end{prop}

\begin{proof}
	Suppose that $M$ is strongly $\Gamma_0$-noetherian (resp. strongly $\Gamma_0$-artinian).
    By \Cref{prop: carac fort G0-art/noet}, $M$ is $\Gamma_0$-noetherian (resp. $\Gamma_0$-artinian) and there exist $\Delta_0\subseteq \Gamma_0$ and $n_0\in \mathbb{N}$ such that $\Gamma_0\setminus\Delta_0$ is finite and $M(e)$ is of gr-length at most $n_0$ for all $e\in\Delta_0$. By \Cref{prop: M G0-art ou G0-noeth <--> N e M/N tbm}, $N$ and $M/N$ are $\Gamma_0$-noetherian (resp. $\Gamma_0$-artinian).  Moreover,  by \Cref{prop: comprimento finito M <---> comp finito N  e M/N}, $c_{\rm gr}(M(e))=c_{\rm gr}(N(e))+c_{\rm gr}(M(e)/N(e))$ for each $e\in \Delta_0$. Thus $N(e)$ and $(M/N)(e)$ are of finite length at most $n_0$. 

    Conversely, suppose that both $N$ and $M/N$ are strongly $\Gamma_0$-noetherian (resp. strongly $\Gamma_0$-artinian). By \Cref{prop: carac fort G0-art/noet}, $N$ and $M/N$ are $\Gamma_0$-noetherian (resp. $\Gamma_0$-artinian) and there exist $\Delta_0',\Delta_0''\subseteq \Gamma_0$ and $n',n''\in \mathbb{N}$ such that $\Gamma_0\setminus\Delta_0'$ and $\Gamma_0\setminus\Delta_0''$ are finite, $N(e')$ is of gr-length at most $n'$ and $M(e'')/N(e'')$ is of gr-length at most $n''$ for all $e'\in\Delta_0'$ and $e''\in\Delta_0''$. By \Cref{prop: M G0-art ou G0-noeth <--> N e M/N tbm}, $M$ is $\Gamma_0$-noetherian (resp. $\Gamma_0$-artinian). Set $\Delta_0=\Delta_0'\cap \Delta_0''$ and $n_0=n'+n''$. Then $\Gamma_0\setminus\Delta_0$ is finite and $c_{\rm gr}(M(e))=c_{\rm gr}(N(e))+c_{\rm gr}(M(e)/N(e))\leq n_0$ for all $e\in\Delta_0$. The result now follows from  \Cref{prop: carac fort G0-art/noet}.
\end{proof}

\begin{coro}\label{coro: R strongly G_0 noeth/art => R/J strongly G_0 noeth/art.}
   Let $R$ be a $\Gamma$-graded ring and $J$ be a graded ideal of $R$. If $R$ is a right strongly $\Gamma_0$-noetherian (resp. $\Gamma_0$-artinian) ring, then $R/J$ is a right strongly $\Gamma_0$-noetherian (resp. $\Gamma_0$-artinian) ring.
\end{coro}

\begin{proof}
    By \Cref{prop: M fortemente G0-art ou G0-noeth <--> N e M/N tbm}, $R/J$ is a right strongly $\Gamma_0$-noetherian (resp.  strongly $\Gamma_0$-artinian) right $R$-module. The result follows because the graded right $R$-submodules  and the graded right $R/J$-submodules of $R/J$ coincide.
\end{proof}

\begin{defi}
 Let $R$ be a $\Gamma$-graded ring. A family $\{M_i:i\in I\}$  of $\Gamma$-graded right $R$-modules is \emph{strongly $\Gamma_0$-finite} if there exists a positive integer $r$ such that the cardinality of the set \[I_e=\{i\in I\colon M_i(e)\neq\{0\}\,\}\] is $\leq r$ for all $e\in\Gamma_0.$ 
\end{defi}
Notice that a strongly $\Gamma_0$-finite family of $\Gamma$-graded right $R$-modules is $\Gamma_0$-finite. We also observe that a finite family of $\Gamma$-graded right $R$-modules is strongly $\Gamma_0$-finite.

\begin{prop}
\label{prop: soma finita de fortemente G0 art/noeth}
        Let $R$ be a $\Gamma$-graded ring and $\{M_i:i\in I\}$ be a strongly $\Gamma_0$-finite family of $\Gamma$-graded right $R$-modules such that there exists a positive integer $s$ for which       
            the set \[\bigcup\limits_{i\in I}\{e\in\Gamma_0\colon M_i(e) \textrm{ is not of finite gr-length}\leq s \,\}\] is finite.
        
    \begin{enumerate}[\rm(1)]
        \item If $M_i$ is strongly $\Gamma_0$-noetherian (resp. strongly $\Gamma_0$-artinian) for each $i\in I$, then $\bigoplus\limits_{i \in I}M_i$ is  strongly $\Gamma_0$-noetherian (resp. strongly $\Gamma_0$-artinian).
        \item If $M_i$ is a strongly $\Gamma_0$-noetherian (resp. strongly $\Gamma_0$-artinian) submodule of a $\Gamma$-graded module $M$ for each $i\in I$, then $\sum_{i\in I}M_i$ is  strongly $\Gamma_0$-noetherian (resp. strongly $\Gamma_0$-artinian).
    \end{enumerate}
    In particular, any finite sum of strongly $\Gamma_0$-noetherian (resp. strongly $\Gamma_0$-artinian) $\Gamma$-graded right $R$-modules is strongly $\Gamma_0$-noetherian (resp. strongly $\Gamma_0$-artinian). 
\end{prop}

\begin{proof}
 Set $X=\bigoplus\limits_{i\in I}M_i$. 

(1) 	By \Cref{prop: carac fort G0-art/noet}, each $M_i$ is $\Gamma_0$-noetherian (resp. $\Gamma_0$-artinian). Thus, by \Cref{prop: soma finita de G0 art/noeth}(1), $X$ is $\Gamma_0$-noetherian (resp. $\Gamma_0$-artinian). 

Since $\{M_i:i\in I\}$ is a strongly $\Gamma_0$-finite family,  there exists a positive integer $r$ such that $X(e)$ is a finite direct sum of at most $r$ nonzero $M_i(e)$'s for all $e\in\Gamma_0$. Moreover, since $$\bigcup\limits_{i \in I}\{e\in\Gamma_0\colon M_i(e) \textrm{ is not of finite gr-length}\leq s\,\}$$ is finite, $X(e)$ is of finite gr-length $\leq rs$ for all but a finite number of $e\in \Gamma_0$. The result now follows from \Cref{prop: carac fort G0-art/noet}.

(2) Consider the natural gr-homomorphism $g\colon X\rightarrow M$. Then $\frac{X}{\ker g}\,\isogr\, \im g=\sum_{i\in I}M_i$. The result follows from (1) and \Cref{prop: M fortemente G0-art ou G0-noeth <--> N e M/N tbm}.

The last statement follows from (1) and (2) because any finite family of strongly $\Gamma_0$-noetherian (resp. strongly $\Gamma_0$-artinian) modules satisfies the hypothesis. 
\end{proof}

\begin{prop}
\label{prop: R strongly G0-noeth --> todo G0-fg eh G0-noeth}
    Let $R$ be a right strongly $\Gamma_0$-noetherian (resp. strongly $\Gamma_0$-artinian) ring  and $M$ be a $\Gamma$-graded right $R$-module. The following statements hold:
    \begin{enumerate}[\rm(1)]
        \item If $M$ is finitely generated, then $M$ is gr-noetherian (resp. gr-artinian).
        \item If $M$ is $\Gamma_0$-finitely generated, then $M$ is $\Gamma_0$-noetherian (resp. $\Gamma_0$-artinian). 
        \item Suppose that $M$ is generated by a set $(x_i)_{i\in I}\in\prod\limits_{i\in I}M_{\gamma_i}$ for some $(\gamma_i)_{i\in I}\in\Gamma^I$. If there exists a positive integer $r$ such that the set $I_e=\{i\in I\colon r(\gamma_i)=e\}$ is of cardinality $\leq r$ for each $e\in\Gamma_0$, and the set $\{i\in I\colon R(d(\gamma_i)) \textrm{ is not of finite gr-length\,}\}$ is finite, 
        then $M$ is strongly $\Gamma_0$-noetherian (resp. strongly $\Gamma_0$-artinian).
        
    \end{enumerate}
\end{prop}

\begin{proof}
	(1) and (2) follow from \Cref{prop: R G0-noeth --> todo G0-fg eh G0-noeth} because $R$ is  right $\Gamma_0$-noetherian (resp. $\Gamma_0$-artinian) by \Cref{prop: carac fort G0-art/noet}.

    (3) We show that the family $\{x_iR\colon i\in I\}$ satisfies the conditions of \Cref{prop: soma finita de fortemente G0 art/noeth}.  
The family is strongly $\Gamma_0$-finite because the set $I_e=\{i\in I\colon r(\gamma_i)=e\}$ is of cardinality $\leq r$ for each $e\in\Gamma_0$. 
    
 Since $R$ is strongly $\Gamma_0$-noetherian (resp. strongly $\Gamma_0$-artinian) there exist $\Delta_0\subseteq\Gamma_0$ and a positive integer $n_0$ such that $\Gamma_0\setminus\Delta_0$ is finite and $R(e)$ is of finite gr-length  $\leq n_0$ for all $e\in\Delta_0$. 
The fact that the sets $\Gamma_0\setminus\Delta_0$ and $I':=\{i\in I\colon R(d(\gamma_i)) \textrm{ is not of finite gr-length\,}\}$ are finite, ensures the existence of a positive integer $s\geq n_0$ such that $R(d(\gamma_i))$ is of finite gr-length $\leq s$ for each $i\in I\setminus I'$.
Moreover, there exists a  surjective gr-homomorphism $R(\gamma_i^{-1})\rightarrow x_iR$ determined by  $1_{d(\gamma_i)}\mapsto x_i$. 
Hence, if $i \in I\setminus I'$, then $x_iR$ is of finite gr-length $\leq c_{\rm gr}(R(\gamma_i^{-1}))=c_{\rm gr}(R(d(\gamma_i)))\leq s$. 
Therefore, the set \[\bigcup\limits_{i\in I}\{e\in\Gamma_0\colon (x_iR)(e) \textrm{ is not of finite gr-length}\leq s \,\}\]
is finite, since it is contained in $I'$. 
\end{proof}

\begin{coro}
    Let $S$ be a full gr-subring of the $\Gamma$-graded ring $R$. Suppose that $S$ is a right strongly $\Gamma_0$-noetherian (resp. strongly $\Gamma_0$-artinian) ring.
    \begin{enumerate}[\rm (1)]
    \item If $R$ is finitely generated as a right $S$-module, then  $R$ is a right gr-noetherian (resp. gr-artinian) ring. 
        \item If  $R$ is $\Gamma_0$-finitely generated as a right $S$-module, then $R$ is a right $\Gamma_0$-noetherian (resp. $\Gamma_0$-artinian) ring.

        \item Suppose that $R$ is generated as an $S$-module by a set $(x_i)_{i\in I}\in\prod\limits_{i\in I}R_{\gamma_i}$ for some $(\gamma_i)_{i\in I}\in\Gamma^I$.  If there exists a positive integer $r$ such that the set $I_e=\{i\in I\colon r(\gamma_i)=e\}$ is of cardinality $\leq r$ for each $e\in\Gamma_0$, and the set $\{i\in I\colon S(d(\gamma_i)) \textrm{ is not of finite gr-length\,}\}$ is finite, then $R$ is a right strongly $\Gamma_0$-noetherian (resp. strongly $\Gamma_0$-artinian) ring. 
           \end{enumerate}
\end{coro}

\begin{proof}
 (1) By \Cref{prop: R strongly G0-noeth --> todo G0-fg eh G0-noeth}(1), $R$ is a right gr-noetherian (resp. gr-artinian) $S$-module. Since all graded right ideals of $R$ are  graded right $S$-modules as well, the result follows.

 (2) By \Cref{prop: R strongly G0-noeth --> todo G0-fg eh G0-noeth}(2), $R$ is a right $\Gamma_0$-noetherian (resp. $\Gamma_0$-artinian) $S$-module. Since all graded right ideals of $R$ are graded right $S$-modules as well, the result follows.

 (3) By \Cref{prop: R strongly G0-noeth --> todo G0-fg eh G0-noeth}(3), $R$ is a right strongly $\Gamma_0$-noetherian (resp. strongly $\Gamma_0$-artinian) $S$-module. Since all graded right ideals of $R$ are graded right $S$-modules as well, the result follows.
\end{proof} 

We now state an interesting property of modules satisfying a strong $\Gamma_0$-chain condition. Compare with \Cref{lem: Fitting para gr-end} and \Cref{prop: pre-Fitting para G0}.

\begin{lema}
\label{lem: Fitting para gr-end caso fort}
    Let $R$ be a $\Gamma$-graded ring, $M$ a $\Gamma$-graded right $R$-module, and $g\in \Endgr(M)$. 
    The following assertions hold:
    \begin{enumerate}[\rm (1)]
        \item Suppose that $M$ is strongly $\Gamma_0$-noetherian. 
        Then there exists $n\in\mathbb{N}$ such that $g(\ker g^n)\subseteq\ker g^n$ and $\ker g^n\cap\im g^n=0$.
        Moreover, if $g$ is surjective, then $g$ is a gr-isomorphism.
        \item Suppose that $M$ is strongly $\Gamma_0$-artinian. 
        Then there exists $n\in\mathbb{N}$ such that $g(\im g^n)\supseteq \im g^n$ and $\ker g^n+\im g^n=M$.
        Moreover, if $g$ is injective, then $g$ is a gr-isomorphism.
    \end{enumerate}
\end{lema}
	
\begin{proof}
    The following two chains of graded submodules of $M$ are tight:
\[0\subseteq \ker g\subseteq \ker g^2\subseteq\dotsb \quad  \textrm{ and } \quad M\supseteq \im g\supseteq \im g^2\supseteq \dotsb\]
Therefore, we can proceed as in the proof of \Cref{lem: Fitting para gr-end}.
\end{proof}

Now we turn our attention to the strong $\Gamma_0$-condition concerning finite length.

\begin{defi}
    We will say that a $\Gamma$-graded right $R$-module $M$ has \emph{strongly $\Gamma_0$-finite gr-length} if there exists $n\in\mathbb{N}$ such that $M(e)$ has gr-length less than $n$ for each $e\in\Gamma_0$.
\end{defi}

We have the following result similar to \Cref{prop: comp finito = art + noet}  and \Cref{prop: G0 comp finito = G0 art + G0 noet}.

\begin{prop} 
\label{prop: comp fort finito <=> fort art e for noet}
Let $R$ be a $\Gamma$-graded ring and $M$ be a $\Gamma$-graded right $R$-module.
The following assertions are equivalent:

\begin{enumerate}[\rm (1)]

\item $M$ is strongly $\Gamma_0$-artinian and strongly $\Gamma_0$-noetherian.

\item $M$ has strongly $\Gamma_0$-finite gr-length.

\end{enumerate}
\end{prop}

\begin{proof}

$(2) \Rightarrow (1)$: Follows from  \Cref{prop: G0 comp finito = G0 art + G0 noet} and \Cref{prop: carac fort G0-art/noet}.

$(1) \Rightarrow (2)$: By \Cref{prop: carac fort G0-art/noet}, $M$ is both $\Gamma_0$-artinian and $\Gamma_0$-noetherian, and there exist $n_0 \in \mathbb{N}$ and subsets $\Delta_0, \Lambda_0 \subseteq \Gamma_0$ such that $\Gamma_0 \setminus \Delta_0$ and $\Gamma_0 \setminus \Lambda_0$ are finite and, for each $e \in \Delta_0 \cap \Lambda_0$, the module $M(e)$ does not contain strict chains of graded submodules of length greater than $n_0$.
By \Cref{prop: G0 comp finito = G0 art + G0 noet}, $M$ has $\Gamma_0$-finite gr-length.
Therefore,
\[
c_{\rm gr}(M(e)) \leq \max\left\{n_0,\, c_{\rm gr}(M(f)) : f \in \Gamma_0 \setminus (\Delta_0 \cap \Lambda_0)\right\}
\]
for each $e \in \Gamma_0$, and it follows that $M$ has strongly $\Gamma_0$-finite gr-length.
\end{proof}

Now we show a version of Fitting's Lemma for graded modules of strongly $\Gamma_0$-finite gr-length. 

\begin{prop}
    Let $R$ be a  $\Gamma$-graded ring and $M$ be a $\Gamma$-graded right $R$-module of strongly $\Gamma_0$-finite gr-length. For any gr-endomorphism $g\in \Endgr(M)$, there exists $n\in\mathbb{N}$ such that $M=\ker g^n\oplus\im g^n$, $g(\ker g^{n})\subseteq\ker g^{n}$, and the restriction $g|_{\im g^n}:\im g^n\to \im g^n$ is a gr-isomorphism.
        Moreover, if $g$ is injective or surjective, then $g$ is a gr-isomorphism.
\end{prop}
	
\begin{proof}
    The argument in the proof of \Cref{prop: Fitting Lemma for graded modules of fin. gr-length}(1) applies here, using \Cref{prop: comp fort finito <=> fort art e for noet} and \Cref{lem: Fitting para gr-end caso fort}.
\end{proof}

We end this subsection with an example of a preadditive category $\mathcal{C}$ such that the $\mathcal{C}_0\times \mathcal{C}_0$ graded ring $R_\mathcal{C}$ defined in \Cref{ex: anel de categoria} has strongly $\Gamma_0$-finite gr-length on the right (left).

\begin{prop}
    Let $A$ be a unital ring.
    Denote by $\proj$-$A$ the category of finitely generated projective right $A$-modules, and let $\mathcal{C}$ be a small full subcategory of $\proj$-$A$.
    Suppose that $A_A\in\mathcal{C}_0$ and there exists $n\in\mathbb{N}$ such that each object of $\mathcal{C}$ is generated by at most $n$ elements.
    Set $\Gamma:=\mathcal{C}_0\times\mathcal{C}_0$.
    \begin{enumerate}[\rm (1)]
        \item If $A$ is a right artinian ring, then $R_\mathcal{C}$ is a right strongly $\Gamma_0$-artinian and a right strongly $\Gamma_0$-noetherian ring.
        \item If $A$ is a left artinian ring, then $R_\mathcal{C}$ is a left strongly $\Gamma_0$-artinian and a left strongly $\Gamma_0$-noetherian ring.
    \end{enumerate}
\end{prop}

\begin{proof}
    We saw in the proof of \Cref{prop: proj-A art e noet} that  $R_\mathcal{C}=R_\mathcal{C}1_fR_\mathcal{C}$ where $f=(A,A)\in\Gamma_0$.
    
    (1) Set $n_0\in\mathbb{N}$ to be the length of $A^n$ as a right $A$-module. 
    Let $e=(Q,Q)\in\Gamma_0$ and suppose that 
    $X_1\supsetneq X_2\supsetneq\cdots\supsetneq X_t$
    is a strict descending chain of graded submodules of $R_\mathcal{C}(e)$.
    Then 
    \[(X_1)_{(Q,A)}\supseteq (X_2)_{(Q,A)}\supseteq\cdots\supseteq (X_t)_{(Q,A)}\]
    is a descending chain of right $A$-submodules of $(R_\mathcal{C})_{(Q,A)}=\mathcal{C}(A,Q)\cong Q_A$.
    This chain is strict, because $(X_i)_{(Q,A)}= (X_{i+1})_{(Q,A)}$ implies that
    \begin{align*}
        X_i=X_iR_\mathcal{C}=X_iR_\mathcal{C}1_fR_\mathcal{C}
        &=(X_i)_{(Q,A)}R_\mathcal{C}\\
        &=(X_{i+1})_{(Q,A)}R_\mathcal{C} =X_{i+1}R_\mathcal{C}1_fR_\mathcal{C}=X_{i+1}R_\mathcal{C}=X_{i+1}.
    \end{align*}
    Since $Q_A$ is isomorphic to a direct summand of $A^n$, it follows that $t\leq n_0$. 
    Therefore, $R_\mathcal{C}(e)$ is of finite gr-length not greater than $n_0$.

    It follows from \Cref{prop: comp fort finito <=> fort art e for noet} that $R_\mathcal{C}$ is right strongly $\Gamma_0$-artinian and right strongly $\Gamma_0$-noetherian.

    (2) Set $n_1\in\mathbb{N}$ to be the length of $A^n$ as a left $A$-module. 
    If $e'=(P,P)\in\Gamma_0$ and 
    $Y_1\supsetneq Y_2\supsetneq\cdots\supsetneq Y_t$
    is a strict descending chain of graded submodules of $(e')R_\mathcal{C}$, then 
    \[(Y_1)_{(A,P)}\supsetneq (Y_2)_{(A,P)}\supsetneq\cdots\supsetneq (Y_t)_{(A,P)}\]
    is a strict descending chain of left submodules of $(R_\mathcal{C})_{(A,P)}=\mathcal{C}(P,A)$.
    Since $\Hom_A(P,A)$ is generated by at most $n$ elements by \cite[Remark~2.11]{Lam2}, it follows that $\Hom_A(P,A)$ is isomorphic to a direct summand of $_AA^n$ and $(e')R_\mathcal{C}$ is of finite gr-length not greater than $n_1$.

    It follows from \Cref{prop: comp fort finito <=> fort art e for noet} that $R_\mathcal{C}$ is left strongly $\Gamma_0$-artinian and left strongly $\Gamma_0$-noetherian.
\end{proof}


\subsection{Right $\Gamma_0$-artinian ring \rlap{$\,\,\,\,\not$}$\implies$ right $\Gamma_0$-noetherian ring}
In classical ring theory, it is well known that if a ring is left artinian, then the ring is left noetherian, see for example \cite[Theorem 4.15]{Lam1}. There exists a version of this result in the group graded context \cite[Corollary 2.9.7]{NastasescuVanOystaeyen}. On the other hand,
the $\Gamma$-graded ring presented in \Cref{coro: contra-exemplos art noeth}(1)  is left $\Gamma_0$-artinian but not left $\Gamma_0$-noetherian.
In another paper currently in preparation, we will show that this behavior is related to issues with the nilpotency of the graded Jacobson radical and we will provide graded versions of the Hopkins-Levitzki Theorem.

The ring $\UT_I(A)$ presented in \Cref{exem: M_I(A)}  is the main object of study of this section. It is an adaptation of \cite[p.~19]{Mit} to the language of groupoid graded rings.

\begin{lema}
\label{lem: ideais unilaterais em UT_I(A)}
    Let $I$ be a partially ordered set, $A$ be a unital ring, and consider $R := \UT_I(A)$. 
    The following assertions hold true:
    \begin{enumerate}[\rm (1)]
        \item Let $X$ be a graded right ideal in $R$ of the form $X = X(e)$, where $e = (i_0, i_0)$ for some $i_0\in I$. 
        Then there exists a family $\{U_j : i_0 \leq j\}$ of right ideals of $A$ such that, for each $j,j'\in \{i\in I:i_0\leq i\}$,
        \[
        X_{(i_0,j)} = U_j E_{i_0 j}
        \qquad\text{and}\qquad
        j \leq j' \implies U_j \subseteq U_{j'}.
        \]
        \item Let $Y$ be a graded left ideal in $R$ of the form $Y = (f)Y$, where $f = (j_0, j_0)$ for some $j_0\in I$.
        Then there exists a family $\{V_i : i \leq j_0\}$ of left ideals of $A$ such that, for each $i,i'\in \{j\in I:j\leq j_0\}$,
        \[
        Y_{(i,j_0)} = V_i E_{i j_0}
        \qquad\text{and}\qquad
        i' \leq i \implies V_i \subseteq V_{i'}.
        \]
    \end{enumerate}
\end{lema}

\begin{proof}
    It suffices to consider $U_j:=\{a\in A: aE_{i_0 j}\in X\}$ for each $j\geq i_0$ in (1), and $V_i:=\{a\in A: aE_{i j_0}\in Y\}$ for each $i\leq j_0$ in (2).
\end{proof}

    Thanks to  \Cref{lem: ideais unilaterais em UT_I(A)}, we can study chains of graded submodules of $R(e)$ (or $(e)R$), analyzing chains of elements of $I$ and chains of right (left) ideals of $A$, as we see in the following result.

\begin{prop}
\label{prop: quando UT(A) eh art ou noeth}
Let $I$ be a totally ordered set, $A$ be a unital ring, and set $R := \UT_I(A)$. 
The following assertions hold true for $i_0,j_0\in I$, $e=(i_0,i_0)$, and $f=(j_0,j_0)$:
\begin{enumerate}[\rm (1)]
    \item  $R(e)$ is gr-artinian if and only if $A$ is right artinian and there does not exist an infinite chain in $I$ of the form
    \[
    i_0 < j_1 < j_2 < \cdots < j_n < \cdots.
    \]
    \item $R(e)$ is gr-noetherian if and only if $A$ is right noetherian and there does not exist an infinite chain in $I$ of the form
    \[
    i_0 < \cdots < j_n< \cdots < j_2 < j_1.
    \]
    \item $(f)R$ is gr-artinian if and only if $A$ is left artinian and there does not exist an infinite chain in $I$ of the form
    \[
    \cdots < i_n <\cdots < i_2 < i_1 < j_0.
    \]
    \item  $(f)R$ is gr-noetherian if and only if $A$ is left noetherian and there does not exist an infinite chain in $I$ of the form
    \[
    i_1 < i_2 < \cdots < i_n < \cdots < j_0.
    \]
\end{enumerate}
\end{prop}

\begin{proof}
    Every strict chain $\{U_n\}_{n\in\mathbb{N}}$ of right (resp. left) ideals of $A$ induces a strict chain $\{(U_nE_{i_0i_0})R\}_{n\in\mathbb{N}}$ (resp. $\{R(U_nE_{j_0j_0})\}_{n\in\mathbb{N}}$) of graded submodules of $R(e)$ (resp. $(f)R$).
    If $\{j_n\}_{n\in\mathbb{N}}$ is a chain in $I$ consisting of elements greater than $i_0$, then $\{E_{i_0j_n}R\}_{n\in\mathbb{N}}$ is a strict chain of graded submodules of $R(e)$.
    Similarly, if $\{i_n\}_{n\in\mathbb{N}}$ is a chain in $I$ consisting of elements less than $j_0$, then $\{RE_{i_nj_0}\}_{n\in\mathbb{N}}$ is a strict chain of graded submodules of $(f)R$.
    Hence, we have proved the ``only if'' part of statements (1)–(4).

    Now we prove the ``if'' part of statements (1)–(4). 
    
    (1) Suppose that $X_1\supsetneq X_2\supsetneq \cdots$ is an infinite chain of graded submodules of $R(e)$. 
    By \Cref{lem: ideais unilaterais em UT_I(A)}(1), for each $n\in\mathbb{N}$ and $j\geq i_0$ there exists a  right ideal $U_{n,j}$ of $A$ such that $(X_n)_{(i_0,j)}=U_{n,j}E_{i_0j}$, and $U_{n,j}\subseteq U_{n,j'}$ if $j\leq j'$.
    For each $j\in [i_0,\infty):=\{j\in I \colon i_0\leq j\}$, consider the descending chain of right ideals of $A$ 
    \[U_{1,j}\supseteq U_{2,j}\supseteq U_{3,j}\supseteq\dotsb \supseteq U_{n,j}\supseteq \dotsb.\]
    Since $A$ is artinian, this chain stabilizes. Set \[n_j:= \textrm{ least $n$ such that } U_{n,j}=U_{n+l,j} \textrm{ for all } l\geq 0.\]
    We claim that the set $N=\{n_j\colon j\in [i_0,\infty)\}$ is unbounded. We prove the claim by way of contradiction. Suppose the claim is not true and fix $n_0>n_j$ for all $j\in [i_0,\infty)$. Then $U_{n_0,j}=U_{n_0+l,j}$ for all $l\geq 0$ and $j\in [i_0,\infty)$. Hence $X_{n_0}=X_{n_0+l}$ for all $l\geq 0$, a contradiction.

    Choose $j_1\in [i_0,\infty)$. Define $W_1=U_{n_{j_1},j_1}$. Since $N$ is unbounded and there are not infinite chains in $I$ of the form
    \[
    i_0 < j_1 < j_2 < \cdots < j_n < \cdots,
    \]
    there exists $j_2<j_1$ with $n_{j_2}>n_{j_1}$. Set $W_2=U_{n_{j_2},j_2}$. Then 
    \[
    W_1=U_{n_{j_1},j_1}\supseteq U_{n_{j_1},j_2}\supsetneq U_{n_{j_2},j_2}=W_2. 
    \]
    In the same way, it can be shown that there exists $j_3<j_2$ with $n_{j_3}>n_{j_2}$. Set $W_3=U_{n_{j_3},j_3}$. Then 
    \[
    W_2=U_{n_{j_2},j_2}\supseteq U_{n_{j_2},j_3}\supsetneq U_{n_{j_3},j_3}=W_3. 
    \]
    In this way, we construct a sequence of right ideals of $A$ 
    \[
    W_1\supsetneq W_2\supsetneq W_3\supsetneq\dotsb \supsetneq W_n\supsetneq \dotsb
    \]
    which contradicts the fact that $A$ is right artinian. Therefore, the infinite chain $X_1\supsetneq X_2\supsetneq \cdots$ of graded submodules of $R(e)$  does not exist.

    (2) It can be shown in a similar way to (1).

    (3) Suppose that $Y_1\supsetneq Y_2\supsetneq \cdots$ is an infinite chain of graded submodules of $(f)R$. 
    By \Cref{lem: ideais unilaterais em UT_I(A)}(2), for each $n\in\mathbb{N}$ and $i\leq j_0$ there exists a left ideal $V_{n,i}$ of $A$ such that $(Y_n)_{(i,j_0)}=V_{n,i}E_{ij_0}$, and $V_{n,i}\subseteq V_{n,i'}$ if $i'\leq i$.
    For each $i\in (-\infty, j_0]:=\{i\in I\colon i\leq j_0\}$, consider the descending chain of left ideals of $A$ 
    \[V_{1,i}\supseteq V_{2,i}\supseteq V_{3,i}\supseteq\dotsb \supseteq V_{n,i}\supseteq \dotsb.\]
    Since $A$ is artinian, this chain stabilizes. Set \[n_i:= \textrm{ least $n$ such that } V_{n,i}=V_{n+l,i} \textrm{ for all } l\geq 0.\]
    We claim that the set $N=\{n_i\colon i\in (-\infty, j_0] \}$ is unbounded. We prove the claim by way of contradiction. Suppose the claim is not true and fix $n_0>n_i$ for all $i\in  (-\infty, j_0]$. Then $V_{n_0,i}=V_{n_0+l,i}$ for all $l\geq 0$ and $i\in  (-\infty, j_0]$. Hence $Y_{n_0}=Y_{n_0+l}$ for all $l\geq 0$, a contradiction. 

    Choose $i_1\in  (-\infty, j_0]$. Define  $W_1=V_{n_{i_1},i_1}$. Since $N$ is unbounded and there are not infinite chains in $I$ of the form
\[
    \cdots < i_n <\cdots < i_2 < i_1 < j_0,
    \]    
     there exists $i_2>i_1$ with $n_{i_2}>n_{i_1}$. Set $W_2=V_{n_{i_2},i_2}$. Then 
    \[
    W_1=V_{n_{i_1},i_1}\supseteq V_{n_{i_1},i_2}\supsetneq V_{n_{i_2},i_2}=W_2. 
    \]
    In the same way, it can be shown that there exists $i_3>i_2$ with $n_{i_3}>n_{i_2}$. Set  $W_3=V_{n_{i_3},i_3}$. Then 
    \[
    W_2=V_{n_{i_2},i_2}\supseteq V_{n_{i_2},i_3}\supsetneq V_{n_{i_3},i_3}=W_3. 
    \]
    In this way, we construct a sequence of left ideals of $A$ 
    \[
    W_1\supsetneq W_2\supsetneq W_3\supsetneq\dotsb \supsetneq W_n\supsetneq \dotsb
    \]
    which contradicts the fact that $A$ is left artinian. Therefore, the infinite chain $Y_1\supsetneq Y_2\supsetneq \cdots$ of graded submodules of $(f)R$  does not exist.

    (4) It can be shown in a similar way to (3).
\end{proof}

Observe that the ``only if'' implications of \Cref{prop: quando UT(A) eh art ou noeth} do not use the fact that the set $I$ is totally ordered. Hence, these implications hold for any partially ordered set $I$.

If $I$ is a partially ordered set with order relation $<$, we will denote by $I^{op}$ the set $I$ endowed with the partial order relation $<^{op}$ defined as follows
\[ i<^{op} j \quad \Longleftrightarrow \quad j<i,\]
for all $i,j\in I$.

Now we are ready to show  examples of one-sided $\Gamma_0$-artinian rings that are not $\Gamma_0$-noetherian on the same side.

\begin{coro}\label{coro: contra-exemplos art noeth}
 Let $\lambda$ be an ordinal greater than the first infinite ordinal $\omega$, and let $I$ be the set of ordinals less than $\lambda$. Set $\Gamma=I\times I$.
 \begin{enumerate}[\rm(1)]
     \item If $A$ is a left artinian unital ring, then $\UT_I(A)$ is a left $\Gamma_0$-artinian ring that is not left $\Gamma_0$-noetherian.
     \item If $A$ is a right artinian unital ring, then $\UT_{I^{op}}(A)$ is a right $\Gamma_0$-artinian ring that is not right $\Gamma_0$-noetherian.
 \end{enumerate}
\end{coro}

\begin{proof}
(1) Since $I$ is a well ordered set and $A$ is left artinian, \Cref{prop: quando UT(A) eh art ou noeth}(3) implies that $\UT_I(A)$ is left $\Gamma_0$-artinian. On the other hand,   the infinite chain $0 < 1 < 2 < \cdots < \omega$ in $I$ shows that $\UT_I(A)$ is not left $\Gamma_0$-noetherian, by \Cref{prop: quando UT(A) eh art ou noeth}(4).

(2) Since in $I^{op}$ there are no infinite ascending chains and $A$ is right artinian, \Cref{prop: quando UT(A) eh art ou noeth}(1) implies that $\UT_{I^{op}}(A)$ is right $\Gamma_0$-artinian. On the other hand, there are bounded below infinite descending chains in $I^{op}$. Thus, \Cref{prop: quando UT(A) eh art ou noeth}(2) implies that  $\UT_{I^{op}}(A)$ is not right $\Gamma_0$-noetherian.
\end{proof}

Some other interesting consequences of \Cref{prop: quando UT(A) eh art ou noeth} are as follows.

\begin{coro}
\label{coro: UT(D) is left artinian}
    Let $I$ be a well-ordered set, $A$ be a unital ring, $\Gamma := I \times I$, and $R := \UT_I(A)$. The following assertions hold:
    \begin{enumerate}[\rm (1)]
        \item If $A$ is right noetherian, then $R$ is a right $\Gamma_0$-noetherian ring.
        \item If $A$ is left artinian, then $R$ is a left $\Gamma_0$-artinian ring.
        \item If $I = \mathbb{N}$ and $A$ is left noetherian, then $R$ is a left $\Gamma_0$-noetherian ring.
        \item If $I$ is infinite, then $R$ is not right $\Gamma_0$-artinian.
    \end{enumerate}
\end{coro}

\begin{proof}
    (1) Since $I$ is well-ordered, there are no chains in $I$ as in \Cref{prop: quando UT(A) eh art ou noeth}(2).
    
    (2) Since $I$ is well-ordered, there are no chains in $I$ as in \Cref{prop: quando UT(A) eh art ou noeth}(3).

    (3) There is no chain in $\mathbb{N}$ as in \Cref{prop: quando UT(A) eh art ou noeth}(4).
    
    (4) In any infinite well-ordered set, it is always possible to construct a chain as in 
    \Cref{prop: quando UT(A) eh art ou noeth}(1).
\end{proof}


\subsection{Gr-injetive modules and the graded Bass-Papp Theorem}

In this subsection, we aim to characterize $\Gamma_0$-noetherian rings via gr-injective modules. In other words, we will obtain a groupoid graded version of the Bass–Papp Theorem, see for example \cite[(3.46)]{Lam2} or \cite[Theorem~5.23]{Goodearl}. Parts of \Cref{prop: defs equiv de gr-inj} and \Cref{teo: baer}, together with \Cref{prop: prod de gr-inj eh gr-inj}, appear in \cite{Lund} for unital groupoid graded rings. 

\begin{defi}\label{def:gr-injective}
	Let $R$ be a $\Gamma$-graded ring and let $E$ be a $\Gamma$-graded right $R$-module. We say that $E$ is \emph{gr-injective} if, for every pair of $\Gamma$-graded right $R$-modules $M, N$ and all $g \in \Homgr(N, M)$, $h \in \Homgr(N, E)$ with $g$ being injective, there exists $h' \in \Homgr(M, E)$ such that $h' \circ g = h$.
\[
\begin{tikzcd}
0 \arrow{r} & N \arrow{r}{g} \arrow[swap]{d}{h} & M \arrow[dashed]{dl}{h'}  & \\
& E &
\end{tikzcd}
\]
\end{defi}

We denote by $\Gamma\text{-gr-}R$ the category whose objects are the $\Gamma$-graded right $R$-modules and morphisms are gr-homomorphisms.  By definition, gr-injective $\Gamma$-graded right $R$-modules are precisely the injective objects of the category $\Gamma\text{-gr-}R$. Now we provide equivalent conditions for a graded module to be gr-injective.

\begin{prop}
\label{prop: defs equiv de gr-inj}
	Let $R$ be a $\Gamma$-graded ring and $E$ a $\Gamma$-graded right $R$-module. 
    The following assertions are equivalent:
	\begin{enumerate}[\rm (1)]
		\item $E$ is gr-injective.
		\item For all $\Gamma$-graded right $R$-modules $M,N$ and $g\in\Homgr(N,M)$, $h\in\HOM_R(N,E)$ with $g$ being injective, there exists $h'\in\HOM_R(M,E)$ such that $h'\circ g = h$.
		\item For all $\Gamma$-graded right $R$-modules $M,N$, for every $\gamma\in\Gamma$, and $g\in\Homgr(N,M)$, $h_\gamma\in\HOM_R(N,E)_\gamma$ with $g$ being injective, there exists $h'_\gamma\in\HOM_R(M,E)_\gamma$ such that $h'_\gamma\circ g = h_\gamma$.
		\item For all $\Gamma$-graded right $R$-modules $M,N$, and all $\gamma,\delta\in\Gamma$ with $d(\gamma)=d(\delta)$, and $g_\delta\in\HOM_R(N,M)_\delta$, $h_\gamma\in\HOM_R(N,E)_\gamma$ with $g_\delta|_{N(d(\delta))}$ being injective, there exists $h'\in\HOM_R(M,E)_{\gamma\delta^{-1}}$ such that $h'\circ g_\delta = h_\gamma$.
	\end{enumerate}
\end{prop}

\begin{proof}
	$(1)\Rightarrow(3)$: Suppose (1) holds, $g\in\Homgr(N,M)$ is injective and $h_\gamma\in\HOM_R(N,E)_\gamma$. 
    Since $g|_{N(\gamma^{-1})}\in\Homgr(N(\gamma^{-1}),M(\gamma^{-1}))$ and $h_\gamma|_{N(\gamma^{-1})}\in\Homgr(N(\gamma^{-1}),E)$, there exists $h'\in\Homgr(M(\gamma^{-1}),E)$ such that $h'\circ g|_{N(\gamma^{-1})} = h_\gamma|_{N(\gamma^{-1})}$. 
    Then, using the same reasoning as in the proof of \cite[Proposition 3.5(1)]{Artigoarxiv}, we can extend $h'$ to $h'_\gamma\in\HOM_R(M,E)_\gamma$ by defining $h'_\gamma=0$ on $M(e)$ for every $e\in\Gamma_0\setminus\{d(\gamma)\}$. 
    Hence, $h'_\gamma\circ g=h_\gamma$.

	$(3)\Rightarrow(2)$: To obtain (2), simply apply (3) to each nonzero homogeneous component of $h=\sum_{\gamma\in\Gamma}h_\gamma$ and take $h':=\sum_{\gamma\in\Gamma}h'_\gamma$.

	$(2)\Rightarrow(3)$: Suppose (2) holds and let $\gamma\in\Gamma$, $g\in\Homgr(N,M)$, and $h_\gamma\in\HOM_R(N,E)_\gamma$ with $g$ being injective. 
    Then, there exists $h'\in\HOM_R(M,E)$ such that $h'\circ g = h_\gamma$. 
    Write $h' = \sum_{\sigma\in\Gamma} h'_\sigma$ with $h'_\sigma\in\HOM_R(M,E)_\sigma$ for each $\sigma\in\Gamma$. 
    Given $\alpha\in\Gamma$ and $n_\alpha\in N_\alpha$, we have
	\[E_{\gamma\alpha} \ni h_\gamma(n_\alpha) = \sum_{\sigma\in\Gamma}(h'_\sigma g)(n_\alpha),\]
	and hence $h_\gamma(n_\alpha) = (h'_\gamma g)(n_\alpha)$, so $h'_\gamma\circ g = h_\gamma$.

	$(3)\Rightarrow(4)$: Suppose (3) holds, $\gamma,\delta\in\Gamma$ with $d(\gamma)=d(\delta)$, $g_\delta\in\HOM_R(N,M)_\delta$ is injective when restricted to $N(d(\delta))$, and $h_\gamma\in\HOM_R(N,E)_\gamma$. 
    Considering $g_\delta|_{N(d(\delta))}$ as an element of $\Homgr(N(d(\delta)),M(\delta))$, there exists $h'_\gamma\in\HOM_R(M(\delta),E)_\gamma$ such that $h'_\gamma \circ g_\delta|_{N(d(\delta))} = h_\gamma|_{N(d(\delta))}$. 
    Using the same reasoning as in the proof of \cite[Proposition 3.5(1)]{Artigoarxiv}, we can extend $h'_\gamma$ to $h'\in\HOM_R(M,E)_{\gamma\delta^{-1}}$ by defining $h'=0$ on $M(e)$ for all $e\in\Gamma_0\setminus\{r(\delta)\}$.
    Since $g_\delta=h_\gamma=0$ on $N(f)$ for each $e\in\Gamma_0\setminus\{d(\delta)\}$, it follows that $h'\circ g_\delta=h_\gamma$.

	$(4)\Rightarrow(1)$: Suppose (4) holds, and let $g\in\Homgr(N,M)$ and $h\in\Homgr(N,E)$ with $g$ being injective. 
    Fix $e\in\Gamma_0$. 
    Since $g|_{N(e)}\in\HOM_R(N(e),M(e))_e$ is injective and $h|_{N(e)}\in\HOM_R(N(e),E)_e$, there exists $h'_e\in\HOM_R(M(e),E)_e = \Homgr(M(e),E(e))$ such that $h'_e\circ g|_{N(e)} = h|_{N(e)}$. 
    Thus, $h' := \bigoplus_{e\in\Gamma_0} h'_e \in \Homgr(M,E)$ is such that $h'\circ g = h$.
\end{proof}

We continue proving some important properties of gr-injective modules.

\begin{lema}
\label{lem: somand dir de gr-inj e gr-inj}
Let $R$ be a $\Gamma$-graded ring and let $E$ be a gr-injective $\Gamma$-graded right $R$-module. 
Then every graded direct summand of $E$ is also gr-injective.
\end{lema}

\begin{proof}
Let $X$ be a graded direct summand of $E$. 
Take $\Gamma$-graded right $R$-modules $M,N$ and $g\in\Homgr(N,M)$, $h\in\Homgr(N,X)$ with $g$ being injective. 
Since $h\in\Homgr(N,E)$, there exists $h'\in\Homgr(M,E)$ such that $h'\circ g = h$. 
Consider $\pi\in\Homgr(E,X)$, the canonical projection from $E$ onto $X$. 
Then $\pi \circ h' \in \Homgr(M,X)$ satisfies $(\pi\circ h')\circ g = \pi \circ h = h$. 
Therefore, $X$ is gr-injective.
\end{proof}

\begin{prop}
\label{prop: prod de gr-inj eh gr-inj}
Let $R$ be a $\Gamma$-graded ring, let $\{E_i : i\in I\}$ be a family of $\Gamma$-graded right $R$-modules, and consider $E := \prod_{i\in I}^{\rm gr} E_i$.  
Then $E$ is gr-injective if and only if  $E_i$ is gr-injective for all $i\in I$.
\end{prop}

\begin{proof}
 Suppose $E$ is gr-injective. For each $i_0\in I$, $E=\left(\prod_{i \in I\setminus \{i_0\}}^{\rm gr} E_i\right) \oplus E_{i_0}$. By \Cref{lem: somand dir de gr-inj e gr-inj}, $E_{i_0}$ is gr-injective for each $i_0\in I$.    

 Conversely, suppose that  $E_i$ is gr-injective for all $i\in I$. Take  $\Gamma$-graded right $R$-modules $M, N$ and $g \in \Homgr(N, M)$, $h \in \Homgr(N, E)$ with $g$ being injective. Consider the canonical projection $\pi_i\in \Homgr(E,E_i)$ for each $i\in I$. Since $\pi_i\circ h\in \Homgr(N,E_i)$, there exists $h_i\in \Homgr(M,E_i)$ such that $h_ig=\pi_ih$. Then $h':=(h_i)_{i\in I}\in \Homgr(M,E)$ is such that $h'g=(h_ig)_{i\in I}=(\pi_ih)_{i\in I}=h$. Therefore, $E$ is gr-injective.
\end{proof}

Gr-injectivity behaves well with respect to shifts, as follows:

\begin{prop}
	\label{prop: I gr-inj sse I(e) gr-inj}
	Let $R$ be a $\Gamma$-graded ring and $E$ a $\Gamma$-graded right $R$-module. The following assertions are equivalent:
	\begin{enumerate}[\rm (1)]
		\item $E$ is gr-injective.
		\item $E(\sigma)$ is gr-injective for every $\sigma\in\Gamma$.
		\item $E(e)$ is gr-injective for every $e\in\Gamma_0$.
	\end{enumerate}
\end{prop}

\begin{proof}
$(1)\Rightarrow(2)$: Since $\Homgr(X,E(\gamma))=\HOM_R(X,E)_\gamma$ for each $\Gamma$-graded right $R$-module $X$ and $\gamma\in\Gamma$, the condition (3) in \Cref{prop: defs equiv de gr-inj} is the definition of gr-injectivity of $E(\gamma)$ for every $\gamma\in\Gamma$.

$(2)\Rightarrow(3)$: This is clear.

$(3)\Rightarrow(1)$: This implication is a consequence of \Cref{prop: prod de gr-inj eh gr-inj} because $E=\bigoplus_{e\in\Gamma_0}E(e)\isogr\prod_{e\in\Gamma_0}^{gr}E(e)$.
\end{proof} 

We now state our graded version of Baer's Criterion. We point out that \Cref{teo: baer} is stated in terms of homomorphisms with degree although \Cref{def:gr-injective} is given in terms of gr-homomorphisms. 

\begin{teo}
\label{teo: baer}
Let $R$ be a $\Gamma$-graded ring and $E$ a $\Gamma$-graded right $R$-module.
The following assertions are equivalent:
\begin{enumerate}[\rm (1)]
	\item $E$ is gr-injective.
	\item For every graded right ideal $U$ of $R$ and every $h\in\HOM_R(U,E)$, there exists $h'\in\HOM_R(R,E)$ such that $h'|_U=h$.
	\item For every graded right ideal $U$ of $R$, every $\gamma\in\Gamma$, and every $h\in\HOM_R(U,E)_\gamma$, there exists $h'\in\HOM_R(R,E)_\gamma$ such that $h'|_U=h$.
\end{enumerate}    
\end{teo}

\begin{proof}
$(1)\Rightarrow(2)$: Follows directly from the characterization of gr-injectivity given in item (2) of \Cref{prop: defs equiv de gr-inj}.

$(2)\Rightarrow(3)$: Follows the same reasoning as the proof of $(2)\Rightarrow(3)$ in \Cref{prop: defs equiv de gr-inj}.

$(3)\Rightarrow(1)$: The proof is very similar to that of (iii)$\Rightarrow$(i) in \cite[Proposition 3.5.2]{Lund} and to the one in \cite[Proposition 5.1]{Goodearl}. 
Assume (3) holds and let $M$ and $N$ be $\Gamma$-graded right $R$-modules, with $g\in\Homgr(N,M)$ being injective and $h\in\Homgr(N,E)$. 
Consider the set
\[
\chi := \{(N',h') : \im g \subseteq N' \subgr M, \, h' \in \Homgr(N',E) \text{ and } h'\circ g = h\}.
\] 
Note that $(\im g, \, h \circ g^{-1}) \in \chi$. 
We partially order $\chi$ by
\[
(N',h') \leqslant (N'',h'') \quad\Longleftrightarrow\quad N' \subseteq N'' \text{ and } h''|_{N'} = h'.
\]
Let $\{(N_i,h_i)\}_{i\in I}$ be a chain in $\chi$. 
Then $N' := \bigcup_{i\in I} N_i$ is a graded submodule of $M$ containing $\im g$ and
\begin{align*}
    h': N' &\longrightarrow E \\
    N_i \ni n &\longmapsto h_i(n)
\end{align*}
is a gr-homomorphism satisfying $h'\circ g = h$. 
Therefore, $(N',h') \in \chi$ is an upper bound of $\{(N_i,h_i)\}_{i\in I}$. 
By Zorn's Lemma, there exists a maximal element $(N_0,h_0) \in \chi$. 
To prove that $E$ is gr-injective, it suffices to show that $N_0 = M$. 
Suppose, by contradiction, that $N_0 \subsetneq M$ and take $x \in \h(M) \setminus N_0$. 
Consider the graded right ideal $U :=  \{r \in R : x r \in N_0\}$ and define
\begin{align*}
    \varphi : U &\longrightarrow E \\
    a &\longmapsto h_0(x a).
\end{align*}
Note that $\varphi \in \HOM_R(U,E)_\gamma$, where $\gamma := \deg x$. By (3), there exists $\varphi' \in \HOM_R(R,E)_\gamma$ extending $\varphi$. Now consider the $\Gamma$-graded right $R$-module $N_1 := N_0 + xR$ and the map
\begin{align*}
    h_1 : N_1 &\longrightarrow E\\
    y + x a &\longmapsto h_0(y) + \varphi'(a).
\end{align*}
This is well defined: if $y + x a = \tilde{y} + x \tilde{a}$, with $y, \tilde{y} \in N_0$ and $a, \tilde{a} \in R$, then $x(a - \tilde{a}) = \tilde{y} - y \in N_0$, and hence:
\[
h_0(\tilde{y}) - h_0(y) =  h_0(x(a - \tilde{a})) = \varphi(a - \tilde{a}) = \varphi'(a - \tilde{a}) = \varphi'(a) - \varphi'(\tilde{a}).
\]
Clearly, $h_1$ is a homomorphism of right $R$-modules. 
It is also a gr-homomorphism: for each $\alpha \in \Gamma$ we have
\[
h_1((N_1)_\alpha) = h_1((N_0)_\alpha + x R_{\gamma^{-1}\alpha}) \subseteq h_0((N_0)_\alpha) + \varphi'(R_{\gamma^{-1}\alpha}) \subseteq E_\alpha + E_{\gamma \gamma^{-1} \alpha} = E_\alpha.
\]
Since $\im g \subseteq N_0$ and $h_1|_{N_0} = h_0$, it follows that $h_1\circ g = h_0\circ g = h$. 
Thus, $(N_1,h_1) \in \chi$, contradicting the maximality of $(N_0,h_0)$ since $N_0 \subsetneq N_1$.
\end{proof}

Out next step is to show that every groupoid graded module is a graded submodule of a gr-injective module. We will show this via a graded version of \cite[Chapter~5]{Goodearl}. This fact could also be proved via a groupoid graded version of \cite[Subsection 4.1]{Balaba}, as it is done in \cite[Section~2.5]{dissertacao}. One could also show that the category $\Gamma\text{-gr-}R$  is a Grothendieck category  with $P=\bigoplus_{\sigma\in\Gamma}R(\sigma)$ as a projective generator, and then apply \cite[Proposition IV.4.4]{Stenstrom}.

Consider the $\Gamma$-graded ring
 \[\mathcal{Z}:=\bigoplus_{\gamma\in\Gamma}{\mathcal{Z}_\gamma},\quad  \textrm{ where } \quad
\mathcal{Z}_\gamma=\begin{cases}
 	\mathbb{Z},& \textrm{if $\gamma\in\Gamma_0$}\\
 	0,& \textrm{if $\gamma\notin\Gamma_0$}
 \end{cases}.\]
 Of course, every $\Gamma$-graded right $\mathcal{Z}$-module is a $\Gamma$-graded abelian group. On the other hand.
every $\Gamma$-graded abelian group is a $\Gamma$-graded right $\mathcal{Z}$-module. Indeed, let $X$ be a $\Gamma$-graded abelian group. We endow $X$ with  a $\Gamma$-graded right $\mathcal{Z}$-module structure via \[x_\gamma\cdot (n_e)_{e\in\Gamma_0}= x_\gamma n_{d(\gamma)}\]
for all $\gamma\in\Gamma$, $x_\gamma\in X_\gamma$ and $(n_e)_{e\in\Gamma_0}\in\mathcal{Z}$.

\begin{defi}
A $\Gamma$-graded right $\mathcal{Z}$-module $X$ is \emph{gr-divisible} provided that $X_\gamma n=X_\gamma$ for all $\gamma\in\Gamma$ and $n\in \mathcal{Z}_{d(\gamma)}\setminus\{0\}$.    
\end{defi}

\begin{exems}\label{exs: divisible}
    \begin{enumerate}[\rm(1)]
        \item The $\Gamma$-graded right $\mathcal{Z}$-module
         \[\mathcal{Q}:=\bigoplus_{\gamma\in\Gamma}\mathcal{Q_\gamma},\quad \textrm{where} \quad  
 \mathcal{Q}_\gamma=\begin{cases}
 	\mathbb{Q},& \textrm{if $\gamma\in\Gamma_0$}\\
 	0,& \textrm{if $\gamma\notin\Gamma_0$}
 \end{cases}\]
 is  gr-divisible.
 \item If $X$ is a gr-divisible right $\mathcal{Z}$-module, then the shift $X(\sigma)$ is gr-divisible for all $\sigma\in\Gamma$.
 \item Direct sums and graded direct products of gr-divisible right $\mathcal{Z}$-modules are gr-divisible.

 \item If $X$ is a gr-divisible right $\mathcal{Z}$-module and $Y$ is a graded submodule, then $X/Y$ is gr-divisible.
    \end{enumerate}
\end{exems}

\begin{prop}\label{prop: divisible <--> gr-injective}
    \begin{enumerate}[\rm (1)]
        \item A $\Gamma$-graded right $\mathcal{Z}$-module $X$ is gr-injective if and only if it is gr-divisible.
        \item Every $\Gamma$-graded right $\mathcal{Z}$-module $X$ is a graded submodule of a gr-divisible right $\mathcal{Z}$-module.
    \end{enumerate}
\end{prop}

\begin{proof}
    (1) Suppose that $X$ is gr-injective. Let $\gamma\in\Gamma$, $n\in \mathcal{Z}_{d(\gamma)}\setminus\{0\}$ and $x\in X_\gamma$.  Let $U=\bigoplus_{\sigma\in\Gamma}U_\sigma$ be the graded right ideal of $\mathcal{Z}$ given by $U_{d(\gamma)}=n\mathbb{Z}$ and  $U_\sigma=\{0\}$ for all $\sigma\in\Gamma\setminus \{d(\gamma)\}$. Define $\varphi\in\HOM_{\mathcal{Z}}(U,X)_\gamma$ by $\varphi(n)=x$. Since $X$ is gr-injective, there exists an extension $\psi\in\HOM_{\mathcal{Z}}(\mathcal{Z},X)_\gamma$ of $\varphi$. Then $\psi(1)n=\psi(n)=\varphi(n)=x$.
    
    Suppose now that $X$ is gr-divisible. We will use \Cref{teo: baer} to show that $X$ is gr-injective. Let $U$ be a graded right ideal of $\mathcal{Z}$, $\gamma\in\Gamma$ and $\varphi\in\HOM_{\mathcal{Z}}(U,X)_\gamma$. Note that $U=\bigoplus_{e\in\Gamma_0}U_e$ and $\varphi(U_e)\subseteq X_{\gamma e}$ for each $e\in\Gamma_0$. 
  We can suppose that $U_{d(\gamma)}\neq0$, otherwise $\varphi=0$. 
	Since $U_e$ is an ideal of $\mathbb{Z}$ for each $e\in\Gamma_0$, take $n\in\mathbb{Z}\setminus\{0\}$ such that $U_{d(\gamma)}=n\mathbb{Z}$. Let  $q=\varphi(n)\in X_\gamma$. Choose $d\in X_\gamma$ such that $dn=q$. Such $d$ exists because $X$ is gr-divisible. Then $\varphi|_{U_{d(\gamma)}}$ extends to
 	\[\psi_{d(\gamma)}:\mathcal{Z}_{d(\gamma)} \to X_\gamma,\quad  1\mapsto d.\]
 	Since $\varphi(U_e)=0$ for all $e\neq d(\gamma)$, we can extend $\psi_{d(\gamma)}$ to $\psi\in\HOM_{\mathcal{Z}}(\mathcal{Z},D)_\gamma$ by making $\psi=0$ in all other homogeneous components of $\mathcal{Z}$.

    (2) We have $X\cong_{gr}\frac{\bigoplus_{i\in I}\mathcal{Z}(\sigma_i)}{N}$ for some set $I$, $\sigma_i\in\Gamma$ for each $i\in I$ and graded subgroup $N$ of $\bigoplus_{i\in I}\mathcal{Z}(\sigma_i)$. Note that $\frac{\bigoplus\limits_{i\in I}\mathcal{Z}(\sigma_i)}{N}$ is a graded submodule of $\frac{\bigoplus\limits_{i\in I}\mathcal{Q}(\sigma_i)}{N}$, where $\mathcal{Q}$ is the gr-divisible right $\mathbb{Z}$-module given in \Cref{exs: divisible}(1). By \Cref{exs: divisible}(2)--(4), $\frac{\bigoplus\limits_{i\in I}\mathcal{Q}(\sigma_i)}{N}$ is gr-divisible.
\end{proof}

Let $R$ be a $\Gamma$-graded ring and $X$ be a $\Gamma$-graded right $\mathcal{Z}$-module. We can endow $\HOM_\mathcal{Z}(R_\mathcal{Z},X_\mathcal{Z})$ with a structure of  $\Gamma$-graded right $R$-module via 
 \[ 
 \begin{array}{rl} 
 	ga:R&\to X\\
 	b&\mapsto g(ab)
 \end{array}\]
 for all $a,b\in R$, $g\in \HOM_\mathcal{Z}(R_\mathcal{Z},X_\mathcal{Z})$, see \cite[Proposition 17(b)]{CLP}.

 \begin{lema}\label{lem: gr-injective from divisible}
    Let $R$ be a $\Gamma$-graded ring and $X$ be a gr-divisible $\Gamma$-graded right $\mathcal{Z}$-module. Then the $\Gamma$-graded right $R$-module $H=\HOM_\mathcal{Z}(R_\mathcal{Z},X_\mathcal{Z})$ is gr-injective.
 \end{lema}

\begin{proof}
   We will use \Cref{teo: baer} in order to show that $H$ is gr-injective. Let $U$ be a graded right ideal of $R$, $\gamma\in\Gamma$ and $\varphi\in \HOM_R(U,H)_\gamma$.  Consider $\varphi'\in \HOM_\mathcal{Z}(U,X)_\gamma$ determined by $r\rightarrow (\varphi(r))(1_{d(\delta)})$ for each $r\in U_\delta$ and $\delta\in\Gamma$. Since $X$ is a gr-injective right  $\mathcal{Z}$-module by \Cref{prop: divisible <--> gr-injective}(1), there exists an extension $\psi\colon \HOM_\mathcal{Z}(R,X)_\gamma$ of $\varphi'$. Thus, $\psi\in H$. 
   Now, there exists a unique  homomorphism of right $R$-modules of degree  $\gamma$, $\Phi\colon R(d(\gamma))\rightarrow H$, such that $\Phi(1_{d(\gamma)})=\psi$. Extend  $\Phi$ to $R$ defining $\Phi (x)=0$ for all $x\in R(e)$, $e\in\Gamma_0\setminus \{d(\gamma)\}$. Then $\Phi\in  \HOM_R(R,H)_\gamma$ and, for all $a\in U(d(\gamma))$ and $b\in R$,
   \[(\Phi(a))(b)=(\psi a)(b)=\psi(ab)=\varphi'(ab)=(\varphi(a))(b). \]
   Therefore $\Phi\colon R\rightarrow H$ is the desired extension of $\varphi$.
\end{proof}

\begin{teo}
\label{teo: todo modulo eh subgr de um gr-inj}
    Let $R$ be a $\Gamma$-graded ring and let $M$ be a $\Gamma$-graded right $R$-module. Then, there exists a gr-injective $\Gamma$-graded right $R$-module $E$ such that $M\subgr E$.
\end{teo}

\begin{proof}
    Viewed as a $\Gamma$-graded right $\mathcal{Z}$-module, $M$ is a graded submodule of a gr-divisible right $\mathcal{Z}$-module $X$ by \Cref{prop: divisible <--> gr-injective}(2). By \Cref{lem: gr-injective from divisible}, $E=\HOM_\mathcal{Z}(R_\mathcal{Z},X_\mathcal{Z})$ is a gr-injective right $R$-module. The result now follows from the graded embeddings of right $R$-modules \[M\cong_{gr} \HOM_R(R,M)\leq_{gr} \HOM_\mathcal{Z}(R,M)\leq_{gr} \HOM_\mathcal{Z}(R_\mathcal{Z},X_\mathcal{Z})=E.\qedhere \]
\end{proof}

We are ready to prove the following graded version of the Bass-Papp Theorem.

\begin{teo}
\label{teo: Bass-Papp}
    A $\Gamma$-graded ring $R$ is right $\Gamma_0$-noetherian if and only if every direct sum of gr-injective right $R$-modules is also gr-injective.
\end{teo}

\begin{proof}
    Suppose that $R$ is a right $\Gamma_0$-noetherian ring. 
    Let $\{E_\lambda:\lambda\in\Lambda\}$ be a family of gr-injective right $R$-modules and set $E:=\bigoplus_{\lambda\in\Lambda}E_\lambda$. 
    We will show that $E$ is gr-injective by proving that \Cref{teo: baer}(3) holds. 
    Let $\gamma\in\Gamma$, $U$ be a graded right ideal of $R$ and $h\in\HOM_R(U,E)_\gamma$. 
    By \Cref{prop: gr-noeth <-> todo submod eh fg}(1), $U(d(\gamma))$ is finitely generated. 
    Therefore, there exists a finite subset $\Lambda'\subseteq\Lambda$ such that $h(U)=h(U(d(\gamma)))\subseteq\bigoplus_{\lambda\in\Lambda'}E_\lambda$. 
    Since $\bigoplus_{\lambda\in\Lambda'}E_\lambda$ is gr-injective by \Cref{prop: prod de gr-inj eh gr-inj}, it follows from \Cref{teo: baer} that $h$ extends to a $h'\in\HOM_R(R,E)_\gamma$.

    Conversely, suppose that every direct sum of gr-injective right $R$-modules is also gr-injective.
    Fix $e\in\Gamma_0$ and let $U_1\subseteq U_2\subseteq U_3\subseteq\cdots$ be an ascending chain of graded submodules of $R(e)$. 
    For each $n\geq 1$, let $E_n$ be a gr-injective $\Gamma$-graded right $R$-module such that $\frac{R(e)}{U_n}\subgr E_n$. Such $E_n$ exists by \Cref{teo: todo modulo eh subgr de um gr-inj}.
    Set $U:=\bigcup_{n=1}^\infty U_n\subgr R(e)$ and $E:=\bigoplus_{n=1}^\infty E_n$.
    We define $h\in\Homgr(U,E)$ by $h(u)=(u+U_n)_{n\geq1}$ for each $u\in U$.
    Since $E$ is gr-injective, there exists $h'\in\Homgr(R(e),E)$ extending $h$. 
    Choose $m\in\mathbb{N}$ such that the $m$-th coordinate of $h'(1_e)$ is zero. 
    Then
    \[u\in U\implies u+U_m=h(u)_m=h'(u)_m=h'(1_e)_mu=0\implies u\in U_m.\]
    Thus, $U=U_m$, proving that $R(e)$ is gr-noetherian.
\end{proof}



\section{Graded Jacobson radical and gr-socle of graded modules}
\label{sec: gr-Jacobson gr-socle}

In this section, we present a general analysis of the graded Jacobson radical and of the gr-socle
of groupoid graded modules, which satisfy certain dual properties.  Our exposition is divided into two parts. In the first one, we present  basic properties of these concepts. In the second part, we characterize the graded Jacobson radical in terms of gr-superfluous submodules and the gr-socle in terms of the gr-essential submodules and show some consequences of that.


\subsection{General properties}

Our exposition of the results about $\radgr(M)$ and $\socgr(M)$ is strongly inspired by \cite[\S8--\S10]{AndersonFuller}.

We begin by defining the main objects of the section.

\begin{defi}\label{def:Jacobsonradicalwith1}
	Let $R$ be a $\Gamma$-graded ring and $M$ a $\Gamma$-graded right  $R$-module. 
    \begin{enumerate}[\rm(1)]
        \item We say that a graded submodule $N$ of $M$ is \emph{gr-maximal} if $N$ is maximal in the set $\{X\subgr M\colon X\neq M\}$. 

\item Recall that a graded submodule $N$ of $M$ is said to be \emph{gr-simple} if $N$ is minimal in the set $\{X\subgr M\colon X\neq 0\}$.
        
        \item If $M$ has a gr-maximal graded submodule, we define the \emph{graded Jacobson radical of $M$}, denoted by $\radgr(M)$, as the intersection of all gr-maximal graded submodules of $M$. If $M$ has no gr-maximal graded submodules, we set $\radgr(M) := M$.
        
        \item Dually, we define the \emph{gr-socle}  of $M$, denoted by $\soc_{\rm gr}(M)$, as the sum of all gr-simple graded submodules of $M$. If $M$ has no gr-simple graded submodules, we set $\soc_{\rm gr}(M):=0$.
    \end{enumerate}
\end{defi}

Easy but important observations about gr-maximal graded submodules and the gr-socle are the following.

\begin{obs}\label{obs:maximal and soc}
Let $R$ be a $\Gamma$-graded ring and $M$ a $\Gamma$-graded right $R$-module.
\begin{enumerate}[(1)]
    \item Suppose that $N\subgr M$. Then $N$ is gr-maximal in $M$ if and only if there exists $e \in \Gamma_0$ such that $N(e)$ is gr-maximal in $M(e)$ and $N(f) = M(f)$ for all $f \in \Gamma_0 \setminus \{e\}$. 
    \item Dually, $N$ is gr-simple if and only if there exists $e \in \Gamma_0$ such that $N(e)$ is gr-simple and $N(f) = 0$ for all $f \in \Gamma_0 \setminus \{e\}$.
    \item By definition, $\soc_{\rm gr}(M)$ is the largest gr-semisimple submodule of $M$. 
    \item  $\socgr(M)$ is a direct sum of gr-simple submodules of $M$ by \cite[Proposition 56]{CLP}.
\end{enumerate}
\end{obs}

For a gr-semisimple module, it is easy to compute the graded Jacobson radical and the gr-socle.

\begin{exems}\label{exms: gr-soc and gr-rad of gr-semisimple}
Let $R$ be a $\Gamma$-graded ring and $M$ a $\Gamma$-graded  right $R$-module.  
\begin{enumerate}[\rm (1)]
    \item $M$ is gr-semisimple if and only if $\socgr(M)=M$. Indeed, by definition, $M$ is gr-semisimple if and only if $M$ is a sum of gr-simple submodules. 
    \item If $M$ is gr-semisimple, then $\radgr(M)=0$. Indeed, suppose that $M=\bigoplus_{i\in I}S_i$, where each $S_i$, $i\in I$, is a gr-simple graded submodule of $M$. For each $i_0\in I$, $N_{i_0}:=\bigoplus_{i\in I\setminus\{i_0\}}S_i$ is a gr-maximal graded submodule of $M$ and $\bigcap_{i\in I}N_i=\{0\}$. 
\end{enumerate}
\end{exems}

Now we present a result that will be used throughout this section.

\begin{lema}
\label{lem: rad e soc via ker e im}
    Let $R$ be a $\Gamma$-graded ring and $M$ a $\Gamma$-graded right $R$-module. The following assertions hold:
    \begin{enumerate}[\rm (1)]
        \item $\radgr(M)=\bigcap\{\ker g: g\in\Homgr(M,S) \text{ and $S$ is a gr-simple $R$-module}\}$;
        \item $\socgr(M)=\sum\{\ima g: g\in\Homgr(S,M) \text{ and $S$ is a gr-simple $R$-module}\}$.
    \end{enumerate}
\end{lema}

\begin{proof}
    (1) Let $m\in\h(M)$.
    If $m\notin\radgr(M)$, then there exists a gr-maximal graded submodule $N$ of $M$ such that $m\notin N$. 
    In this case, $M/N$ is gr-simple and $m\notin\ker g$, where $g:M\to M/N$ is the canonical projection.
    Conversely, if there exists a gr-simple $\Gamma$-graded right $R$-module $S$ and $g\in\Homgr(M,S)$ such that $m\notin\ker g$, then $g$ induces a gr-isomorphism $M/\ker g\to S$, and hence $\ker g$ is a gr-maximal graded submodule of $M$ that does not contain $m$.

    (2) $(\subseteq)$: This follows since $\socgr(M)$ is a sum of gr-simple $R$-modules.
    $(\supseteq)$: If $S$ is a gr-simple $\Gamma$-graded right $R$-module and $0\neq g\in\Homgr(S,M)$, then, by gr-simplicity, $\ima g\isogr S$ is gr-simple, and it follows that $\ima g\subseteq\socgr(M)$.
\end{proof}

By \Cref{obs:maximal and soc}(4), $\socgr(M)$ is a direct sum of gr-simple $R$-modules. 
For $\radgr(M)$, we have the following result.

\begin{coro}
\label{coro: M/rad contido num prod de gr-simp}
    Let $R$ be a $\Gamma$-graded ring and $M$ a $\Gamma$-graded right $R$-module. 
    Then there exist a family $\{S_i:i\in I\}$ of gr-simple $\Gamma$-graded right $R$-modules and an injective gr-homomorphism $M/\radgr(M)\to\prod^{\rm gr}_{i\in I}S_i$. 
\end{coro}

\begin{proof}
    Let $\mathcal{S}$ be a complete set of gr-simple $\Gamma$-graded right $R$-modules up to gr-isomorphism; that is, each gr-simple $\Gamma$-graded right $R$-module is gr-isomorphic to one and only one element of $\mathcal{S}$.
    For each $S\in\mathcal{S}$, we have a gr-homomorphism from $M$ to the graded direct product $S^{\Homgr(M,S)}$ that sends $m$ to the tuple $(g(m))_g$.
    Hence, we obtain a gr-homomorphism $M\to \prod_{S\in\mathcal{S}}^{\rm gr}S^{\Homgr(M,S)}$, whose kernel is $\radgr(M)$ by \Cref{lem: rad e soc via ker e im}(1).
\end{proof}

\begin{coro}
\label{coro: rad e soc sao prerad}
    Let $R$ be a $\Gamma$-graded ring and $M, N$ be $\Gamma$-graded right $R$-modules. The following assertions hold:
    \begin{enumerate}[\rm (1)]
        \item If either $g\in\Homgr(M,N)$ or $g\in\HOM(M,N)_\gamma$ for some $\gamma\in\Gamma$, then 
        \[g(\radgr(M))\subseteq\radgr(N)\quad\text{ and }\quad g(\socgr(M))\subseteq\socgr(N).\] 
        \item $\radgr(M)$ and $\socgr(M)$ are $\Gamma$-graded $(\END_R(M),R)$-bimodules.
        \item $M\cdot\radgr(R_R)\subseteq\radgr(M)$ and $M\cdot\socgr(R_R)\subseteq\socgr(M)$.
        \item $\radgr(R_R)$ and $\socgr(R_R)$ are graded ideals of $R$.
    \end{enumerate}
\end{coro}

\begin{proof}
    (1) Suppose that $g\in\HOM(M,N)_\gamma$ for some $\gamma\in\Gamma$.
    
    Let $S$ be a gr-simple $\Gamma$-graded right $R$-module and $g'\in\Homgr(N,S)$. 
    Then $g'g\in\Homgr(M,S(\gamma))$.
    Since $S$ is gr-simple, then $S(\gamma)$ is either zero or gr-simple \cite[Lemma 5.1]{Artigoarxiv}. It follows from \Cref{lem: rad e soc via ker e im}(1) that $\radgr(M)\subseteq \ker (g'g)$, and thus $g(\radgr(M))\subseteq \ker g'$.
    Again by \Cref{lem: rad e soc via ker e im}(1), we obtain $g(\radgr(M))\subseteq\radgr(N)$.

    Now, let $g'\in\Homgr(S,M)$, where $S$ is a gr-simple $\Gamma$-graded right $R$-module, so that $gg'|_{S(d(\gamma))}\in\Homgr(S(\gamma^{-1}),N)$.
    Since $S(\gamma^{-1})$ is either zero or gr-simple by \cite[Lemma 5.1]{Artigoarxiv}, we have from \Cref{lem: rad e soc via ker e im}(2) that $g(\ima g')= \ima (gg')\subseteq\socgr(N)$.
    Hence, by \Cref{lem: rad e soc via ker e im}(2) again, we conclude that $g(\socgr(M))\subseteq\socgr(N)$.

    The case when $g\in\Homgr(M,N)$ is similar, with the difference that we do not need to work with shifts of $S$. 

    (2) Follows from (1).

    (3) Apply (1) to each $m \in M$ and $g \in \HOM_R(R, M)$ given by $g(a) = ma$ for all $a \in R$.

    (4) Follows from (3).
\end{proof}

\begin{coro}
\label{coro: rad da soma direta}
    Let $R$ be a $\Gamma$-graded ring and $\{M_i:i\in I\}$ be a family of $\Gamma$-graded right $R$-modules. Then
    \[\radgr\left(\bigoplus_{i\in I}M_i\right)=\bigoplus_{i\in I}\radgr(M_i)\quad\text{ and }\quad \socgr\left(\bigoplus_{i\in I}M_i\right)=\bigoplus_{i\in I}\socgr(M_i).\]
\end{coro}

\begin{proof}
    Set $M := \bigoplus_{i \in I} M_i$, and let $\iota_i : M_i \to M$ and $p_i : M \to M_i$ be the canonical injection and projection maps, respectively, for each $i \in I$. 
    Since each $\iota_i$ and $p_i$ is a gr-homomorphism, it follows from \Cref{coro: rad e soc sao prerad}(1) that
    \[\radgr(M)=\sum_{i\in I}\iota_i\pi_i(\radgr(M))\subseteq \sum_{i\in I}\iota_i(\radgr(M_i))\subseteq\radgr(M),\]
    and thus $\radgr(M)=\sum_{i\in I}\iota_i(\radgr(M_i))=\bigoplus_{i\in I}\radgr(M_i)$.

    The same argument applies to the gr-socle.
\end{proof}

The next result is specific of the groupoid graded situation and it helps to simplify the proof of other results.

\begin{prop}
\label{prop: rad M(e)}
    Let $R$ be a $\Gamma$-graded ring, $M$ a $\Gamma$-graded right $R$-module, and $\sigma\in\Gamma$. The following assertions hold:
    \begin{enumerate}[\rm (1)]
        \item $\radgr(M)=\bigoplus_{e\in\Gamma_0}\radgr(M(e))$ and $(\radgr(M))(\sigma)=\radgr(M(\sigma))$.
        \item $\socgr(M)=\bigoplus_{e\in\Gamma_0}\socgr(M(e))$ and $(\socgr(M))(\sigma)=\socgr(M(\sigma))$.
    \end{enumerate}
\end{prop}

\begin{proof}
    (1) By \Cref{coro: rad da soma direta}, we have $\radgr(M)=\bigoplus_{e\in\Gamma_0}\radgr(M(e))$. 
    Thus, $(\radgr(M))(e)=\radgr(M(e))$ for each $e\in\Gamma_0$.
    By \cite[Lemma 2.6]{Artigoarxiv}, $M(\sigma)$ and $M(r(\sigma))$ are equal as $R$-modules and have the same homogeneous components, although labeled with different degrees. Hence, $\radgr(M(\sigma))=\radgr(M(r(\sigma)))$ as sets.
    By what we have just proved, $\radgr(M(r(\sigma)))=\radgr(M)(r(\sigma))$. Again, it follows  from \cite[Lemma 2.6]{Artigoarxiv} that $\radgr(M)(r(\sigma))=(\radgr(M))(\sigma)$ as sets. Therefore, 
$(\radgr(M))(\sigma)=\radgr(M(\sigma))$.

    (2) It is analogous to (1).
\end{proof}

\begin{lema}
\label{lem: rad(M/rad)=0}
    Let $R$ be a $\Gamma$-graded ring, $M$ a $\Gamma$-graded right  $R$-module and $N\subgr M$. The following assertions hold:
    \begin{enumerate}[\rm (1)]
        \item $\radgr(M)$ is the smallest graded submodule $M'$ of $M$ such that \mbox{$\radgr(M/M')=0$}.
        \item $\socgr(M)$ is the greatest graded submodule $M'$ of $M$ such that $\socgr(M')=M'$.
    \end{enumerate}
\end{lema}

\begin{proof}
    (1) Let $M'\subgr M$ be such that $\radgr(M/M')=0$. 
    If $\pi:M\to M/M'$ is the canonical projection, then by \Cref{coro: rad e soc sao prerad}(1), $\pi(\radgr(M))\subseteq\radgr(M/M')=0$, and it follows that $\radgr(M)\subseteq\ker\pi=M'$.
    To see that \mbox{$\radgr(M/\radgr(M))=0$}, let $\varphi:M/\radgr(M)\to\prod_{i\in I}^{\rm gr}S_i$ be the injective gr-homomorphism given by \Cref{coro: M/rad contido num prod de gr-simp}.
    If $p_j:\prod_{i\in I}^{\rm gr}S_i\to S_j$ is the canonical projection, then 
    $$p_j\varphi(\radgr(M/\radgr(M)))\subseteq \radgr(S_j)=0$$ 
    by \Cref{coro: rad e soc sao prerad}(1) and \Cref{exms: gr-soc and gr-rad of gr-semisimple}(2).
    Thus, $\radgr(M/\radgr(M))\subseteq\ker\varphi=0$.

    (2) Let $M'\subgr M$. Then $\socgr(M')=M'$ if and only if $M'$ is gr-semisimple. By \Cref{obs:maximal and soc}(3), $\socgr(M)$ is the greatest gr-semisimple graded submodule of $M$. 
\end{proof}

\begin{coro}
\label{coro: quando g(rad)=rad}
    Let $R$ be a $\Gamma$-graded ring and $M, N$ be $\Gamma$-graded right $R$-modules. 
    \begin{enumerate}[\rm (1)]
        \item Suppose that $g\in\Homgr(M,N)$ is surjective and $\ker g\subseteq\radgr(M)$. 
        Then $g(\radgr(M))=\radgr(N)$. 
        In particular, if $M'\subgr \radgr(M)$, then $$\radgr(M/M')=\radgr(M)/M'.$$
        \item Suppose that $g\in\Homgr(M,N)$ is injective and $\socgr(N)\subseteq\ima g$. 
        Then $g(\socgr(M))=\socgr(N)$.
        In particular, if $N'\subgr N$, then $$\socgr(N')=N'\cap\socgr(N).$$
    \end{enumerate}
\end{coro}

\begin{proof}
    (1) By \Cref{coro: rad e soc sao prerad}(1), we have $g(\radgr(M))\subseteq\radgr(N)$. 
    Now we proceed to show the reverse inclusion. Since $g$ is surjective and $\ker g\subseteq\radgr(M)$, it follows that $g$ induces a gr-isomorphism $M/\radgr(M)\to N/g(\radgr(M))$.
    Therefore, by \Cref{lem: rad(M/rad)=0}(1), we obtain that $\radgr(N/g(\radgr(M)))=0$ and $\radgr(N)\subseteq g(\radgr(M))$. 
    For the last part of the statement, it suffices to consider the canonical projection $g:M\to M/M'$.

    (2) By \Cref{coro: rad e soc sao prerad}(1), $g(\socgr(M))\subseteq \socgr(N)$. Now we proceed to show the reverse inclusion.
    Since $\socgr(N)\subseteq\ima g$ and $g$ is injective, it follows that $g$ induces a gr-isomorphism $g^{-1}(\socgr(N))\to \socgr(N)$.
    Then, by \Cref{lem: rad(M/rad)=0}(2), we obtain that $\socgr(g^{-1}(\socgr(N)))=g^{-1}(\socgr(N))$ and $g^{-1}(\socgr(N))\subseteq\socgr(M)$.
    Therefore, $\socgr(N)\subseteq g(\socgr(M))$. 
    To prove the last statement consider the inclusion $g:N'\cap\socgr(N)\to N'$. Observe that $\socgr(N')\subseteq \socgr(N)$ by definition. Hence $\socgr(N')\subseteq \ima g$.  Now $N'\cap\socgr(N)$ is gr-semisimple  because any graded submodule of a gr-semisimple module is gr-semisimple \cite[Corollary 54]{CLP}. Hence $N'\cap\socgr(N)=\socgr(N'\cap\socgr(N))=g(\socgr(N'\cap\socgr(N)))=\socgr(N')$.
\end{proof}

\subsection{Gr-superfluous and gr-essential submodules}

We begin by defining the dual concepts of gr-superfluous and gr-essential submodules.

\begin{defi}
    Let $R$ be a $\Gamma$-graded ring, $M$ a $\Gamma$-graded right  $R$-module, and $N\subgr M$. 
    \begin{enumerate}[\rm (1)]
       \item We say that $N$ is \emph{gr-superfluous} in $M$ if for every $X\subgr M$ 
       \[N+X=M\quad \textrm{implies}\quad X=M.\]
        \item We say that $N$ is \emph{gr-essential} in $M$ if for every $X\subgr M$ 
        \[N\cap X=0\quad \textrm{implies}\quad X=0.\]
    \end{enumerate} 
\end{defi}

Some elementary examples and properties of these submodules are as follows.

\begin{exems}\label{exems: gr-superfluous gr-essential}
Let $R$ be a $\Gamma$-graded ring and $M$ be a $\Gamma$-graded right  $R$-module.
\begin{enumerate}[\rm (1)]
    \item $\{0\}$ is gr-superfluous in $M$ and $M$ is gr-essential in $M$. 
    On the other hand, if $M\neq\{0\}$, then $\{0\}$ is not gr-essential in $M$ and $M$ is not gr-superfluous in $M$.
    \item If $N$ is gr-superfluous in $M$ and $L\subgr N$, then $L$ is gr-superfluous in $M$.
    \item If $N$ is is gr-essential in $M$ and $N\subseteq L\subgr M$, then $L$ is gr-essential in $M$.    
    \item Finite sums of gr-superfluous submodules of $M$ are gr-superfluous in $M$.
    \item Finite intersections of gr-essential submodules of $M$ are gr-essential in $M$.
\end{enumerate}    
\end{exems}

The next useful result shows that gr-essential submodules appear very often.

\begin{lema}\label{lem: direct summand of an essential module}
  Let $R$ be a $\Gamma$-graded ring and $M$ be a $\Gamma$-graded right  $R$-module. Any graded submodule of $M$ is a graded direct summand of a gr-essential submodule of $M$.  
\end{lema}

\begin{proof}
    Let $N\subgr M$. By Zorn's Lemma, the set   $$\mathcal{X}=\{X\subgr M:N\cap X=0\}$$ has a maximal element $N'$.
    If $X\subgr M$ and $(N\oplus N')\cap X=0$, then $N\cap (N'\oplus X)=0$, and we get $X=0$ by the maximality of $N'$ in $\mathcal{X}$.
    Hence, $N\oplus N'$ is gr-essential in $M$.
\end{proof}

The next result is specific of the groupoid graded situation and it helps to simplify the proof of other results.

\begin{lema}
\label{lem: M superfluo sse todo M(e)}
    Let $R$ be a $\Gamma$-graded ring, $M$ be a $\Gamma$-graded right $R$-module, and $N\subgr M$. The following assertions hold:
    \begin{enumerate}[\rm (1)]
        \item $N$ is gr-superfluous in $M$ if and only if $N(e)$ is gr-superfluous in $M(e)$ for each $e\in\Gamma_0$.
        \item $N$ is gr-essential in $M$ if and only if $N(e)$ is gr-essential in $M(e)$ for each $e\in\Gamma_0$.
    \end{enumerate}    
\end{lema}

\begin{proof}
    (1) If $N(e)$ is gr-superfluous in $M(e)$ for every $e\in\Gamma_0$, and $X\subgr M$ is such that $N+X=M$, then $N(e)+X(e)=M(e)$ implies $X(e)=M(e)$ for each $e\in\Gamma_0$, and therefore $X=M$, whence $N$ is gr-superfluous in $M$.
    Conversely, suppose that $N$ is gr-superfluous in $M$ and let $e\in\Gamma_0$.
    If $X\subgr M(e)$ is such that $N(e)+X=M(e)$, then $N+\left(X\oplus\left(\bigoplus_{f\in\Gamma_0\setminus\{e\}}M(f)\right)\right)=M$ implies that $X\oplus \left(\bigoplus_{f\in\Gamma_0\setminus\{e\}}M(f)\right)=M$, and thus $X=M(e)$.
    Therefore, $N(e)$ is gr-superfluous in $M(e)$.

    (2) The proof is similar to (1).
\end{proof}

An interesting reason for studying gr-superfluous and gr-essential submodules is the following.

\begin{prop}\label{prop: gr-ess  injective  gr-superfl surjective}
    Let $R$ be a $\Gamma$-graded ring, $M, N$ be $\Gamma$-graded right $R$-modules and $g\in\Homgr(M,N)$. Let $M'\subgr M$ and $N'\subgr N$ be such that $g(M')\subseteq N'$. The following statements hold.
    \begin{enumerate}[\rm (1)]
        \item Suppose that $M'$ is gr-essential in $M$. Then $g$ is injective if and only if the induced gr-homomorphism $g'\colon M'\rightarrow N'$, $x\mapsto g(x)$, is injective.
        \item  Suppose that $N'$ is gr-superfluous in $N$. Then $g$ is surjective if and only if the induced gr-homomorphism $\overline{g}\colon M/M'\rightarrow N/N'$, $x+M'\mapsto g(x)+N'$, is surjective.        
    \end{enumerate}
\end{prop}

\begin{proof}
(1) If $g$ is injective, it is straightforward to show that $g'$ is injective. 
Suppose now that  $g'$ is injective. It implies that $\ker (g)\cap M'=\{0\}$. Since $M'$ is gr-essential, we get $\ker (g)=\{0\}$.

(2) If $g$ is surjective, it is straightforward to show that $\overline{g}$ is surjective. 
Suppose now that  $\overline{g}$ is surjective.  It implies that  $N=\ima (g)+N'$. Since $N'$ is gr-superfluous, we obtain $N=\ima (g)$.
\end{proof}

Now we give equivalent definitions of $\radgr(M)$ and $\socgr(M)$ that do not depend on the existence of gr-maximal or gr-simple submodules of $M$.

\begin{prop}
\label{prop: (in)essential}
    Let $R$ be a $\Gamma$-graded ring and $M$ a $\Gamma$-graded right $R$-module. The following assertions hold.
    \begin{enumerate}[\rm (1)]
        \item $\radgr(M)=\sum\{N\subgr M: N \text{ is gr-superfluous in } M\}$;
        \item $\socgr(M)=\bigcap\{N\subgr M: N \text{ is gr-essential in } M\}$.
    \end{enumerate}
\end{prop}

\begin{proof}
    (1) $(\subseteq)$: Let $m\in\h(\radgr(M))$. 
    We will prove that $mR$ is gr-superfluous in $M$. 
    Suppose that 
    \begin{equation}\label{eq: rad is sum of superfluous}
        mR+X=M\quad \textrm{for some } X\subgr M.
    \end{equation}  
    Assume by contradiction that $m\notin X$. 
    By Zorn's Lemma, the set $$\{Y\subgr M:X\subseteq Y\text{ and }m\notin Y\}$$ has a maximal element $M'$. 
    From \eqref{eq: rad is sum of superfluous}, we get that $M'$ is a gr-maximal graded submodule of $M$. Since $m\in\radgr(M)$, then $m\in M'$, a contradiction.
    Therefore, $m\in X$, and it follows that $X=mR+X=M$.

    $(\supseteq)$: If $M'$ is a gr-maximal graded submodule of $M$ and $N$ is a gr-superfluous graded submodule of $M$, then $M'\neq M$ implies $N+M'\neq M$, and it follows that $M'\supseteq N$ by gr-maximality of $M'$. 
    Therefore, $\radgr(M)$ contains every gr-superfluous graded submodule of $M$.

    (2) $(\subseteq)$: If $S$ is a gr-simple graded submodule of $M$ and $N$ is a gr-essential graded submodule of $M$, then $S\neq 0$ implies $S\cap N\neq0$, whence $S\subseteq N$ by gr-simplicity.
    Thus, $\socgr(M)$ is contained in every gr-essential graded submodule of $M$.

    $(\supseteq)$: Suppose that $m\in\h(M)$ is contained in every gr-essential graded submodule of $M$. 
    We will show that $mR$ is a gr-semisimple module by showing that every graded submodule of $mR$ is a direct summand of $mR$ \cite[Proposition 56]{CLP}. 
    Let $N\subgr mR$.
    By Lemma~\ref{lem: direct summand of an essential module}, there exists $N'\subgr M$ such that 
    $N\oplus N'$ is gr-essential in $M$.  Thus,  $m\in N\oplus N'$ because of the way $m$ was chosen.
    Therefore, $mR=N\oplus (N'\cap mR)$.
\end{proof}

Our next result presents a description of the graded Jacobson radical and the gr-socle   in terms of homogeneous elements.

\begin{prop}
\label{prop: rad M_sigma}
    Let $R$ be a $\Gamma$-graded ring, $M$ a $\Gamma$-graded right $R$-module, and $\sigma\in\Gamma$. The following assertions hold:
    \begin{enumerate}[\rm (1)]
        \item $\radgr(M)_\sigma=\{m\in M_\sigma: mR \text{ is gr-superfluous in } M\}$.
        \item $\socgr(M)_\sigma=\{m\in M_\sigma: mR \text{ is gr-semisimple}\}$.
    \end{enumerate}
\end{prop}

\begin{proof}
    (1) Let $m\in \radgr(M)_\sigma$ for some $\sigma \in \Gamma$. By \Cref{prop: (in)essential}, there exist $N_1,\dotsc,N_r$ gr-superfluous submodules of $M$ such that $m\in N_1+\cdots+N_r$. By \Cref{exems: gr-superfluous gr-essential}(4), $N_1+\cdots+N_r$ is gr-superfluous and $mR\subgr N_1+\cdots+N_r$ is gr-superfluous too.  
    The inclusion $\supseteq$ follows from \Cref{prop: (in)essential}.

    (2) Let $m\in\socgr(M)_\sigma$ for some $\sigma \in \Gamma$. Now $mR$ is a graded submodule of $\socgr(M)$. The result now follows because any graded submodule of a gr-semisimple module is gr-semisimple \cite[Corollary 54]{CLP}.
    The other inclusion is a consequence of \Cref{obs:maximal and soc}(3).
\end{proof}

\begin{coro}
\label{coro: cond suf para rad ser superfluo}
    Let $R$ be a $\Gamma$-graded ring and $M$ be a $\Gamma$-graded right $R$-module. The following assertions hold:
    \begin{enumerate}[\rm (1)]
        \item Suppose that every proper graded submodule of $M$ is contained in a gr-maximal graded submodule of $M$. Then $\radgr(M)$ is gr-superfluous in $M$.
        \item $\socgr(M)$ is gr-essential in $M$ if and only if every nonzero graded submodule of $M$ contains a gr-simple $R$-module.
        \item If $N$ is a gr-maximal graded submodule of $M$, then $\radgr(M/N)=0$.
        \item If $N$ is a gr-essential graded submodule of $M$, then $\socgr(N)=\socgr(M)$.
        \item Suppose that $N$ is a gr-superfluous graded submodule of $M$  and \mbox{$\radgr(M/N)=0$.} 
        Then $N=\radgr(M)$. 
        \item Suppose that $N$ is a gr-essential graded submodule of $M$  and $\socgr(N)=N$.
        Then $N=\socgr(M)$. 
    \end{enumerate}
\end{coro}

\begin{proof}
    (1) If $X\subgr M$ and $X\neq M$, then taking $M'$ a gr-maximal graded submodule of $M$ such that $X\subseteq M'$, we have $\radgr(M)+X\subseteq M'\subsetneq M$. 

    (2) By \Cref{coro: quando g(rad)=rad}(2), we have $\socgr(X)=\socgr(M)\cap X$ for each graded submodule $X$ of $M$.
    Therefore, $\socgr(M)$ is gr-essential in $M$ if and only if $\socgr(X)\neq0$ for every nonzero graded submodule $X$ of $M$.

    (3) If $N$ is a gr-maximal graded submodule of $M$, then $M/N$ is gr-simple, and $\radgr(M/N)=0$ by \Cref{exms: gr-soc and gr-rad of gr-semisimple}(2).
    
    (4) If $N$ is a gr-essential graded submodule of $M$, then $N\cap S=S$ for every gr-simple submodule of $M.$ Therefore, $N$ contains any gr-simple submodule of $M$ and $\socgr(N)=\socgr(M)$.

    (5) By \Cref{prop: (in)essential}(1), $N\subseteq\radgr(M)$, and by \Cref{lem: rad(M/rad)=0}(1), \mbox{$\radgr(M)\subseteq N$.}

    (6) Follows from \Cref{prop: (in)essential}(2) and \Cref{lem: rad(M/rad)=0}(2).
\end{proof}

The converse of \Cref{coro: cond suf para rad ser superfluo}(1) does not hold even in the classical case, see \cite[Exercise 9.4]{AndersonFuller}.

Similarly to the ungraded case \cite[Theorem 10.4]{AndersonFuller}, we can characterize finitely (gr-co)generated modules via the graded Jacobson radical and the gr-socle. As a consequence, we can also characterize $\Gamma_0$-finitely (gr-co)generated modules.

\begin{prop}
\label{prop: carac fini (co)gerado via rad e soc}
    Let $R$ be a $\Gamma$-graded ring and $M$ be a $\Gamma$-graded right $R$-module. The following assertions hold:
    \begin{enumerate}[\rm (1)]
        \item $M$ is finitely generated if and only if $M/\radgr(M)$ is finitely generated and $\radgr(M)$ is gr-superfluous in $M$.
        \item $M$ is finitely gr-cogenerated if and only if $\socgr(M)$ is finitely gr-cogenerated and $\socgr(M)$ is gr-essential in $M$. 
        \item $M$ is $\Gamma_0$-finitely generated if and only if $M/\radgr(M)$ is $\Gamma_0$-finitely generated and $\radgr(M)$ is gr-superfluous in $M$.
        \item $M$ is $\Gamma_0$-finitely gr-cogenerated if and only if $\socgr(M)$ is $\Gamma_0$-finitely gr-cogenerated and $\socgr(M)$ is gr-essential in $M$. 
    \end{enumerate}
\end{prop}

\begin{proof}
    (1) If $M/\radgr(M)=\sum_{i=1}^n(x_i+\radgr(M))R$ for some $x_1,\dots,x_n\in\h(M)$, and $\radgr(M)$ is gr-superfluous in $M$, then $M=\radgr(M)+x_1R+\cdots+x_nR$, and it follows that $M=x_1R+\cdots+x_nR$.
    Now suppose that $M$ is finitely generated.
    Then it is clear that $M/\radgr(M)$ is finitely generated.
    Assume that $X\subgr M$ is such that $\radgr(M)+X=M$.
    Since $M$ is finitely generated,  there exist gr-superfluous graded submodules $N_1,\dots, N_s$ of $M$ such that $N_1+\cdots+N_s+X=M$ by \Cref{prop: (in)essential}(1).
    By \Cref{exems: gr-superfluous gr-essential}(4), $N_1+\cdots+N_s$ is gr-superfluous. Hence $X=M$ and $\radgr(M)$ is gr-superfluous in~$M$.

    (2) Suppose that $\socgr(M)$ is finitely gr-cogenerated, $\socgr(M)$ is gr-essential in $M$, and $\{M_i:i\in I\}$ is a family of graded submodules of $M$ such that \mbox{$\bigcap_{i\in I}M_i=0$}.
    Since $\socgr(M)$ is finitely gr-cogenerated, there exist $i_1,\dots,i_n\in I$ such that \linebreak $\bigcap_{t=1}^n(\socgr(M)\cap M_{i_t})=0$.
    The gr-essentiality of $\socgr(M)$ implies that \mbox{$\bigcap_{t=1}^n M_{i_t}=0$}.
    Conversely, suppose that $M$ is finitely gr-cogenerated.
    It is easy to see that every graded submodule of $M$ is finitely gr-cogenerated.
    Assume that $X\subgr M$ is such that $\socgr(M)\cap X=0$.
    Since $M$ is finitely gr-cogenerated, there exist gr-essential graded submodules $N_1,\dots, N_s$ of $M$ such that $N_1\cap\cdots\cap N_s\cap X=0$  by \Cref{prop: (in)essential}(2). By   \Cref{exems: gr-superfluous gr-essential}(5),  $N_1\cap\cdots\cap N_s$ is gr-essential. Hence  $X=0$ and $\socgr(M)$ is gr-essential in $M$.

    (3) By (1), $M(e)$ is finitely generated for all $e\in \Gamma_0$ if and only if $M(e)/\radgr(M(e))$ is finitely generated and $\radgr(M(e))$ is gr-superfluous in $M(e)$ for all $e\in \Gamma_0$. The result now follows from \Cref{prop: rad M(e)}(1) and \Cref{lem: M superfluo sse todo M(e)}(1).

    (4) By (2), $M(e)$ is finitely gr-cogenerated for all $e\in \Gamma_0$  if and only if $\socgr(M(e))$ is finitely gr-cogenerated and $\socgr(M(e))$ is gr-essential in $M(e)$ for each $e\in\Gamma_0$. The result now follows from \Cref{prop: rad M(e)}(2) and \Cref{lem: M superfluo sse todo M(e)}(2).
\end{proof}

\begin{coro}\label{coro: gr-noetherian radgr superfluous}
    Let $R$ be a $\Gamma$-graded ring and $M$ be a $\Gamma$-graded right $R$-module. The following assertions hold: 
    \begin{enumerate}[\rm (1)]
        \item If $M$ is $\Gamma_0$-noetherian, then $\radgr(M)$ is gr-superfluous in $M$.
        \item If $M$ is $\Gamma_0$-artinian, then $\socgr(M)$ is gr-essential in $M$. 
    \end{enumerate}
\end{coro}

\begin{proof}
   (1) and (2) follow from \Cref{prop: G0-noeth <-> todo submod eh G0-fg} and \Cref{prop: carac fini (co)gerado via rad e soc}.
\end{proof}

\begin{coro}
   Let $R$ be a $\Gamma$-graded ring, $M,N$ be  $\Gamma$-graded right $R$-modules and $g\in \Homgr(M,N)$. The following statements hold.
   \begin{enumerate}[\rm(1)]
       \item Suppose that $M$ is $\Gamma_0$-artinian. Then $g$ is injective if and only if the induced gr-homomorphism $g'\colon \socgr(M)\rightarrow \socgr(N)$, $x\mapsto g(x)$, is injective.
       \item Suppose that $N$ is $\Gamma_0$-noetherian. Then $g$ is surjective if and only if the induced gr-homomorphism $\overline{g}\colon M/\radgr(M)\rightarrow N/\radgr(N)$, $x+\radgr(M)\mapsto g(x)+\radgr(N)$, is surjective. 
   \end{enumerate}
\end{coro}

\begin{proof}
(1) By  \Cref{coro: gr-noetherian radgr superfluous}(2), $\socgr(M)$ is gr-essential in $M$.  
The result now follows from \Cref{coro: rad e soc sao prerad}(1), and \Cref{prop: gr-ess  injective  gr-superfl surjective}(1)  with $M'=\socgr(M)$ and $N'=\socgr(N)$.

(2) By  \Cref{coro: gr-noetherian radgr superfluous}(1), $\radgr(N)$ is gr-superfluous in $N$.  
The result now follows from  \Cref{coro: rad e soc sao prerad}(1), and \Cref{prop: gr-ess  injective  gr-superfl surjective}(2)  with $M'=\radgr(M)$ and $N'=\radgr(N)$.
\end{proof}

By \cite[Lemma 55]{CLP}, every nonzero finitely generated graded module has a gr-maximal graded submodule. By \Cref{obs:maximal and soc}(1), it also holds for $\Gamma_0$-finitely generated graded modules. We provide another proof of this fact together with its dual statement.

\begin{coro}
\label{coro: fin cogerado tem minimal}
    Let $R$ be a $\Gamma$-graded ring and $M$ be a nonzero $\Gamma$-graded right $R$-module. 
    The following assertions hold:
    \begin{enumerate}[\rm (1)]
        \item If $M$ is $\Gamma_0$-finitely generated, then $M$ has a gr-maximal graded submodule.
        \item If $M$ is $\Gamma_0$-finitely gr-cogenerated, then $M$ has a gr-simple graded submodule.
    \end{enumerate}
\end{coro}

\begin{proof}
    (1) By \Cref{prop: carac fini (co)gerado via rad e soc}(3), $\radgr(M)\neq M$ if $M$ is $\Gamma_0$-finitely generated.
    
    (2) By \Cref{prop: carac fini (co)gerado via rad e soc}(4), $\socgr(M)\neq0$ if $M$ is $\Gamma_0$-finitely gr-cogenerated.
\end{proof}

The following result characterizes finitely generated gr-semisimple modules using the graded Jacobson radical and gr-artinianity, see \cite[Proposition 10.15]{AndersonFuller} for the classical case. Moreover, it shows that we can write ``$\socgr(M)$ is finitely generated'' in  items (2) and (4) of \Cref{prop: carac fini (co)gerado via rad e soc}. 

\begin{prop}
\label{prop: carac gr-ss fg}
    Let $R$ be a $\Gamma$-graded ring and $M$ be a $\Gamma$-graded right $R$-module. 
    The following assertions are equivalent:
    \begin{enumerate}[\rm (1)]
        \item $M$ is gr-semisimple and finitely generated.
        \item $M$ is gr-semisimple and gr-noetherian.
        \item $M$ is gr-semisimple and gr-artinian.
        \item $M$ is gr-semisimple and finitely gr-cogenerated. 
        \item $M$ is a finite direct sum of gr-simple graded submodules.
        \item $\radgr(M)=0$ and $M$ is finitely gr-cogenerated.
        \item $\radgr(M)=0$ and $M$ is gr-artinian.
    \end{enumerate}
\end{prop}

\begin{proof}
    $(1)\Leftrightarrow(2)\Leftrightarrow(3)\Leftrightarrow(5)$: This is straightforward since gr-semisimple modules are direct sums of gr-simple modules by \cite[Proposition 52]{CLP}.
    
    $(3)\Rightarrow(4)$: This follows from \Cref{prop: gr-noeth <-> todo submod eh fg}(2).

    $(4)\Rightarrow (6)$: This follows from \Cref{exms: gr-soc and gr-rad of gr-semisimple}(2).

    $(6)\Rightarrow(5)$: By (6) and \Cref{coro: M/rad contido num prod de gr-simp}, there exist a family $\{S_i:i\in I\}$ of gr-simple $\Gamma$-graded right $R$-modules and an injective gr-homomorphism $\varphi:M\to\prod^{\rm gr}_{i\in I}S_i$. 
    Let $\pi_j:\prod^{\rm gr}_{i\in I}S_i\to S_j$ be the canonical projection for each $j\in I$.
    Since $\bigcap_{i\in I}\ker \pi_i=0$ and $M$ is finitely gr-cogenerated, it follows that $\bigcap_{t=1}^n\ker \pi_{i_t}=0$ for some $i_1,\dots i_n\in I$, and therefore we have an injective gr-homomorphism $M\to S_{i_1}\oplus\cdots\oplus S_{i_n}$. 
    Therefore, (5) holds by \cite[Proposition 53(ii)]{CLP}.

    $(3)\Rightarrow(7)$: This follows from \Cref{exms: gr-soc and gr-rad of gr-semisimple}(2).

    $(7)\Rightarrow(6)$: This follows from \Cref{prop: gr-noeth <-> todo submod eh fg}(2).
\end{proof}

Applying the foregoing result to $M(e)$ for each $e\in\Gamma_0$, we obtain the following corollary.

\begin{coro}
\label{coro: carac gr-ss fg}
    Let $R$ be a $\Gamma$-graded ring and $M$ be a $\Gamma$-graded right $R$-module. 
    The following assertions are equivalent:
    \begin{enumerate}[\rm (1)]
        \item $M$ is gr-semisimple and $\Gamma_0$-finitely generated.
        \item $M$ is gr-semisimple and $\Gamma_0$-noetherian.
        \item $M$ is gr-semisimple and $\Gamma_0$-artinian.
        \item $M$ is gr-semisimple and $\Gamma_0$-finitely gr-cogenerated. 
        \item $M(e)$ is a finite direct sum of gr-simple graded submodules for each $e\in\Gamma_0$.
        \item $\radgr(M)=0$ and $M$ is $\Gamma_0$-finitely gr-cogenerated.
        \item $\radgr(M)=0$ and $M$ is $\Gamma_0$-artinian.\qed
    \end{enumerate}
\end{coro}

Before stating the next corollary, we recall the following definition from \cite[p. 49]{Artigoarxiv}:

\begin{defi}
    We say that a $\Gamma$-graded right $R$-module $S$ is $\Gamma_0$-simple if $S(e)$ is a gr-simple $R$-module for each $e \in \Gamma'_0(S)$.
\end{defi}

\begin{coro}
\label{coro: carac fort gr-ss}
    Let $R$ be a $\Gamma$-graded ring and $M$ be a $\Gamma$-graded right $R$-module. 
    The following assertions are equivalent:
    \begin{enumerate}[\rm (1)]
        \item $M$ is gr-semisimple and there exists $n\in\mathbb{N}$ such that $M(e)$ is generated by at most $n$ homogeneous elements for each $e\in\Gamma_0$.
        \item $M$ is gr-semisimple and strongly $\Gamma_0$-noetherian.
        \item $M$ is gr-semisimple and strongly $\Gamma_0$-artinian.
        \item $M$ is a finite direct sum of $\Gamma_0$-simple graded submodules.
        \item $\radgr(M)=0$ and $M$ is strongly $\Gamma_0$-artinian.
    \end{enumerate}
\end{coro}

\begin{proof}
    $(1)\Rightarrow(4)$: (1) implies that, for each $e\in\Gamma_0$, we have a decomposition in direct sum of gr-simple submodules $M(e)=\bigoplus_{i=1}^{n_e}S_{i,e}$ with $n_e\leq n$.
    For each $e\in\Gamma_0$ and $1\leq i\leq n$, set $S_i(e)=S_{i,e}$ if $i\leq n_e$, and $S_i(e)=0$ for $i>n_e$.
    Then $S_i:=\bigoplus_{e\in\Gamma_0}S_i(e)$ is a $\Gamma_0$-simple $R$-module for each $1\leq i \leq n$, and $M=S_1\oplus\cdots\oplus S_n$.

    $(4)\Rightarrow (2)$ and $(3)$: If $M=S_1\oplus \cdots\oplus S_n$, where each $S_i$ is $\Gamma_0$-simple, then $M$ is gr-semisimple and the sequence $c_{\Gamma_0-\rm gr}(M)$ is bounded by $n$. 
    It follows from \Cref{prop: comp fort finito <=> fort art e for noet} that $M$ is strongly $\Gamma_0$-noetherian and strongly $\Gamma_0$-artinian.

    $(2)$ or $(3)\Rightarrow(1)$: Suppose that $M$ is gr-semisimple and strongly $\Gamma_0$-noetherian (resp. strongly $\Gamma_0$-artinian). 
    By \Cref{prop: carac fort G0-art/noet}, $M$ is $\Gamma_0$-noetherian (resp. $\Gamma_0$-artinian) and there exist $\Delta_0\subseteq\Gamma_0$ and $n_0\in\mathbb{N}$ such that $\Gamma_0\setminus\Delta_0$ is finite and, for each $e\in\Delta_0$, $M(e)$ is a direct sum of at most $n_0$ gr-simple submodules.
    Therefore, from \Cref{coro: carac gr-ss fg}(5), there exists $n\in\mathbb{N}$ such that each $M(e)$, $e\in\Gamma_0$, is a direct sum of at most $n$ gr-simple submodules.

    $(3)\Leftrightarrow(5)$: By \Cref{lem: fort G0-art/noet => G0-art/noet}, both the items imply that $M$ is $\Gamma_0$-artinian. Therefore, this equivalence follows from \Cref{coro: carac gr-ss fg}.
\end{proof}

\begin{coro}
\label{coro: M gr-art => M/rad(M) gr-ss}
    Let $R$ be a $\Gamma$-graded ring and $M$ be a $\Gamma$-graded right $R$-module.
    \begin{enumerate}[\rm (1)]
        \item If $M$ is gr-artinian, then $M/\radgr(M)$ is a finite direct sum of gr-simple $R$-modules.
        \item If $M$ is   $\Gamma_0$-artinian, then $M/\radgr(M)$ is a $\Gamma_0$-finitely generated  gr-semisimple module.
        \item  If $M$ is strongly  $\Gamma_0$-artinian, then $M/\radgr(M)$ is a finite direct sum of $\Gamma_0$-simple $R$-modules.
    \end{enumerate}
\end{coro}

\begin{proof}
 (1)   If $M$ is gr-artinian, then $M/\radgr(M)$ is gr-artinian by \Cref{prop: M art ou noeth <--> N e M/N tbm}.
    Since $\radgr(M/\radgr(M))=0$ by \Cref{lem: rad(M/rad)=0}(1), the result follows from \Cref{prop: carac gr-ss fg}. 

(2) It suffices to apply (1) to $M(e)$ for each $e\in\Gamma_0$, noting that \[\frac{M}{\radgr(M)}(e)\isogr \frac{M(e)}{(\radgr(M))(e)}=\frac{M(e)}{\radgr(M(e))}\] by \Cref{prop: rad M(e)}(1).

(3) If $M$ is strongly  $\Gamma_0$-artinian, then $M/\radgr(M)$ is strongly  $\Gamma_0$-artinian by \Cref{prop: M fortemente G0-art ou G0-noeth <--> N e M/N tbm}.
    Since $\radgr(M/\radgr(M))=0$ by \Cref{lem: rad(M/rad)=0}(1), the result follows from \Cref{coro: carac fort gr-ss}.
\end{proof}

It is clear that a direct sum $M_1\oplus\cdots\oplus M_n$ of graded modules is finitely generated if and only if each $M_i$ is finitely generated. 
The dual statement  is also true.

\begin{coro}
    Let $R$ be a $\Gamma$-graded ring and $M$ be a $\Gamma$-graded right $R$-module.
    \begin{enumerate}[\rm (1)]
        \item Suppose that $M=M_1\oplus\cdots\oplus M_n$ is a direct sum of $\Gamma$-graded right $R$-modules.
    Then $M$ is finitely gr-cogenerated if and only if  $M_i$ is finitely gr-cogenerated for all $i=1,\dots,n$. 
    \item Suppose that $M=\bigoplus_{i\in I} M_i$ is a direct sum of a $\Gamma_0$-finite family $\{M_i\colon i\in I\}$ of $\Gamma$-graded right $R$-modules. Then $M$ is $\Gamma_0$-finitely gr-cogenerated if and only if  $M_i$ is  $\Gamma_0$-finitely gr-cogenerated for all $i\in I$. 
    \end{enumerate}
\end{coro}

\begin{proof}
(1)    The ``only if'' part follows because submodules of finitely gr-cogenerated modules are finitely gr-cogenerated modules.

    Suppose that  $M_1,\dots,M_n$ are finitely gr-cogenerated. By \Cref{prop: carac fini (co)gerado via rad e soc}(2), we have to show that $\socgr(M)$ is finitely gr-cogenerated and $\socgr(M)$ is gr-essential in $M$. 

First observe that $\socgr(M),\socgr(M_1),\dotsc,\socgr(M_n)$ are gr-semissimple modules and, by \Cref{coro: rad da soma direta}, $\socgr(M)=\socgr(M_1)\oplus\cdots\oplus\socgr(M_n)$. Moreover, $\socgr(M_1),\dotsc,\socgr(M_n)$ are finitely gr-cogenerated because they are graded submodules of finitely gr-cogenerated graded modules. By \Cref{prop: carac gr-ss fg}, $\socgr(M_1)$, $\dotsc$, $\socgr(M_n)$ are finitely generated. Thus, $\socgr(M)$ is finitely generated. Again by \Cref{prop: carac gr-ss fg}, $\socgr(M)$ is finitely gr-cogenerated.

By \Cref{prop: carac fini (co)gerado via rad e soc}(2),  $\socgr(M_i)$ is gr-essential in $M_i$ for each $i=1,\dots,n$. It remains to prove that $\socgr(M)$ is gr-essential in $M$, and by induction, it suffices to show that $\socgr(M_1)\oplus\socgr(M_2)$ is gr-essential in $M_1\oplus M_2$.
    Let $X$ be a nonzero graded submodule of $M_1\oplus M_2$ and take $0\neq x=m_1+m_2\in X$, where $m_1\in\h(M_1)$ and $m_2\in\h(M_2)$.
    We can suppose that $m_1\neq0$.
    Since $\socgr(M_1)$ is gr-essential in $M_1$, we have $m_1R\cap\socgr(M_1)\neq0$, from which there exists $a\in\h(R)$ such that $0\neq m_1a\in\socgr(M_1)$.
    If $m_2a\in \socgr(M_2)$, then $0\neq xa\in X\cap(\socgr(M_1)\oplus\socgr(M_2))$.
    If $m_2a\notin\socgr(M_2)$, then, analogously, there exists $a'\in\h(R)$ such that $0\neq m_2aa'\in\socgr(M_2)$, and it follows that $0\neq xaa'=m_1aa'+m_2aa'\in X\cap (\socgr(M_1)\oplus\socgr(M_2))$.
    Therefore, $X\cap (\socgr(M_1)\oplus\socgr(M_2))\neq0$, as desired.

(2) Since the family $\{M_i\colon i\in I\}$ is $\Gamma_0$-finite, the set $I_e:=\{i\in I: M_i(e)\neq0\}$ is finite for all $e\in\Gamma_0$. Thus we can apply (1) to 
$M(e)$ for all $e\in\Gamma_0$. And we obtain that, for all $e\in\Gamma_0$, $M(e)$ is finitely gr-cogenerated if and only if  $M_i(e)$ is finitely gr-cogenerated for all $i\in I_e$. Therefore, $M$ is $\Gamma_0$-finitely gr-cogenerated if and only if  $M_i$ is $\Gamma_0$-finitely gr-cogenerated for all $i\in I$.  
\end{proof}

Our next aim is to characterize the different chain conditions concerning finite gr-length. For this purpose,  we need the following definition.

\begin{defi}
\label{def: socle}
    Let $R$ be a $\Gamma$-graded ring and $M$ be a $\Gamma$-graded right $R$-module. 
    Set $\soc_{\rm gr}^0(M)=\{0\}$. Inductively, for each $n\in\mathbb{N}$, we define $\soc_{\rm gr}^{n+1}(M)$ as the graded submodule of $M$ containing $\soc_{\rm gr}^n(M)$ such that 
    $$\frac{\soc_{\rm gr}^{n+1}(M)}{\soc_{\rm gr}^n(M)}=\soc_{\rm gr}\left(\frac{M}{\soc_{\rm gr}^n(M)}\right).$$
\end{defi}
We have the following chain  for each graded right $R$-module $M$:
\begin{equation}
\label{eq: socle series}    
    0=\soc_{\rm gr}^0(M)\subseteq \soc_{\rm gr}^1(M)\subseteq\soc_{\rm gr}^2(M)\subseteq\cdots\subseteq\soc_{\rm gr}^n(M)\subseteq\cdots\subseteq M
\end{equation}
By \Cref{prop: rad M(e)}(2), $\socgr(M(e))=\socgr(M)(e)$ for each $e\in\Gamma_0$. Thus, it can be shown by induction that 
\begin{equation}
\label{eq: socle series M(e)}    
\soc_{\rm gr}^n(M)(e)=\soc_{\rm gr}^n(M(e))  
\end{equation}
for each $e\in\Gamma_0$ and $n\in\mathbb{N}$. Thus, if $\soc_{\rm gr}^n(M)(e)=\soc_{\rm gr}^{n+1}(M)(e)$ for some $e\in\Gamma_0$ and $n\in\mathbb{N}$, then $\soc_{\rm gr}^n(M)(e)=\soc_{\rm gr}^{n+k}(M)(e)$ for all $k\geq 0$. Therefore the ascending chain \eqref{eq: socle series}  is tight.

\begin{lema}
\label{lem: M gr-art => socle chain is strict}
    Let $R$ be a $\Gamma$-graded ring and $M$ be a $\Gamma$-graded right $R$-module. If  $M$ is gr-artinian  and $\soc_{\rm gr}^n(M)\neq M$ for some $n\in\mathbb{N}$, then $\soc_{\rm gr}^n(M)\subsetneq\soc_{\rm gr}^{n+1}(M)$. As a consequence,  if  $M$ is $\Gamma_0$-artinian  and $\soc_{\rm gr}^n(M)\neq M$ for some $n\in\mathbb{N}$, then $\soc_{\rm gr}^n(M)\subsetneq\soc_{\rm gr}^{n+1}(M)$.
\end{lema}

\begin{proof}
    Suppose that $n\in\mathbb{N}$ and $\soc_{\rm gr}^n(M)\neq M$.
    Then $M/\soc_{\rm gr}^n(M)$ is a nonzero gr-artinian $R$-module by \Cref{prop: M art ou noeth <--> N e M/N tbm}.
    Therefore, $M/\soc_{\rm gr}^n(M)$ contains a gr-simple graded submodule. 
    It follows that 
    $$\soc_{\rm gr}^{n+1}(M)/\soc_{\rm gr}^n(M)=\soc_{\rm gr}(M/\soc_{\rm gr}^n(M))\neq0,$$ 
    and thus $\soc_{\rm gr}^n(M)\subsetneq\soc_{\rm gr}^{n+1}(M)$.

    Suppose now that $M$ is $\Gamma_0$-artinian. If $\soc_{\rm gr}^n(M)\neq M$, then there exists $e\in \Gamma_0$ such that $\soc_{\rm gr}^n(M)(e)\neq M(e)$. By \eqref{eq: socle series M(e)}, $\soc_{\rm gr}(M(e))=\soc_{\rm gr}^n(M)(e)\neq M(e)$.
    By the gr-artinian case and \eqref{eq: socle series M(e)}, $$\soc_{\rm gr}^n(M)(e)=\soc_{\rm gr}(M(e))\subsetneq\soc_{\rm gr}^{n+1}(M(e))=\soc_{\rm gr}^{n+1}(M)(e).$$ Therefore, $\soc_{\rm gr}^n(M)\subsetneq\soc_{\rm gr}^{n+1}(M)$.    
    \end{proof}

Now we are ready to characterize finite gr-length.

\begin{prop}
\label{prop: carac comp finito via socle}
    Let $R$ be a $\Gamma$-graded ring and $M$ be a $\Gamma$-graded right $R$-module. 
    The following assertions are equivalent:
    \begin{enumerate}[\rm (1)]
        \item $M$ has finite gr-length.
        \item $M$ is gr-artinian and there exists $n\in\mathbb{N}$ such that $\soc_{\rm gr}^n(M)=M$.
        \item $M$ is gr-noetherian and there exists $n\in\mathbb{N}$ such that $\soc_{\rm gr}^n(M)=M$.
    \end{enumerate}
\end{prop}

\begin{proof}
    $(1) \Rightarrow (2)$ and (3): If $M$ has finite gr-length, then $M$ is both gr-artinian and gr-noetherian by \Cref{prop: comp finito = art + noet}. Moreover, the chain \eqref{eq: socle series} stabilizes at $M$ by \Cref{lem: M gr-art => socle chain is strict}.

    (2) or $(3)\Rightarrow(1)$: Suppose that $M$ is gr-artinian (resp. gr-noetherian) and $\soc_{\rm gr}^n(M)=M$ for some $n\in\mathbb{N}$.
    Fix $k\in\mathbb{N}$. Then $\soc_{\rm gr}^k(M)$, $M/\soc_{\rm gr}^k(M)$, and $\soc_{\rm gr}^{k+1}(M)/\soc_{\rm gr}^k(M)=\soc_{\rm gr}(M/\soc_{\rm gr}^k(M))$ are gr-artinian (resp. gr-noetherian) modules by \Cref{prop: M art ou noeth <--> N e M/N tbm}. 
    Therefore, the gr-semisimple module $\soc_{\rm gr}^{k+1}(M)/\soc_{\rm gr}^k(M)$ is gr-noetherian (resp. gr-artinian) by \Cref{prop: carac gr-ss fg}.
    Applying inductively \Cref{prop: M art ou noeth <--> N e M/N tbm}, we obtain that $\soc_{\rm gr}^k(M)$ is gr-noetherian (resp. gr-artinian) for each $k\in\mathbb{N}$.
    In particular, $M=\soc_{\rm gr}^n(M)$ is gr-noetherian (resp. gr-artinian), and it follows from \Cref{prop: comp finito = art + noet} that $M$ has finite gr-length.
\end{proof}

As a consequence of the last proposition, we can characterize $\Gamma_0$-finite length. 

\begin{coro}\label{coro: carac G0 comp finito via socle}
  Let $R$ be a $\Gamma$-graded ring and $M$ be a $\Gamma$-graded right $R$-module. 
    The following assertions are equivalent:
    \begin{enumerate}[\rm (1)]
        \item $M$ has $\Gamma_0$-finite gr-length.
        \item $M$ is $\Gamma_0$-artinian and, for each $e\in \Gamma_0$, there exists $n_e\in\mathbb{N}$ such that $\soc_{\rm gr}^{n_e}(M)(e)=M(e)$.
        \item $M$ is $\Gamma_0$-noetherian and, for each $e\in\Gamma_0$, there exists $n_e\in\mathbb{N}$ such that $\soc_{\rm gr}^{n_e}(M)(e)=M(e)$.
    \end{enumerate}
\end{coro}

\begin{proof}
$(1)\Leftrightarrow (2):$
$M$ has $\Gamma_0$-finite gr-length if and only if $M(e)$ has finite gr-length for each $e\in \Gamma_0$. Equivalently, by \Cref{prop: carac comp finito via socle}, 
$M(e)$ is gr-artinian and there exists $n_e\in\mathbb{N}$ such that $\soc_{\rm gr}^{n_e}(M)(e)=\soc_{\rm gr}^{n_e}(M(e)) =M(e)$ for each $e\in \Gamma_0$.

$(1)\Leftrightarrow (3):$ 
The proof is analogous to that of $(1)\Leftrightarrow(2)$.
\end{proof}

We can also characterize strongly $\Gamma_0$-finite length. 

\begin{prop}
\label{prop: carac fort G0 comp finito via socle}
    Let $R$ be a $\Gamma$-graded ring and $M$ be a $\Gamma$-graded right $R$-module. 
    The following assertions are equivalent:
    \begin{enumerate}[\rm (1)]
        \item $M$ has strongly $\Gamma_0$-finite gr-length.
        \item $M$ is strongly $\Gamma_0$-artinian and there exists $n\in\mathbb{N}$ such that $\soc_{\rm gr}^n(M)=M$.
        \item $M$ is strongly $\Gamma_0$-noetherian and there exists $n\in\mathbb{N}$ such that $\soc_{\rm gr}^n(M)=M$.
    \end{enumerate}
\end{prop}

\begin{proof}
$(1) \Rightarrow (2)$ and (3): If $M$ has strongly $\Gamma_0$-finite gr-length, then $M$ is both strongly $\Gamma_0$-artinian and strongly $\Gamma_0$-noetherian by \Cref{prop: comp fort finito <=> fort art e for noet}. Moreover,  $M$ is $\Gamma_0$-artinian and $\Gamma_0$-noetherian by \Cref{lem: fort G0-art/noet => G0-art/noet}, and the chain \eqref{eq: socle series} stabilizes at $M$ because it is tight and by \Cref{lem: M gr-art => socle chain is strict}.

    (2) or $(3)\Rightarrow(1)$: Suppose that $M$ is strongly $\Gamma_0$-artinian (resp. strongly $\Gamma_0$-noetherian) and $\soc_{\rm gr}^n(M)=M$ for some $n\in\mathbb{N}$.
    Fix $k\in\mathbb{N}$. Then $\soc_{\rm gr}^k(M)$, $M/\soc_{\rm gr}^k(M)$, and $\soc_{\rm gr}^{k+1}(M)/\soc_{\rm gr}^k(M)=\soc_{\rm gr}(M/\soc_{\rm gr}^k(M))$ are strongly $\Gamma_0$-artinian (resp. strongly $\Gamma_0$-noetherian) modules by \Cref{prop: M fortemente G0-art ou G0-noeth <--> N e M/N tbm}. 
    Therefore, the gr-semisimple module $\soc_{\rm gr}^{k+1}(M)/\soc_{\rm gr}^k(M)$ is strongly $\Gamma_0$-noetherian (resp. strongly $\Gamma_0$-artinian) by \Cref{coro: carac fort gr-ss}.
    Applying inductively \Cref{prop: M fortemente G0-art ou G0-noeth <--> N e M/N tbm}, we obtain that $\soc_{\rm gr}^k(M)$ is strongly $\Gamma_0$-noetherian (resp. strongly $\Gamma_0$-artinian) for each $k\in\mathbb{N}$.
    In particular, $M=\soc_{\rm gr}^n(M)$ is strongly $\Gamma_0$-noetherian (resp. strongly $\Gamma_0$-artinian), and it follows from \Cref{prop: comp fort finito <=> fort art e for noet} that $M$ has strongly $\Gamma_0$-finite gr-length.    
\end{proof}

In Section~\ref{sec: gr-semilocal rings}, for gr-semilocal rings $R$, we will relate (strongly) ($\Gamma_0$-)finite gr-length of graded $R$-modules $M$ with the powers of the  graded Jacobson radical of $R_R$, instead of the terms of the Loewy series of $M$.

\section{Graded Jacobson radical of graded rings}
\label{sec: rad(R)}

In this section, we present  characterizations of the graded ideal $\radgr(R_R)$ together with graded versions of some classical results about the Jacobson radical of rings. Some characterizations of  $\radgr(R_R)$ are direct consequence of the results in Section~\ref{sec: gr-Jacobson gr-socle} about graded modules, but some others come from the graded ring structure.

\begin{defi}
    Let $R$ be a $\Gamma$-graded ring. 
\begin{enumerate}[\rm(1)]
    \item A graded right ideal $\mathfrak{m}$ of $R$ is said to be \emph{gr-maximal} if $\mathfrak{m}$ is a gr-maximal graded submodule of $R_R$.
    \item A $\Gamma$-graded right $R$-module $M$ is \emph{faithful} if \[
    \rann_R(M):=\{a\in R\colon xa=0 \textrm{ for all } x\in M\}=\{0\}.
    \]
    \item The graded ring $R$ is said to be \emph{right gr-primitive} if $R$ has a faithful gr-simple right module.
    \item A graded ideal $P$ of $R$ is \emph{right gr-primitive} if the quotient ring $R/P$ is right gr-primitive.
\end{enumerate}
\end{defi}

 \begin{lema}\label{lem: charact right gr-primitive ideal} 
     Let $R$ be a $\Gamma$-graded ring. A graded ideal $P$ of $R$ is right gr-primitive if and only if there exists a gr-simple right $R$-module $S$ such that $P=\rann(S)$.
 \end{lema}

 \begin{proof}
     Suppose that $P$ is a right gr-primitive ideal of $R$. Let $S$ be a faithful gr-simple right $R/P$-module. When considered as a $\Gamma$-graded right $R$-module in the natural way, $S$ is a gr-simple right $R$-module with $\rann_R(S)=P$. 

    Conversely, suppose that there exists a gr-simple right $R$-module $S$ such that  $P=\rann_R(S)$. Then $S$ has a natural structure of $\Gamma$-graded right $R/P$-module. Moreover, $S$ is a gr-simple $R/P$-module with  $\rann_{R/P} (S)=\{\overline{a}\in R/P\colon x\overline{a}=0 \textrm{ for all } x\in S\}=\{\overline{0}\}$. 
 \end{proof}

\begin{teo}
\label{teo: a carac de radgr(R)}
    Let $R$ be a $\Gamma$-graded ring, $\gamma \in \supp(R)$, and $a \in R_{\gamma}$. 
    The following assertions are equivalent: 
    \begin{enumerate}[\rm (1)]
        \item $a \in \radgr(R_R)$.
        \item $a \in \radgr(R(r(\gamma)))$.
        \item $a\in\maxi$ for each gr-maximal graded ideal $\maxi$ of $R$.
        \item $a\in\maxi$ for each gr-maximal graded submodule $\maxi$ of $R(r(\gamma))$.
        \item $aR$ is gr-superfluous in $R_R$.
        \item $aR$ is gr-superfluous in $R(r(\gamma))$.
        \item $1_{r(\gamma)} - ax$ is right invertible in $R_{r(\gamma)}$ for each $x \in R_{\gamma^{-1}}$.
        \item $1_{r(\gamma)} - ax$ is invertible in $R_{r(\gamma)}$ for each $x \in R_{\gamma^{-1}}$.
        \item $1_{r(\gamma)}-yax$ is invertible in $R_{r(\gamma)}$ for all $x,y\in\h(R)$ satisfying $yax\in R_{r(\gamma)}$.
        \item $Sa = 0$ for each gr-simple $\Gamma$-graded right $R$-module $S$.
        \item $a\in P$ for each right gr-primitive graded ideal $P$ of $R$.
    \end{enumerate}
\end{teo}

\begin{proof}
    The equivalences $(1)\Leftrightarrow(3)$ and $(2)\Leftrightarrow(4)$ follow by definition.

    By \Cref{prop: rad M_sigma}(1), we have $(1)\Leftrightarrow(5)$ and $(2)\Leftrightarrow(6)$.

    From \Cref{prop: rad M(e)}(1), we obtain $(1)\Leftrightarrow(2)$.

    $(7)\Rightarrow(8)$: Let $x \in R_{\gamma^{-1}}$. 
    By (7), there exists $u\in R_{r(\gamma)}$ such that $(1_{r(\gamma)}-ax)u=1_{r(\gamma)}$. 
    Again, by (7), we get that $u=1_{r(\gamma)}+a(xu)$ is right invertible in $R_{r(\gamma)}$. 
    Hence, $u$ is invertible in $R_{r(\gamma)}$ and $1_{r(\gamma)}-ax=u^{-1}$.

    $(8)\Rightarrow(9)$: Let $\sigma,\tau\in\Gamma$ be such that $\tau\gamma\sigma=r(\gamma)$, and take $x\in R_\sigma$, and $y\in R_\tau$.
    Then $xy\in R_{\sigma\tau}=R_{\gamma^{-1}}$, and it follows from (8) that there exists an inverse $u$ of $1_{r(\gamma)}-axy$ in $R_{r(\gamma)}$.
    Hence, $(1_{r(\gamma)}-axy)u=1_{r(\gamma)}=u(1_{r(\gamma)}-axy)$ implies that $axyu=u-1_{r(\gamma)}=uaxy$, and it follows that 
    \begin{align*}
        (1_{r(\gamma)}-yax)(1_{r(\gamma)}+yuax)
        =1_{r(\gamma)}-yax+yuax-yaxyuax=1_{r(\gamma)};\\
        (1_{r(\gamma)}+yuax)(1_{r(\gamma)}-yax)
        =1_{r(\gamma)}-yax+yuax-yuaxyax=1_{r(\gamma)}.
    \end{align*}
    Therefore, $1_{r(\gamma)}-yax$ is invertible in $R_{r(\gamma)}$.

    $(9)\Rightarrow(7)$: Apply (9) with $y=1_{r(\gamma)}$. 

    $(6)\Rightarrow(7)$: Let $x \in R_{\gamma^{-1}}$. 
    Since $aR+(1_{r(\gamma)}-ax)R=R(r(\gamma))$, (6) implies that $(1_{r(\gamma)}-ax)R=R(r(\gamma))$, and it follows that $1_{r(\gamma)}\in (1_{r(\gamma)}-ax)R$.

    $(7)\Rightarrow(10)$: Suppose, by contradiction, that (7) holds but there exists a gr-simple right $R$-module $S$ such that $Sa \neq 0$. 
    Let $s \in \h(S)$ be such that $sa \neq 0$. 
    Since $S$ is gr-simple, we have $saR = S$. 
    Choosing $x \in R_{\gamma^{-1}}$ such that $s = sax$ and $u \in R_{r(\gamma)}$ such that $(1_{r(\gamma)} - ax)u = 1_{r(\gamma)}$, we get $s = s 1_{r(\gamma)} = s (1_{r(\gamma)} - ax) u = 0$,
    a contradiction.

    $(10)\Rightarrow(2)$: Let $S$ be a gr-simple right $R$-module and $g\in\Homgr(R(r(\gamma)),S)$. 
    Then $g(a)=g(1_{r(\gamma)}a)=g(1_{r(\gamma)})a\in Sa=0$ by (10), and thus $a\in\ker g$. 
    By \Cref{lem: rad e soc via ker e im}(1), $a\in\radgr(R(r(\gamma)))$.   

    $(10)\Leftrightarrow(11)$ This equivalence follows from \Cref{lem: charact right gr-primitive ideal} 
\end{proof}

\begin{coro}
\label{coro: radgr(R) como o maior ideal tq}
Let $R$ be a $\Gamma$-graded ring. The following statements hold:
\begin{enumerate}[\rm (1)]
    \item $\radgr(R_R)$ is the largest graded ideal $J$ of $R$ such that $1_e+J_e\subseteq \U(R_e)$ for every $e\in\Gamma_0'(R)$.
    \item $\radgr(R_R)=\radgr( _RR)$.
\end{enumerate}
\end{coro}

\begin{proof}
(1)   By \Cref{coro: rad e soc sao prerad}(4) (or \Cref{teo: a carac de radgr(R)}(11)) $\radgr(R_R)$ is a graded ideal. The statement now follows from \Cref{teo: a carac de radgr(R)}(9).

(2) The version of \Cref{teo: a carac de radgr(R)} for $\radgr(_RR)$  gives the corresponding statement of the previous item, which says that $\radgr(_RR)$ is the largest graded ideal $J$ of $R$ such that $1_e+J_e\subseteq \U(R_e)$ for every $e\in\Gamma_0'(R)$, as desired.
\end{proof}

By \Cref{coro: radgr(R) como o maior ideal tq}(2), we set the following definition.

\begin{defi}
    Let $R$ be a $\Gamma$-graded ring. The \emph{graded Jacobson radical} of the graded ring $R$, denoted by $\radgr(R)$, is the graded ideal $\radgr(R_R)=\radgr( _RR)$.
\end{defi}

We continue with some other characterizations of $\radgr(R)$. One of them deals with gr-superfluous right (left) ideals.  A graded right (left) ideal is gr-superfluous if it is a gr-superfluous submodule of $R_R$ ($_RR$). 

\begin{coro}
\label{coro: radgr(R)e=rad(Re)}
    The following assertions hold true for a $\Gamma$-graded ring $R$:
    \begin{enumerate}[\rm (1)]
        \item $\radgr(R)$ is the largest graded ideal $J$ of $R$ such that $J_e=\rad(R_e)$ for every $e\in\Gamma_0$.
        \item $\radgr(R)$ is the largest graded ideal $J$ of $R$ such that $1_eJ1_e=\radgr(1_eR1_e)$ for every $e\in\Gamma_0$.
        \item $\radgr(R)$ is the largest gr-superfluous right (left) ideal in $R$.
    \end{enumerate}
\end{coro}

\begin{proof}
     (1) Let $e\in\Gamma_0$. By \cite[Corollary 4.5(A)]{Lam1}, $\rad(R_e)$ is the largest ideal $J_e$ of $R_e$ such that $1_e+J_e\subseteq \U(R_e)$.       
     The result now follows from \Cref{coro: radgr(R) como o maior ideal tq}(1).
    
    (2)     For each $e\in\Gamma_0$, we have $\radgr(R)_e = \rad(R_e)$ and $1_e\radgr(R)1_e = \radgr(1_eR1_e)$ by \Cref{teo: a carac de radgr(R)}(7).
    In particular, $\radgr(1_eR1_e)_e = \radgr(R)_e = \rad(R_e)$ for every $e\in\Gamma_0$.
    Thus, (2) is a consequence of (1).

    (3) By  \Cref{prop: (in)essential}(1), $\radgr(R)$ contains all gr-superfluous right $R$-submodules of $R$. Moreover, $R_R$ is $\Gamma_0$-finitely generated. Thus, Proposition~\ref{prop: carac fini (co)gerado via rad e soc}(3) implies that $\radgr(R)$ is gr-superfluous as a right ideal. The result for left ideals follows analogously using the left versions of \Cref{prop: (in)essential}(1) and Proposition~\ref{prop: carac fini (co)gerado via rad e soc}(3).
\end{proof}

We comment in passing the following fact.  Given a graded ideal $J$ of a graded ring $R$, it can be considered as a graded submodule of $R_R$ and as a graded submodule of $_RR$.  By \Cref{coro: radgr(R)e=rad(Re)}(3), a right (left) graded ideal is gr-superfluous if and only if it is contained in $\radgr(R)$. Hence $J$ is a gr-superfluous submodule of $R_R$ if and only if $J$ is a gr-superfluous submodule of $_RR$.

The following result is a graded version of \cite[Proposition 11.1]{Lam1}.

\begin{coro}
\label{coro: rad = 0 <=> exist faithful semisimple}
    Let $R$ be a $\Gamma$-graded ring.
    Then $\radgr(R)=0$ if and only if there exists a faithful gr-semisimple $\Gamma$-graded right $R$-module.
\end{coro}

\begin{proof}
    Suppose that $M$ is a faithful gr-semisimple $\Gamma$-graded right $R$-module.
    By \Cref{teo: a carac de radgr(R)}(10), $M\cdot\radgr(R)=0$, and it follows that $\radgr(R)=0$.

    Conversely, assume that $\radgr(R)=0$. 
    Let $\{S_i:i\in I\}$ be a complete set of representatives of gr-isomorphism classes of gr-simple $\Gamma$-graded right $R$-modules.
    Then $M:=\bigoplus_{i\in I}S_i$ is a gr-semisimple $R$-module such that
    \[\rann(M)=\bigcap_{i\in I}\rann(S_i)=\radgr(R)\]
    by \Cref{teo: a carac de radgr(R)}(10).
\end{proof}

\begin{prop}
\label{prop: rad of R/I for I contained in rad}
    Let $R$ and $S$ be $\Gamma$-graded rings. The following assertions hold:
    \begin{enumerate}[\rm (1)]
        \item If $\varphi\colon R\rightarrow S$ is surjective gr-homomorphism of rings, then $\varphi(\radgr(R))\subseteq\radgr(S)$. 
        \item If $\varphi\colon R\rightarrow S$ is surjective gr-homomorphism of rings with $\ker\varphi\subseteq \radgr(R)$, then $\varphi(\radgr(R))=\radgr(S)$. 
        \item If $J$ is a graded ideal contained in $\radgr(R)$, then $\radgr\left(R/J\right)=\radgr(R)/J$.
        \item $\radgr\left(\frac{R}{\radgr(R)}\right)=\{0\}.$
    \end{enumerate}
    
        \end{prop}

\begin{proof}
    (1) Let $\gamma\in\supp(R)$ and $a\in\radgr(R)_\gamma$. We will show that $\varphi(a)\in\radgr(S)$.
    Take $y\in S_{\gamma^{-1}}$. 
    By the surjectivity of $\varphi$, we have $y=\varphi(x)$ for some $x\in R_{\gamma^{-1}}$, and by \Cref{teo: a carac de radgr(R)}(7), there exists $u\in R_{r(\gamma)}$ such that $(1_{r(\gamma)} - ax)u=1_{r(\gamma)}$.
    Then 
    \[(1_{r(\gamma)} - \varphi(a)y)\varphi(u)=(\varphi(1_{r(\gamma)}) - \varphi(a)\varphi(x))\varphi(u)=\varphi((1_{r(\gamma)} - ax)u)=\varphi(1_{r(\gamma)})=1_{r(\gamma)},\]
    and it follows that $1_{r(\gamma)} - \varphi(a)y$ is right invertible in $S_{r(\gamma)}$.
    By \Cref{teo: a carac de radgr(R)}(7), $\varphi(a)\in\radgr(S)_\gamma$.

    (2) Let $\gamma\in\supp(S)$ and $b\in\radgr(S)_\gamma$.
    Since $\varphi$ is surjective, we have $b=\varphi(a)$ for some $a\in R_\gamma$.
    We will show that $a\in\radgr(R)$.
    Take $x\in R_{\gamma^{-1}}$. 
    Then $\varphi(x)\in S_{\gamma^{-1}}$ and, by \Cref{teo: a carac de radgr(R)}(7), there exists $v\in S_{r(\gamma)}$ such that $(1_{r(\gamma)} - b\varphi(x))v=1_{r(\gamma)}$.
    Since $\varphi$ is surjective, we have $v=\varphi(u)$ for some $u\in R_{r(\gamma)}$.
    Then
    \[\varphi((1_{r(\gamma)} - ax)u)=(1_{r(\gamma)} - b\varphi(x))v=1_{r(\gamma)}=\varphi(1_{r(\gamma)}),\]
    and it follows that $1_{r(\gamma)}-(1_{r(\gamma)} - ax)u\in\ker\varphi\subseteq\radgr(R)$.
    Again by \Cref{teo: a carac de radgr(R)}(7), $(1_{r(\gamma)} - ax)u$ is right invertible in $R_{r(\gamma)}$, and thus $(1_{r(\gamma)} - ax)$ also is.
    Therefore, $a\in\radgr(R)_\gamma$ by \Cref{teo: a carac de radgr(R)}(7).

    (3) Apply (2) for the canonical map $\varphi:R\to R/J$.

    (4) This is immediate from (3)
\end{proof}

Note that \Cref{prop: rad of R/I for I contained in rad}(2) is similar to the first part of \Cref{coro: quando g(rad)=rad}(1); however, unlike the case of gr-homomorphism of modules (\Cref{coro: rad e soc sao prerad}(1)), the surjectivity is necessary in \Cref{prop: rad of R/I for I contained in rad}(1).
For example, in the ungraded setting, consider the  inclusion map from a local integral domain with nonzero Jacobson radical in its field of fractions.

The graded ring $R/\radgr(R)$ shares some common properties with the graded ring $R$, as shown in the following graded version of \cite[Proposition~4.8]{Lam1}.

\begin{prop} 
\label{prop: R e R/J tem os mesmos gr-simples}
Let $R$ be a $\Gamma$-graded ring and $\overline{R} = R/\radgr(R)$. The following assertions hold true:

\begin{enumerate}[\rm (1)] 
    \item $S$ is a gr-simple $\Gamma$-graded $R$-module if and only if $S$ is a gr-simple $\Gamma$-graded $\overline R$-module.
    \item  $\overline{R}$ is a gr-semisimple ring if and only if it is a gr-semisimple $R$-module.
    \item $a\in \h(R)$ is right (resp. left) gr-invertible in $R$ if and only if $a+\radgr(R)\in \h(\overline R)$ is right (resp. left) gr-invertible in $\overline R$.
    \item  $\radgr(R)$ is the largest graded right (resp. left) ideal $J$ of $R$ such that, for each $e\in\Gamma_0$ and $u\in R_e$, we have that $u$ is right (resp. left) invertible in $R_e$ if and only if $u+J$ is right (resp. left) invertible in $(R/J)_e$.
\end{enumerate}
\end{prop}

\begin{proof}

(1) If $S$ is a gr-simple $\Gamma$-graded $\overline R$-module, then it is clear that $S$ is a gr-simple $\Gamma$-graded $R$-module by restriction of scalars. 
Conversely, if $S$ is a gr-simple $\Gamma$-graded right $R$-module, then $S$ becomes a gr-simple $\Gamma$-graded right $\overline R$-module by defining the action $x(a+\radgr(R)):=xa$ for each $x\in S$ and $a\in R$. This is well-defined by \Cref{teo: a carac de radgr(R)}(10).

 (2) By (1), $\overline{R}$ is a sum of gr-simple $\Gamma$-graded $R$-modules if and only if $\overline{R}$ is a sum of gr-simple $\Gamma$-graded $\overline{R}$-modules.

(3) Let $\gamma\in\Gamma$ and $a\in R_\gamma$ be such that $a+\radgr(R)$ is right gr-invertible in $\overline R$. 
Then there exists $b\in R_{\gamma^{-1}}$ such that $ab+\radgr(R)=1_{r(\gamma)}+\radgr(R)$. 
Thus, $ab\in 1_{r(\gamma)}+\radgr(R)_{r(\gamma)}\subseteq \U(R_{r(\gamma)})$ by \Cref{coro: radgr(R) como o maior ideal tq}(1), and it follows that $a$ is right gr-invertible in $R$.  The converse is obvious. 
The left case is similar.

(4) Suppose that $J$ is a graded right ideal of $R$ such that, for each $e\in\Gamma_0'(R)$ and $u\in R_e$, we have that $u$ is right invertible in $R_e$ if and only if $u+J$ is right invertible in $(R/J)_e$.
Let $\gamma\in\supp(R)$ and $a\in J_\gamma$.
For each $x\in R_{\gamma^{-1}}$, $(1_{r(\gamma)}-ax)+J=1_{r(\gamma)}+J$ is invertible in $(R/J)_{r(\gamma)}$, which implies that $1_{r(\gamma)}-ax$ is right invertible in $R_{r(\gamma)}$.
Therefore, $a\in\radgr(R)$ by \Cref{teo: a carac de radgr(R)}(7).
The left case is analogous.
\end{proof}

We give a definition that will be needed in the next theorem.

\begin{defi}
  Let $R$ be a $\Gamma$-graded ring. A graded right ideal $I$ of $R$ will be called \emph{gr-principal} if there exists $a\in\h(R)$ such that $I=aR$.  
\end{defi}

We now present the graded version of an important classical result. It gives characterizations of gr-semisimple rings in terms of graded chain conditions and the graded Jacobson radical. The equivalence of (1) and (2) in \Cref{teo: R gr-ss <=> rad(R)=0 e R art} for the ring $R_\mathcal{C}$ of a small preadditive category $\mathcal{C}$ was given in \cite[Theorem~4.4]{Mit}.

\begin{teo}
\label{teo: R gr-ss <=> rad(R)=0 e R art}
	Let $R$ be a $\Gamma$-graded ring. The following assertions are equivalent:
	\begin{enumerate}[\rm (1)]
		\item $R$ is a gr-semisimple ring.
		\item $\radgr(R)=0$ and $R$ is a right  $\Gamma_0$-artinian ring.
		\item $\radgr(R)=0$ and $R$ satisfies the descending chain condition on gr-principal graded right  ideals.
	\end{enumerate}
\end{teo}

\begin{proof}
		$(1)\Rightarrow(2)$: It follows from  \Cref{coro: carac gr-ss fg} since $R_R$ is $\Gamma_0$-finitely generated.
	
	$(2)\Rightarrow(3)$: Every descending chain of gr-principal graded right ideals is contained in some $R(e)$, with $e \in \Gamma_0$.
	
	$(3)\Rightarrow(1)$: Suppose (3) holds. 
    We claim that if $e \in \Gamma'_0(R)$ and $I$ is a nonzero graded $R$-submodule of $R(e)$, then there exist a gr-simple module $S$ and $X \subgr I$ such that $I = S \oplus X$. 
    In fact, there exists a minimal element $S$ in the set $\{aR : 0\neq a \in \h(I)\}$ because $R$ satisfies the descending chain condition on gr-principal graded right  ideals. 
    Then $S$ is gr-simple, and since $\radgr(R) = 0$, we obtain a gr-maximal graded right ideal $\mathfrak{m}$ of $R$ such that $S \cap \mathfrak{m} = 0$. 
    Hence, by the maximality of $\mathfrak{m}$, we get $R(e) = S \oplus \mathfrak{m}(e)$, and it follows that $I = S \oplus (I \cap \mathfrak{m}(e))$.
	
	Suppose, by contradiction, that $R$ is not a gr-semisimple ring. 
    Then there exists $e \in \Gamma'_0(R)$ such that $R(e)$ is not a gr-semisimple $R$-module. 
    By the claim above, there exist a gr-simple module $S_1$ and $X_1 \subgr R(e)$ such that $R(e) = S_1 \oplus X_1$. 
    Since $R(e)$ is not gr-simple, we have $X_1 \neq 0$, and thus there exist a gr-simple $R$-module $S_2$ and $X_2 \subgr X_1$ such that $X_1 = S_2 \oplus X_2$. Since $S_2 \neq 0$, we obtain $X_1 \supsetneq X_2$. Moreover,  $X_2\neq 0$ because $R(e)$ is not gr-semisimple.  Proceeding inductively, we obtain a strictly descending chain of graded direct summands of $R(e)$
	\[
	X_1 \supsetneq X_2 \supsetneq X_3 \supsetneq \cdots\,,
	\]
	contradicting the assumption in (3), since each $X_i$ is generated by an idempotent of $R_e$.
\end{proof}

We continue with a groupoid graded version of Nakayama's lemma.

\begin{prop} 
Let $R$ be a $\Gamma$-graded ring and $J$ be a graded right ideal of $R$. The following assertions are equivalent:
\begin{enumerate}[\rm (1)]

    \item $J \subseteq \radgr(R)$.
    
    \item For each finitely generated $\Gamma$-graded right $R$-module $M$, $MJ = M$ implies \mbox{$M = 0$}.
    
    \item For each $\Gamma$-graded right $R$-module $M$ and $N \subgr M$ such that ${M}/{N}$ is finitely generated,  $N + MJ = M$ implies $N = M$.
    
    \item For each $\Gamma_0$-finitely generated $\Gamma$-graded right $R$-module $M$, $MJ = M$ implies $M = 0$.
    
    \item For each $\Gamma$-graded right $R$-module $M$ and $N \subgr M$ such that ${M}/{N}$ is $\Gamma_0$-finitely generated,  $N + MJ = M$ implies $N = M$.
    
\end{enumerate}
\end{prop}

\begin{proof}

$(1) \Rightarrow (2)$: Let $M$ be a nonzero finitely generated $\Gamma$-graded right $R$-module. 
Since $M$ is finitely generated, there exists a gr-maximal graded submodule $M'$ of $M$ by \cite[Lemma 55]{CLP}. 
Thus, ${M}/{M'}$ is gr-simple and $ \left( {M}/{M'} \right)\cdot J = 0$ by (1) and \Cref{teo: a carac de radgr(R)}(10). 
Therefore, $ M J \subseteq M'$, which shows that $MJ \neq M$.

$(2) \Rightarrow (3)$: It suffices to apply (2) to the graded right $R$-module ${M}/{N}$.

$(3) \Rightarrow (1)$: Suppose that (3) holds, but that there exists a homogeneous element $a \in \h(J)$ such that $a \notin \radgr(R)$. 
Then there exists a gr-maximal graded right ideal $\mathfrak{m}$ of $R$ such that $a \notin \mathfrak{m}$. 
Hence, $\mathfrak{m} + J = R$, which implies $\mathfrak{m} + RJ = R$. 
Since $R/\mathfrak{m}$ is gr-simple, it follows from (3) that $\mathfrak{m} = R$, a contradiction.

$(2) \Rightarrow (4)$: If $M$ is a $\Gamma_0$-finitely generated graded right $R$-module such that $MJ = M$, then $M(e)$ is finitely generated and $M(e)\cdot J = (MJ)(e) = M(e)$ for each $e \in \Gamma_0$. Applying (2), we get that $M(e) = 0$ for all $e \in \Gamma_0$. Hence $M = 0$.

$(4) \Rightarrow (2)$: Every finitely generated graded right $R$-module is $\Gamma_0$-finitely generated.

$(3) \Leftrightarrow (5)$: The proof of this equivalence is analogous to that of $(2) \Leftrightarrow (4)$.
\end{proof}

If $R$ is a $\Gamma$-graded ring and $J$ is a graded ideal of $R$ such that $1_eJ1_e = \radgr(1_eR1_e)$ for each $e \in \Gamma_0$, then $J \subseteq \radgr(R)$ by \Cref{coro: radgr(R)e=rad(Re)}(2).
According to \cite[Lemma 5]{Kelly}, if $R = R_{\mathcal{C}}$ for some small additive category $\mathcal{C}$, then $1_eJ1_e = \radgr(1_eR1_e)$ for each $e \in \Gamma_0$ implies that $J = \radgr(R)$.   
We generalize this result for more general groupoid graded rings.
Before stating the result, we observe that if $\mathcal{C}$ is a small additive category, then the $\mathcal{C}_0\times \mathcal{C}_0$-graded ring $R_\mathcal{C}$ has the following property: for any $A_1,A_2\in\mathcal{C}_0$, the canonical projection $p_i\in\mathcal{C}(A_1\oplus A_2, A_i)$ is a right gr-invertible element belonging to $\h(1_{e_i}R_\mathcal{C}1_e)$ for each $i=1,2$, where $e_i=(A_i,A_i)$ and $e=(A_1\oplus A_2,A_1\oplus A_2)$.

\begin{prop}
    Let $R$ be a $\Gamma$-graded ring such that, for each $e_1,e_2\in\Gamma_0$, there exist $e\in\Gamma_0$ and right gr-invertible elements belonging to $\h(1_{e_1}R1_e)$ and $\h(1_{e_2}R1_e)$. Then the following assertions hold: 
    \begin{enumerate}[\rm (1)]
        \item If $J$ and $J'$ are graded ideals of $R$ such that $1_eJ1_e=1_eJ'1_e$ for every $e\in\Gamma_0$, then $J=J'$.
        \item If $J$ is a graded ideal of $R$ such that $1_eJ1_e = \radgr(1_eR1_e)$ for every $e \in \Gamma_0$, then $J = \radgr(R)$.
    \end{enumerate}
\end{prop}

\begin{proof}
  (1)  Let $J$ and $J'$ be graded ideals such that $1_eJ1_e=1_eJ'1_e$ for each $e\in\Gamma_0$.
    Let $e_1,e_2\in\Gamma_0$.
    Take $e\in\Gamma_0$, $p_1\in\h(1_{e_1}R1_e)$, $p_2\in\h(1_{e_2}R1_e)$, $q_1\in\h(1_eR1_{e_1})$ and $q_2\in\h(1_eR1_{e_2})$ such that $p_1q_1=1_{e_1}$ and $p_2q_2=1_{e_2}$. 
    Hence,
    \[1_{e_1}J1_{e_2}
    =p_1q_1Jp_2q_2
    \subseteq p_11_eRJR1_eq_2
    = p_11_eJ1_eq_2
    = p_11_eJ'1_eq_2
    \subseteq 1_{e_1}J'1_{e_2}.
    \]
    Interchanging $J$ and $J'$, we also get $1_{e_1}J'1_{e_2}\subseteq 1_{e_1}J1_{e_2}$.
    Therefore, $1_{e_1}J1_{e_2}=1_{e_1}J'1_{e_2}$ for each $e_1,e_2\in\Gamma_0$, and it follows that $J=J'$.

    (2) Set $J'=\radgr(R)$. By \Cref{coro: radgr(R)e=rad(Re)}(2), $1_eJ'1_e=\radgr(1_eR1_e)$ for every $e\in\Gamma_0$. The result now follows from (1).
\end{proof}

We end this section computing the graded Jacobson radical of some relevant graded rings.
We begin with the rings $\M_I(A)$ and $\UT_I(A)$ (see \Cref{exem: M_I(A)}) for a unital ring $A$.

\begin{prop}
\label{prop: rad(M_I(A))}
If $A$ is a unital ring and $I$ is a non-empty set, then  $\radgr\left( \M_I(A) \right) = \M_I \left( \rad (A) \right)$.
\end{prop}

\begin{proof}
    $(\subseteq)$: Let $i,j\in I$ and $a\in A$ be such that $aE_{ij}\in\radgr(\M_I(A))$. 
    For each $b\in A$, we have 
    $(1_A-ab)E_{ii}=1_AE_{ii}-(aE_{ij})(bE_{ji})\in\U(\M_I(A)_{(i,i)})$
    by \Cref{teo: a carac de radgr(R)}(8). It follows that $1_A-ab$ is invertible in $A$  for each $b\in A$. 
    Therefore, $a\in\rad(A)$.

    $(\supseteq)$: Take $a\in\rad(A)$ and $i,j\in I$.
    For each $b\in A$, we have that $1_AE_{ii}-(aE_{ij})(bE_{ji})=(1_A-ab)E_{ii}$ is invertible in $\M_I(A)_{(i,i)}$, since $1_A-ab$ is invertible in $A$.
    By \Cref{teo: a carac de radgr(R)}(8), $aE_{ij} \in\radgr(\M_I(A))$.
\end{proof}

\begin{prop}
\label{prop: radgr UT(A)}
    Let $I$ be a partially ordered set, $A$ be a unital ring, and let $R := \UT_I(A)$. Then, the graded Jacobson radical of $R$ is given by
    \[
    \radgr(R) = \{(a_{ij})_{ij} \in \UT_I(A) : \forall i\in I,\, a_{ii} \in\rad(A)\}.
    \]
\end{prop}

\begin{proof}
    Let $a\in A$ and $i,j\in I$ with $i\leq j$.

    If $i<j$, then $aE_{ij}\in\radgr(R)$ by \Cref{teo: a carac de radgr(R)}(8), since $R_{(j,i)}=0$.

    Suppose now that $i=j$.  For each $b\in A$, we have that $1_AE_{ii}-(aE_{ii})(bE_{ii})=(1_A-ab)E_{ii}$ is invertible in $AE_{ii}=R_{(i,i)}$ if and only if $1_A-ab$ is invertible in $A$.
    Therefore, $aE_{ii}\in\radgr(R)$ if and only if $a\in\rad(A)$.
\end{proof}

\begin{coro}
\label{coro: rad UT(D)}
    If $I$ is a partially ordered set, $A$ is a unital ring with $\rad(A)=0$, and $R := \UT_I(A)$, then $\radgr(R) = \{(a_{ij})_{ij} \in \M_I(A) : a_{ij} \neq 0 \implies i < j\}$.\qed
\end{coro}

	Observe that if $\mathcal{C}$ is a small preadditive category, then, for each $X,Y\in\mathcal{C}_0$, $\radgr(R[\mathcal{C}])_{(X,Y)}=J(Y,X)$, where $J$ denotes the Jacobson radical of $\mathcal{C}$, as in \cite[p. 21]{Mit} or \cite{Kelly}.

In \cite[p. 472]{Harada}, it was proved that if $A$ is a unital ring and $\mathcal{C}$ is a full sub-additive subcategory of $\mathrm{mod}$-$A$ consisting of projective modules, then the Jacobson radical of $\mathcal{C}$ consists of the morphisms whose images are superfluous in their codomains.
We generalize this result for categories of gr-projective modules over groupoid graded rings.

\begin{defi}
    Let $R$ be a $\Gamma$-graded ring.
    We say that a $\Gamma$-graded right $R$-module $P$ is \emph{gr-projective} if, for all $\Gamma$-graded right $R$-modules $M$ and $N$, every surjective $g \in \Homgr(M, N)$ and every $h \in \Homgr(P, N)$, there exists $g' \in \Homgr(P, M)$ such that $h = g\circ g'$ 
    \cite[Proposition 35]{CLP}.
\[
\begin{tikzcd}
& P \arrow[dashed]{dl}[swap]{g'} \arrow{d}{h} & \\
M \arrow{r}[swap]{g} & N \arrow{r} & 0
\end{tikzcd}
\]
\end{defi}

\begin{prop}
\label{prop: rad proj-R}
    Let $R$ be a $\Gamma$-graded ring.
    Suppose that $\mathcal{C}$ is a small full subcategory of  $\Gamma$-${\rm gr}$-$R$ such that every object of $\mathcal{C}$ is a gr-projective $R$-module.
    Then, for each $P,Q\in \mathcal{C}_0$, we have
    \[\radgr(R_\mathcal{C})_{(Q,P)}=\{g\in\Homgr(P,Q):\im (g) \text{ is gr-superfluous in }Q\}.\]
\end{prop}

\begin{proof}
    $(\supseteq)$: We  adapt the proof of \cite[Theorem 2.4]{Mares}. 
    Let $g \in \Homgr(P,Q)$ be such that $\im (g)$ is gr-superfluous in $Q$.  
Suppose that $g' \in \Homgr(Q,P)$.
Since $\im (g g') \subseteq \im (g)$, it follows that $\im(g g')$ is gr-superfluous in $Q$.  
Then $\im (g g') + \im (\id_Q - g g') = Q$ implies $\im (\id_Q - g g') = Q$.  
Since $Q$ is gr-projective, it follows that there exists $h \in \Homgr(Q,Q)$ such that $(\id_Q - g g') h = \id_Q$.  
By \Cref{teo: a carac de radgr(R)}(7), we have $g \in \radgr(R_\mathcal{C})$.

$(\subseteq)$: We adapt the proof of \cite[Proposition 1.1]{Ware}. 
Suppose that $g\in\radgr(R_\mathcal{C})_{(Q,P)}$ and $X\subgr Q$ is such that $\im(g) + X = Q$.  
Then $\pi g: P \to Q/X$ is surjective, where $\pi: Q \to Q/X$ is the canonical projection.  
Since $Q$ is gr-projective, there exists $g' \in \Homgr(Q,P)$ such that $\pi g g' = \pi$.  
\[
\begin{tikzcd}
& Q \arrow[dashed]{dl}[swap]{g'} \arrow{d}{\pi} & \\
P \arrow{r}[swap]{\pi g} & Q/X \arrow{r} & 0
\end{tikzcd}
\]
For each $q \in Q$, we have $\pi(q) = \pi g g'(q)$, and it follows that $q - g g'(q) \in X$.  
Hence, $\im(\id_Q - g g') \subseteq X$.  
But since $g \in \radgr(R_\mathcal{C})$, it follows from \Cref{teo: a carac de radgr(R)}(8) that $\id_Q - g g'$ is a gr-isomorphism, and therefore $\im(\id_Q - g g') = Q$.  
Thus, $X = Q$, and it follows that $\im (g)$ is gr-superfluous in $Q$.
\end{proof}

By \cite[Theorem 13.1(1)]{Lam2}, if $E$ is an injective module, then the Jacobson radical of the ring $\End(E)$ is the set of endomorphisms whose kernel is essential in $E$. We generalize this result with the following dual version of \Cref{prop: rad proj-R}.

\begin{prop}
\label{prop: rad inj-R}
    Let $R$ be a $\Gamma$-graded ring.
    Suppose that $\mathcal{C}$ is a small full subcategory of  $\Gamma$-${\rm gr}$-$R$ such that every object of $\mathcal{C}$ is a gr-injective $R$-module.
    Then, for each $E,E'\in \mathcal{C}_0$, we have
    \[\radgr(R_\mathcal{C})_{(E',E)}=\{g\in\Homgr(E,E'):\ker (g) \text{ is gr-essential in }E\}.\]
\end{prop}

\begin{proof}
    $(\supseteq)$:  
    Let $g \in \Homgr(E,E')$ be such that $\ker (g)$ is gr-essential in $E$.  
Suppose that $g' \in \Homgr(E',E)$.
Since $\ker ( g) \subseteq \ker (g'g)$, it follows that $\ker(g 'g)$ is gr-essential in $E$.  
Then $\ker (g' g) \cap \ker (\id_E - g' g) = 0$ implies $\ker (\id_E - g' g) = 0$.  
Since $E$ is gr-injective, it follows that there exists $h \in \Homgr(E,E)$ such that $ h \circ (\id_E - g' g) = \id_E$.  
By the left version of \Cref{teo: a carac de radgr(R)}(7), we have $g \in \radgr(R_\mathcal{C})$.

$(\subseteq)$:  
Suppose that $g\in\radgr(R_\mathcal{C})_{(E',E)}$ and $X\subgr E$ is such that $\ker(g) \cap X = 0$.  
Then $g|_{X}: X \to E'$ is injective.
Since $E$ is gr-injective, there exists $g' \in \Homgr(E',E)$ such that $ g' g|_X$ is the inclusion $X\to E$.  
\[
\begin{tikzcd}
0 \arrow{r} & X \arrow{r}{g|_X} \arrow[hook]{d} & E' \arrow[dashed]{dl}{g'}  & \\
& E &
\end{tikzcd}
\]
For each $x \in X$, we have $ g'g(x)=x$, and it follows that $X\subseteq\ker(\id_E - g' g)$.  
But since $g \in \radgr(R_\mathcal{C})$, it follows from the left version of \Cref{teo: a carac de radgr(R)}(8) that $\id_E - g' g$ is a gr-isomorphism, and therefore $\ker(\id_E - g' g)=0$.  
Thus, $X = 0$, and it follows that $\ker (g)$ is gr-essential in $E$.
\end{proof}


\section{Gr-semilocal rings}\label{sec: gr-semilocal rings}

In this section, we present  gr-semilocal rings. It is an interesting class of rings that relates the concepts we have studied thus far.

\begin{defi}
    Let $R$ be a $\Gamma$-graded ring. We say that $R$ is \emph{gr-semilocal} if the graded ring $R/\radgr(R)$ is gr-semisimple. By \Cref{prop: rad of R/I for I contained in rad}(4) and \Cref{teo: R gr-ss <=> rad(R)=0 e R art}, this is equivalent to say that $R/\radgr(R)$ is a right (or left) $\Gamma_0$-artinian ring.
\end{defi}

Important classes of gr-semilocal rings are given in the following result.

\begin{prop}
\label{prop: R gr-art => R/rad(R) gr-ss}
    Let $R$ be a $\Gamma$-graded ring. The following statements hold:
\begin{enumerate}[\rm (1)]
    \item If $R$ is  right  $\Gamma_0$-artinian, then $R$ is a  gr-semilocal ring. 
    \item If the number of maximal graded right  ideals $\mathfrak{m}$ such that $\mathfrak{m}(e)\subsetneq R(e)$ is finite for each $e\in\Gamma_0$, then $R$ is a gr-semilocal ring.
\end{enumerate}     
\end{prop}

\begin{proof}
(1)   By \Cref{coro: R G_0 noeth/art => R/J G_0 noeth/art.}, $R/\radgr(R)$ is a  right  $\Gamma_0$-artinian  ring.

(2)  Suppose that $R(e)\neq \{0\}$, i.e. $1_e\neq 0$, for some $e\in\Gamma_0$. Note that there exists at least one maximal graded right ideal such that $\mathfrak{m}(e)\subsetneq R(e)$. Indeed, by Zorn's Lemma, the set $$\{J\colon J \textrm{ graded right ideal of } R\, \colon 1_e\notin J \}$$ contains a maximal element, which turns out to be a graded maximal right ideal.  Let  $\mathfrak{m}_1,\mathfrak{m}_2,\dotsc,\mathfrak{m}_{r_e}$ be the graded maximal right ideals such that  $\mathfrak{m}_i(e)\subsetneq R(e)$  for $i=1,\dotsc, r_e$. Then $\radgr(R)(e)=\bigcap_{i=1}^{r_e} \mathfrak{m}_i(e)$. Hence, \[\frac{R}{\radgr(R)}(e)\hookrightarrow \bigoplus_{i=1}^{r_e} \frac{R(e)}{\mathfrak{m}_i(e)}.\]    This implies that $\left(R/\radgr(R)\right) (e)$ is  a gr-artinian graded right $R$-module because 
$\left(R/\radgr(R)\right) (e)$ is an $R$-submodule of a gr-artinian graded right $R$-module. 
Thus, $\left(R/\radgr(R)\right) (e)$ is  a gr-artinian graded right $R/\radgr(R)$-module because the graded right $R$-submodules and the graded right $R/\radgr(R)$-submodules of $\left(R/\radgr(R)\right) (e)$ coincide. 

If $R(e)=\{0\}$, then clearly $\left(R/\radgr(R)\right) (e)$ is  a gr-artinian graded right $R/\radgr(R)$-module. 

Therefore $R/\radgr(R)$ is a right $\Gamma_0$-artinian ring.
\end{proof}

Now we prove a generalization of \Cref{prop: rad of R/I for I contained in rad}(3)  for gr-semilocal rings.

\begin{prop}
 Let $R$ be a gr-semilocal ring and $K$ be any graded ideal of $R$. Then $R/K$ is a semilocal ring and \[\radgr\left(\frac{R}{K}\right)=\frac{\radgr(R)+K}{K}.\]    
\end{prop}

\begin{proof}
  Set $J=\radgr(R)$ and consider the canonical projection $\pi:R\rightarrow R/K$, $a\mapsto a+K$. Denote $\overline{R}:=\pi(R)=R/K$ and $\overline{J}:=\pi(J)=(J+K)/K$. 
  By \Cref{prop: rad of R/I for I contained in rad}(1), $$\overline{J}\subseteq  \radgr\left(\overline{R}\right).$$    
  Using \Cref{prop: rad of R/I for I contained in rad}(3), we get
  \begin{equation}\label{eq: rad of quotient rings}
      \frac{\radgr\left(\overline{R}\right)}{\overline{J}}=\radgr\left(\frac{\overline{R}}{\overline{J}}\right)\cong_{gr}\radgr\left(\frac{R}{J+K}\right).
  \end{equation}
  The ring $R/J$ is gr-semisimple, hence its quotient ring $R/(J+K)$ is gr-semisimple as well. By Examples~\ref{exms: gr-soc and gr-rad of gr-semisimple}(2), $\radgr(R/(J+K))=\{0\}.$ By \eqref{eq: rad of quotient rings}, $\radgr(\overline{R})/\overline{J}=\{0\}$. Therefore $\radgr(\overline{R})=\overline{J}$. Furthermore,
  \[\frac{\overline{R}}{\radgr\left(\overline{R}\right)}=\frac{\overline{R}}{\overline{J}}\cong_{gr}\frac{R}{J+K}.\]
  We have just showed that $R/(J+K)$ is gr-semisimple, thus $\overline{R}$ is gr-semilocal.
\end{proof}

\begin{prop}
    Let $R$ be $\Gamma$-graded ring and $M$ be a $\Gamma$-graded right $R$-module. 
    If $R$ is a gr-semilocal ring, then $M/\radgr(M)$ is a gr-semisimple $R$-module and $\radgr(M)=M\cdot\radgr(R)$.
\end{prop}

\begin{proof}
    Since $R/\radgr(R)$ is a gr-semisimple ring, then every graded right $R/\radgr(R)$-module is gr-semisimple by \cite[Proposition 59]{CLP}. Thus $\frac{M}{M\cdot\radgr(R)}$ is a gr-semisimple $R/\radgr(R)$-module.
    Therefore, $\frac{M}{M\cdot\radgr(R)}$ is a gr-semisimple $R$-module by \Cref{prop: R e R/J tem os mesmos gr-simples}(1). It follows from \Cref{exms: gr-soc and gr-rad of gr-semisimple}(2) that $\radgr\left(\frac{M}{M\cdot\radgr(R)}\right)=0$.
    Thus, $\radgr(M)\subseteq M\cdot\radgr(R)$ by \Cref{lem: rad(M/rad)=0}(1).
    On the other hand,  $M\cdot\radgr(R)\subseteq \radgr(M)$ by \Cref{coro: rad e soc sao prerad}(3).
\end{proof}

For any $\Gamma$-graded ring $R$ with $J=\radgr(R)$, \Cref{teo: a carac de radgr(R)}(10) implies that $\soc_{\rm gr}(M)\subseteq\{m\in M:mJ=0\}$ for all $\Gamma$-graded $R$-modules $M$.  When the graded ring $R$ is gr-semilocal, this inclusion is an equality and we can obtain the terms in the Loewy series of graded right $R$-modules from the powers of $J$. More precisely, we have the following graded version of \cite[Proposition 4.14]{Goodearl}.

\begin{prop}
\label{prop: soc(M)=lann(J)}
    Let $R$ be a $\Gamma$-graded ring and set $J=\radgr(R)$. 
    If $R$ is  gr-semilocal, then 
     \[\soc_{\rm gr}^n(M)=\{m\in M:mJ^n=0\}\]
     for all $\Gamma$-graded right $R$-modules $M$ and $n\in\mathbb{N}$.   
\end{prop}

\begin{proof}
The proof is by induction on $n\in\mathbb{N}$. Setting $J^0=R$ and $\soc_{\rm gr}^0(M)=0$, the case $n=0$ is clear.    Suppose that the equality of the statement is true for some $n\in\mathbb{N}$. 
    By \Cref{teo: a carac de radgr(R)}(10),  $\soc_{\rm gr}\left(\frac{M}{\soc_{\rm gr}^n(M)}\right)J=0$. Thus,  
    $\soc_{\rm gr}^{n+1}(M)J\subseteq\soc_{\rm gr}^n(M)$ and $\soc_{\rm gr}^{n+1}(M)J^{n+1}\subseteq\soc_{\rm gr}^n(M)J^n=0$ by the induction hypothesis.
    It remains to show that $N:=\{m\in M:mJ^{n+1}=0\}\subseteq \soc_{\rm gr}^{n+1}(M)$.
    Again, by induction hypothesis, we have $\soc_{\rm gr}^n(M)=\{m\in M:mJ^n=0\}\subseteq N$ and $NJ\subseteq\{m\in M:mJ^n=0\}=\soc_{\rm gr}^n(M)$.
    Thus, $\left(\frac{N}{\soc_{\rm gr}^n(M)}\right)J=0$, and it follows that $N/\soc_{\rm gr}^n(M)$ is a $\Gamma$-graded right module over the gr-semisimple ring $R/J$.
    By \cite[Proposition 59]{CLP}, $N/\soc_{\rm gr}^n(M)$ is a gr-semisimple $R/J$-module. Since $R$ and $R/J$ have the same gr-simple modules by \Cref{prop: R e R/J tem os mesmos gr-simples}(1), we get that $N/\soc_{\rm gr}^n(M)$ is gr-semisimple as an $R$-module. Hence $\frac{N}{\soc_{\rm gr}^n(M)}\subseteq\soc_{\rm gr}\left(\frac{M}{\soc_{\rm gr}^n(M)}\right)$ and $N\subseteq \soc_{\rm gr}^{n+1}(M)$, as desired.
\end{proof}

The next result is just \Cref{prop: soc(M)=lann(J)} applied to the regular right $R$-module $M=R_R$. 

\begin{coro}
    Let $R$ be a $\Gamma$-graded ring. 
    If $R$ is gr-semilocal, then   $\soc_{\rm gr}^n(R_R)$ is the left annihilator of $\radgr(R)^n$ in $R$ for each $n\in\mathbb{N}$. \qed
\end{coro}

As a consequence of \Cref{prop: soc(M)=lann(J)}, we have  characterizations of (strongly) ($\Gamma_0$-)finite gr-length when $R$ is a gr-semilocal ring.

\begin{coro}
\label{coro: comp finito <=> MJn=0}
    Let $R$ be a $\Gamma$-graded gr-semilocal ring and set $J=\radgr(R)$.
    For a $\Gamma$-graded right $R$-module $M$, the following assertions are equivalent:
    \begin{enumerate}[\rm (1)]
        \item $M$ has finite gr-length.
        \item $M$ is gr-artinian and there exists $n\in\mathbb{N}$ such that $MJ^n=0$.
        \item $M$ is gr-noetherian and there exists $n\in\mathbb{N}$ such that $MJ^n=0$.
    \end{enumerate}
\end{coro}

\begin{proof}
    This follows from \Cref{prop: carac comp finito via socle} and \Cref{prop: soc(M)=lann(J)}.
\end{proof}

\begin{coro}
\label{coro: G0 comp finito <=> MJn=0}
    Let $R$ be a $\Gamma$-graded gr-semilocal ring and set $J=\radgr(R)$.
    For a $\Gamma$-graded right $R$-module $M$, the following assertions are equivalent:
    \begin{enumerate}[\rm (1)]
        \item $M$ has  $\Gamma_0$-finite gr-length.
        \item $M$ is $\Gamma_0$-artinian and, for each $e\in\Gamma_0$, there exists $n_e\in\mathbb{N}$ such that $M(e)J^{n_e}=0$.
        \item $M$ is $\Gamma_0$-noetherian and, for each $e\in\Gamma_0$, there exists $n_e\in\mathbb{N}$ such that $M(e)J^{n_e}=0$.
    \end{enumerate}
\end{coro}

\begin{proof}
    This follows from \Cref{coro: carac G0 comp finito via socle} and \Cref{prop: soc(M)=lann(J)}.
\end{proof}

\begin{coro}
\label{coro: for G0 comp finito <=> MJn=0}
    Let $R$ be a $\Gamma$-graded gr-semilocal ring and set $J=\radgr(R)$.
    For a $\Gamma$-graded right $R$-module $M$, the following assertions are equivalent:
    \begin{enumerate}[\rm (1)]
        \item $M$ has strongly $\Gamma_0$-finite gr-length.
        \item $M$ is strongly $\Gamma_0$-artinian and there exists $n\in\mathbb{N}$ such that $MJ^n=0$.
        \item $M$ is strongly $\Gamma_0$-noetherian and there exists $n\in\mathbb{N}$ such that $MJ^n=0$.
    \end{enumerate}
\end{coro}

\begin{proof}
    This follows from \Cref{prop: carac fort G0 comp finito via socle} and \Cref{prop: soc(M)=lann(J)}.
\end{proof}

Now we are able to give a characterization of gr-semilocal rings. It
 is a graded version of \cite[Proposition 15.17]{AndersonFuller}  and \cite[Exercise 20.13]{LamExercises}.

\begin{teo}
\label{teo: carac R/rad gr-ss}
    The following assertions are equivalent for a $\Gamma$-graded ring $R$:
    \begin{enumerate}[\rm (1)]
        \item $R$ is a gr-semilocal ring.
        \item For each $\Gamma$-graded right $R$-module $M$, we have 
        $$\socgr(M)=\{m\in M:m\cdot\radgr(R)=0\}.$$
        \item Every graded direct product of gr-semisimple $\Gamma$-graded right $R$-modules is gr-semisimple.
        \item  Every graded direct product of shifts of a gr-semisimple $\Gamma$-graded right $R$-module is gr-semisimple.
        \item Every graded direct product of gr-simple $\Gamma$-graded right $R$-modules is gr-semisimple.
    \end{enumerate}
\end{teo}

\begin{proof}
    $(1)\Rightarrow(2)$: Follows from \Cref{prop: soc(M)=lann(J)}.
    
    $(2)\Rightarrow(3)$: Let $M=\prod_{i\in I}^{\rm gr}M_i$ where each $M_i$ is a gr-semisimple $\Gamma$-graded right $R$-module.
    By \Cref{teo: a carac de radgr(R)}(10), $M_i\cdot\radgr(R)=0$ for each $i\in I$.
    Therefore, $M\cdot\radgr(R)=0$, and it follows from (2) that $M=\socgr(M)$ is gr-semisimple.

    $(3)\Rightarrow(4),(5)$: This is clear because gr-simple modules and shifts of gr-semisimple modules are gr-semisimple modules.

    $(4)\Rightarrow(1)$: Denote $\overline{R}:=R/\radgr(R)$.
    By \Cref{prop: rad of R/I for I contained in rad}(4) and \Cref{coro: rad = 0 <=> exist faithful semisimple}, there exists a faithful gr-semisimple $\Gamma$-graded right $\overline{R}$-module $M$.
    We define a gr-homomorphism $\varphi:\overline{R}_{\overline{R}}\to \prod^{gr}_{m\in\h(M)}M(\deg m)$ by $\varphi(\overline{a})=(m\overline{a})_{m\in\h(M)}$ for each $\overline{a}\in\overline{R}$.
    Since $M$ is faithful, it follows that $\varphi$ is injective, and therefore $\overline{R}_{\overline{R}}$ is gr-isomorphic to a graded submodule of $\prod^{gr}_{m\in\h(M)}M(\deg m)$, which is gr-semisimple by (4).
    Hence, $\overline{R}_{\overline{R}}$  is gr-semisimple because graded submodules of gr-semisimple modules are gr-semisimple by \cite[Corollary 54]{CLP}.

    $(5)\Rightarrow(1)$: By  \Cref{coro: M/rad contido num prod de gr-simp}, the graded right $R$-module $R/\radgr(R)$ embeds in a graded direct product of gr-simple right $R$-modules.  Thus,  by (5), $R/\radgr(R)$ is a graded submodule of a gr-semisimple right $R$-module. Then \cite[Corollary 54]{CLP} asserts that $R/\radgr(R)$ is a gr-semisimple right $R$-module.   Since $R$ and $R/\radgr(R)$ have the same gr-simple modules by \Cref{prop: R e R/J tem os mesmos gr-simples}(1),  we get that $R/\radgr(R)$ is a gr-semisimple ring.
\end{proof}

Our next aim is to give other characterizations of gr-semilocal rings $R$ in terms of the subrings $1_eR1_e$ and $R_e$ for $e\in\Gamma_0$. This allows us to connect our concept of gr-semilocal rings with that of semilocal category given in \cite{Facchini}. Moreover, this will provide many examples of gr-semilocal rings graded by groupoids of the form $I\times I$. 
Keeping our main objective in mind, we recall the following definition from \cite[Section~5.4]{Artigoarxiv}.

\begin{defi}
    Let $R$ be a $\Gamma$-graded ring and $e\in\Gamma_0$.
    We say that $R$ is \emph{right $e$-faithful} if, for each $\gamma\in e\Gamma$ and $0\neq a\in R_\gamma$, there exists $r\in R_{\gamma^{-1}}$ such that $0\neq ar\in R_e$.
\end{defi}

\begin{lema}
\label{prop: rad R=0 --> R e-fiel}
    Let $R$ be a  $\Gamma$-graded ring.
	If $\radgr(R)=0$, then $R$ is right $e$-faithful for every $e\in\Gamma_0$.
\end{lema}

\begin{proof}
    Let $e\in\Gamma_0$, $\gamma\in e \Gamma$, and $0\neq a\in R_\gamma$.
    Since $a\notin\radgr(R)$, it follows from \Cref{teo: a carac de radgr(R)}(8) that there exists $x\in R_{\gamma^{-1}}$ such that $1_e - ax$ is not invertible in $R_e$.
    In particular, $0\neq ax \in R_e$.
	Therefore, $R$ is right $e$-faithful.
\end{proof}

\begin{prop}
\label{prop: R gr-semilocal <--> Re semilocal}
	Let $R$ be a $\Gamma$-graded ring. 
    The following assertions are equivalent:
    \begin{enumerate}[\rm (1)]
        \item $R$ is a gr-semilocal ring.
        \item $1_eR1_e$ is a gr-semilocal ring for every $e\in\Gamma_0$. 
        \item $R_e$ is a semilocal ring for every $e\in\Gamma_0$. 
    \end{enumerate}
\end{prop}

\begin{proof}
    $(1)\Rightarrow(2)$: By \cite[Proposition 5.32(2)]{Artigoarxiv}, if $T$ is a $\Gamma$-graded gr-semisimple ring, then $1_eT1_e$ is gr-semisimple for all $e\in\Gamma_0$.  Now,  
    \[\frac{1_eR1_e}{\radgr(1_eR1_e)}=\frac{1_eR1_e}    {1_e(\radgr(R))1_e}\isogr 1_e\left(\frac{R}{\radgr(R)}\right)1_e\,\]
holds for all $e\in \Gamma_0$ by \Cref{coro: radgr(R)e=rad(Re)}(2). 
    
    Thus, since $R/\radgr(R)$ is a gr-semisimple ring, the graded ring $\frac{1_eR1_e}{\radgr(1_eR1_e)}$ 
    is gr-semisimple for all $e\in\Gamma_0$. That is, $1_eR1_e$ is gr-semilocal for all $e\in\Gamma_0$.

    $(2)\Rightarrow(3)$:
 By \cite[Proposition 5.32(3)]{Artigoarxiv}, if $T$ is a $\Gamma$-graded gr-semisimple ring, then $T_e$ is semisimple for all $e\in\Gamma_0$.
     Let $e\in\Gamma_0$. 
    Using \Cref{coro: radgr(R)e=rad(Re)}(1), we have
	\[\frac{R_e}{\rad(R_e)}=\frac{R_e}{\radgr(1_eR1_e)_e}\cong\left(\frac{1_eR1_e}{\radgr(1_eR1_e)}\right)_e\,.\]
Thus, since $\frac{1_eR1_e}{\radgr(1_eR1_e)}$ is gr-semisimple, the ring $\frac{R_e}{\rad(R_e)}$	is  semisimple.  Therefore, $R_e$ is a semilocal ring.

    $(3)\Rightarrow(1)$: By  \cite[Proposition 5.34(1)]{Artigoarxiv}, if $T$ is a $\Gamma$-graded ring such that $T_e$ is semisimple and $T$ is $e$-faithful for all $e\in\Gamma_0$, then $T$ is a gr-semisimple ring.    
    
    Again by \Cref{coro: radgr(R)e=rad(Re)}(1), we have
	\[\frac{R_e}{\rad(R_e)}=\frac{R_e}{\radgr(R)_e}\cong\left(\frac{R}{\radgr(R)}\right)_e\,.\]
    Hence, if (3) holds, then $\left(\frac{R}{\radgr(R)}\right)_e$ is semisimple for all $e\in \Gamma_0$. Moreover, the fact that $\radgr\left(\frac{R}{\radgr(R)}\right)=\{0\}$ by \Cref{prop: rad of R/I for I contained in rad}(4) implies that $\frac{R}{\radgr(R)}$ is right $e$-faithful for all $e\in\Gamma_0$ by \Cref{prop: rad R=0 --> R e-fiel}. Therefore,  $R/\radgr(R)$ is a gr-semisimple ring, as desired.    
    \end{proof}

\begin{defi}
    Following \cite[Definition~4.61]{Facchini}, a \emph{semilocal category} is a preadditive category with a nonzero object such that the endomorphism ring of every nonzero object is a semilocal ring.
\end{defi}

\begin{coro}
	Let $\mathcal{C}$ be small  preadditive category with a nonzero object. 
    The category $\mathcal{C}$ is semilocal if and only if the $\mathcal{C}_0\times \mathcal{C}_0$-graded ring $R_\mathcal{C}$ is gr-semilocal.
\end{coro}

\begin{proof}
By 	\Cref{prop: R gr-semilocal <--> Re semilocal}, $R_\mathcal{C}$ is gr-semilocal if and only if $(R_\mathcal{C})_{(X,X)}=\mathcal{C}(X,X)$ is a semilocal ring for all  $X\in\mathcal{C}_0$.
\end{proof}

Examples of (small) semilocal categories can be obtained from \Cref{prop: proj-A art e noet} and \Cref{prop: mod-A art e noet}, since right (or left) $\Gamma_0$-artinian rings are gr-semilocal by \Cref{prop: R gr-art => R/rad(R) gr-ss}(1).
In the following result, we present some additional ones.
Further examples can be found using the results of \cite{HerberaShamsuddin1995}.

\begin{prop}
    Let $A$ be a unital ring and $\mathcal{C}$ be a small full subcategory of the category of right $A$-modules.
    If one of the following assertions hold, then the $\mathcal{C}_0\times\mathcal{C}_0$-graded ring $R_\mathcal{C}$ is gr-semilocal.
    \begin{enumerate}[\rm (1)]
        \item Every object of $\mathcal{C}$ is an artinian $A$-module.
        \item $A$ is a finite dimensional algebra over a field $k$ and every object of $\mathcal{C}$ is a finitely generated $A$-module.
        \item $A$ is a finite dimensional algebra over a field $k$ and $\mathcal{C}$ is a skeleton of $\fgmod$-$A$.
        \item $A$ is a commutative semilocal ring and every object of $\mathcal{C}$ is a finitely generated $A$-module.
        \item $A$ is a semilocal ring and every object of $\mathcal{C}$ is a finitely presented $A$-module.
        \item $A$ is a semilocal ring and every object of $\mathcal{C}$ is a finitely generated projective $A$-module.
    \end{enumerate}
\end{prop}

\begin{proof}
    By \Cref{prop: R gr-semilocal <--> Re semilocal}, it suffices to show that, in each case, $\mathcal{C}(M,M) = (R_\mathcal{C})_{(M,M)}$ is a semilocal ring for every $M \in \mathcal{C}_0$.

    (1) If $M$ is an artinian $A$-module, then $\End_A(M)$ is a semilocal ring by \cite[Corollary~6]{Camps_Dicks} or \cite[Theorem 3.19]{Facchini}.

    (2) Follows from (1), since every finite dimensional algebra over a field is an artinian ring and every finitely generated module over an artinian ring is artinian. 
    
    (3) Immediate from (2). For another approach, see \cite[Corollary II.1.12]{Coelho}.

    (4) If $M$ is a finitely generated module over a commutative semilocal ring, then $\End(M)$ is a semilocal ring by \cite[Theorem 4.58]{Facchini}.

    (5) If $M$ is a finitely presented module over a semilocal ring, then $\End_A(M)$ is a semilocal ring by \cite[Proposition~3.1]{Facchini_Herbera} or \cite[Theorem 5.48]{Facchini}.

    (6) Follows from (5), since every finitely generated projective module is finitely presented by \cite[Exercise 4.1(a)]{RespostasLam2}.
\end{proof}

All  properties and results about gr-semilocal rings presented in this section are natural generalizations of well-known properties and results of semilocal rings. We conclude the section with a property that holds for semilocal rings but does not, in general, generalize in the natural way to gr-semilocal rings.

Recall that a ring $R$ is Dedekind finite if $ab=1 \Rightarrow ba=1$ for $a,b\in R$. Semilocal rings are known to be Dedekind finite \cite[Proposition~(20.8)]{Lam1}. If $R$ is a $\Gamma$-graded ring, we could say that $R$ is gr-Dedekind finite if for all $\gamma\in\Gamma$ and homogeneous elements $a\in R_\gamma$, $b\in R_{\gamma^{-1}}$ the implication $ab=1_{r(\gamma)}\Rightarrow ba=1_{d(\gamma)}$ holds. 
We next present an example of a groupoid graded gr-semilocal ring that is not gr-Dedekind finite according to this definition.

\begin{exem}
    Let $\Gamma$ be the groupoid $\{1,2\}\times\{1,2\}$. Let $D$ be a division ring.  Let $R=\M_3(R)$ endowed with the $\Gamma$-grading
     \[R_{(1,1)}=\begin{bmatrix}
    D & D & 0 \\
    D & D & 0 \\
    0 & 0 & 0 \\
    \end{bmatrix},\quad  R_{(1,2)}=\begin{bmatrix}
    0 & 0 & D \\
    0 & 0 & D \\
    0 & 0 & 0 \\
    \end{bmatrix}, \]
\[    R_{(2,1)}=\begin{bmatrix}
    0 & 0 & 0 \\
    0 & 0 & 0 \\
    D & D & 0 \\
    \end{bmatrix},\quad  R_{(2,2)}=\begin{bmatrix}
    0 & 0 & 0 \\
    0 & 0 & 0 \\
    0 & 0 & D \\
    \end{bmatrix}.\]    
The identity elements of $R$ are
\[ 
\mathbb{I}_{(1,1)}=\begin{bmatrix}
    1 & 0 & 0 \\
    0 & 1 & 0 \\
    0 & 0 & 0 \\
    \end{bmatrix}=E_{11}+E_{22},\quad \mathbb{I}_{(2,2)}=\begin{bmatrix}
    0 & 0 & 0 \\
    0 & 0 & 0 \\
    0 & 0 & 1 \\
    \end{bmatrix}=E_{33}.
\]
The graded ring $R$ is gr-semisimple as shown in \cite[Example~7.2(2)]{Artigoarxiv}. Thus, $R$ is a gr-semilocal ring.

Consider the elements $A=\left(\begin{smallmatrix}
    0 & 0 & 0 \\
    0 & 0 & 0 \\
    1 & 0 & 0 \\
    \end{smallmatrix}\right)\in R_{(2,1)}$ and $B=\left(\begin{smallmatrix}
    0 & 0 & 1 \\
    0 & 0 & 0 \\
    0 & 0 & 0 \\
    \end{smallmatrix}\right)\in R_{(1,2)}$. Then $AB=\left(\begin{smallmatrix}
    0 & 0 & 0 \\
    0 & 0 & 0 \\
    0 & 0 & 1 \\
    \end{smallmatrix}\right)=\mathbb{I}_{(2,2)}$, but $BA=\left(\begin{smallmatrix}
    1 & 0 & 0 \\
    0 & 0 & 0 \\
    0 & 0 & 0 \\
    \end{smallmatrix}\right)\neq \mathbb{I}_{(1,1)}$. Hence $R$ is not a gr-Dedekind-finite ring in the sense explained above. 
\end{exem}

 Item (2) of the following result provides an alternative (weaker) definition of gr-Dedekind finite for groupoid graded rings and shows that if $\Gamma$ is a group, then every gr-semilocal ring is gr-Dedekind finite according to this definition.

\begin{lema}
    Let $R$ be a $\Gamma$-graded ring. The following assertions hold.
    \begin{enumerate}[\rm (1)]
        \item Suppose that $e\in\Gamma_0$ is such that $R(e)$ is a gr-artinian or gr-noetherian $R$-module.
        Then, for all $\gamma\in e\Gamma e$, $a\in R_\gamma$, and $b\in R_{\gamma^{-1}}$, the implication $ab=1_e\Rightarrow ba=1_e$ holds.
        \item If $R$ is a gr-semilocal ring, then for all $\gamma\in\Gamma$ with $d(\gamma)=r(\gamma)$ and homogeneous elements $a\in R_\gamma$, $b\in R_{\gamma^{-1}}$, the implication $ab=1_{r(\gamma)}\Rightarrow ba=1_{d(\gamma)}$ holds.
    \end{enumerate}
\end{lema}

\begin{proof}
    (1) Let $\gamma\in e\Gamma e$, $a\in R_\gamma$, and $b\in R_{\gamma^{-1}}$ such that $ab=1_e$.
    
    Suppose that $R(e)$ is gr-noetherian and consider the map $g\in\END(R_R)_\gamma$ given by left multiplication by $a$. 
    Since $ab=1_e$, we have $\im g = R(e)$.
    Then $g$ is gr-invertible in $\END(R_R)$ by \Cref{lem: Fitting}(1), and it follows that $a$ is gr-invertible in $R$ because $R\cong_{gr} \END(R_R)$ by \cite[Lemma 3.4]{Artigoarxiv}. 
    Thus, $ba=a^{-1}aba=a^{-1}1_ea=1_e$.
    
    Now suppose that $R(e)$ is gr-artinian and set $g'\in\END(R_R)_{\gamma^{-1}}$ given by left multiplication by $b$. 
    Since $ab=1_e$, we get $R(e)\cap \ker g' = \{0\}$.
    Then $g'$ is gr-invertible in $\END(R_R)$ by \Cref{lem: Fitting}(2), and $b$ is gr-invertible in $R$ because $R\cong_{gr} \END(R_R)$ by \cite[Lemma 3.4]{Artigoarxiv}. 
    Thus, $ba=babb^{-1}=b1_eb^{-1}=1_e$.

    (2) Suppose that $R$ is a gr-semilocal ring.
    Then $R/\radgr(R)$ is a right $\Gamma_0$-artinian ring.
    It follows from (1) that for all $\gamma\in\Gamma$ with $d(\gamma)=r(\gamma)$ and homogeneous elements $a\in R_\gamma$, $b\in R_{\gamma^{-1}}$ the implication $\overline{a}\overline{b}=\overline{1_{r(\gamma)}}\Rightarrow \overline{b}\overline{a}=\overline{1_{d(\gamma)}}$ holds in $R/\radgr(R)$.
    The result now follows from \Cref{prop: R e R/J tem os mesmos gr-simples}(3).
\end{proof}

\end{document}